\documentclass{article}

\usepackage[top=1in, bottom=1in, left=1in, right=1in]{geometry}
\usepackage{parskip}
\usepackage{amsmath, amssymb, amsfonts, amsthm, bm, mathrsfs}

\usepackage{graphicx}
\usepackage[percent]{overpic}
\usepackage{subfigure}
\usepackage{tikz}
\usepackage{calc}
\usetikzlibrary{shapes.geometric, arrows, positioning, decorations.pathreplacing}

\tikzset{ 
    tree_edge/.style={
        thick, -
    },
    edge from parent/.style={
        draw, thick, -
    } 
}

\tikzstyle{level 1}=[level distance=12.5mm,sibling distance=3.35cm] 
\tikzstyle{level 2}=[level distance=12.5mm,sibling distance=2cm]
\tikzstyle{level 3}=[level distance=12.5mm,sibling distance=1cm]

\usepackage{algorithm}
\usepackage{algpseudocode}

\usepackage{url}
\usepackage[colorlinks=true,
            linkcolor=blue,
            citecolor=black,
            urlcolor=blue]{hyperref}
\usepackage[capitalize,nameinlink]{cleveref}

\usepackage{appendix}
\usepackage{placeins}
\usepackage[normalem]{ulem}

\theoremstyle{plain}
\newtheorem{theorem}{Theorem}[section]
\newtheorem{lemma}[theorem]{Lemma}
\newtheorem{proposition}[theorem]{Proposition}

\theoremstyle{definition}
\newtheorem{definition}[theorem]{Definition}
\newtheorem{example}[theorem]{Example}

\theoremstyle{remark}
\newtheorem{remark}[theorem]{Remark}

\numberwithin{equation}{section}
\numberwithin{algorithm}{section}

\DeclareGraphicsExtensions{.eps,.pdf,.png,.jpg}

\AddToHook{env/lemma/begin}{\crefalias{theorem}{lemma}}
\AddToHook{env/proposition/begin}{\crefalias{theorem}{proposition}}
\AddToHook{env/corollary/begin}{\crefalias{theorem}{corollary}}
\AddToHook{env/definition/begin}{\crefalias{theorem}{definition}}
\AddToHook{env/example/begin}{\crefalias{theorem}{example}}
\AddToHook{env/remark/begin}{\crefalias{theorem}{remark}}

\AddToHook{cmd/bibliography/before}{\FloatBarrier}

\makeatletter
\providecommand*{\theHALG@line}{}
\renewcommand*{\theHALG@line}{\theHalgorithm.\arabic{ALG@line}}
\makeatother

\title{A Semi-Lagrangian Adaptive Rank (SLAR) Method for High-Dimensional Vlasov Dynamics\thanks{The authors acknowledge support from AFOSR grants FA9550-24-1-0254 and DOE grant DE-SC0023164. AJC is also supported by ONR grant N0001424-1-2242.}}

\author{Nanyi Zheng\thanks{Department of Mathematical Sciences, University of Delaware, Newark, DE
  (\href{mailto:nyzheng@udel.edu}{nyzheng@udel.edu}).}
\and William A. Sands\thanks{Department of Mathematical Sciences, University of Delaware, Newark, DE
  (\href{mailto:wsands@udel.edu}{wsands@udel.edu}). Corresponding author.}
\and Daniel Hayes\thanks{Department of Mathematical Sciences, University of Delaware, Newark, DE
  (\href{mailto:dphayes@udel.edu}{dphayes@udel.edu}).}
\and Andrew J. Christlieb\thanks{Department of Computational Mathematics, Science and Engineering, 
  Michigan State University, East Lansing, MI 48824
  (\href{mailto:christli@msu.edu}{christli@msu.edu}).}
\and Jing-Mei Qiu\thanks{Department of Mathematical Sciences, University of Delaware, Newark, DE
  (\href{mailto:jingqiu@udel.edu}{jingqiu@udel.edu}).}}

\date{}

\hypersetup{pdftitle={A Semi-Lagrangian Adaptive Rank (SLAR) Method for High-Dimensional Vlasov Dynamics}, pdfauthor={}}

\begin{document}

\maketitle

\bigskip
\noindent
{\bf Abstract.}
We extend our previous work on a semi-Lagrangian adaptive rank (SLAR) integrator, in the finite difference framework for nonlinear Vlasov-Poisson systems, to the general high-order tensor setting. The proposed scheme retains the high-order accuracy of semi-Lagrangian methods, ensuring stability for large time steps and avoiding dimensional splitting errors.
The primary contribution of this paper is the novel extension of the algorithm from the matrix to the high-dimensional tensor setting, which enables the simulation of Vlasov models in up to six dimensions. The key technical components include (1) a third-order high-dimensional polynomial reconstruction that scales as $\mathcal{O}(d^2)$, providing a pointwise approximation of the solution at the foot of characteristics in a semi-Lagrangian scheme; (2) a recursive hierarchical adaptive cross approximation of high-order tensors in a hierarchical Tucker format, characterized by a tensor tree; (3) a low-complexity Poisson solver in the hierarchical Tucker format that leverages the FFT for efficiency. The computed adaptive rank kinetic solutions exhibit low-rank structures within branches of the tensor tree resulting in substantial computational savings in both storage and time. The resulting algorithm achieves a computational complexity of $\mathcal{O}(d^4 N r^{3+\lceil\log_2d\rceil})$, where $N$ is the number of grid points per dimension, $d$ is the problem dimension, and $r$ is the maximum rank in the tensor tree, overcoming the curse of dimensionality. Through extensive numerical tests, we demonstrate the efficiency of the proposed algorithm and highlight its ability to capture complex solution structures while maintaining a computational complexity that scales linearly with $N$.



\section{Introduction}
\label{sec:Intro}

In this work, we develop algorithms for the nonlinear Vlasov-Poisson (VP) system, which is widely used to study various plasma phenomena including instabilities, wave-particle interactions, and charged particle dynamics in fusion devices and accelerators, as well as large-scale astrophysical events. In the present setting, we are primarily interested in the time evolution of an electron probability distribution function $f(\bm{x}, \bm{v}, t)$ whose evaluation represents the probability of finding an electron at position $\bm{x} \in \mathbb{R}^{d_x}$, with velocity $\bm{v} \in \mathbb{R}^{d_v}$, at time $t$. We shall denote the phase space over which the distribution evolves by $\Omega_{\bm{x}} \times \Omega_{\bm{v}} \subset \mathbb{R}^{d_x + d_v}$. In practical applications of plasmas, this phase space is high-dimensional and generally requires either four or six dimensions, corresponding to 2D2V or 3D3V problems, respectively. The exponential storage and computational cost associated with such high-dimensional problems ---commonly known as the \emph{curse of dimensionality}--- poses a significant challenge in VP simulations. The primary objective of this work is to construct a structured numerical framework that remains \emph{scalable} in high dimensions, meaning that the algorithmic building blocks can be systematically extended to arbitrary values of $d_x$ and $d_v$, while controlling both memory usage and preserving high-order accuracy.

To present the model, we shall assume that the heavier ions remain stationary and provide a uniform neutralizing background charge density $\rho_{0}$. After non-dimensionalizing time by the inverse electron plasma frequency $\omega_{pe}^{-1}$ and space by the plasma Debye length $\lambda_{D}$, the dimensionless VP system is given by
\begin{align}
    &\frac{\partial f}{\partial t} + \bm{v} \cdot \nabla_{\bm{x}} f - \bm{E}(\bm{x}, t) \cdot \nabla_{\bm{v}} f = 0, \quad \Omega_{\bm{x}} \times \Omega_{\bm{v}} \times (0,T], 
    \label{eq:Vlasov} \\
    &-\Delta_{\bm{x}} \phi = \rho(\bm{x}, t), \quad
    \rho(\bm{x}, t) = \rho_0 - \int_{\mathbb{R}^{d_v}} f(\bm{x}, \bm{v}, t) \,\mathrm{d}\bm{v},
    \quad \Omega_{\bm{x}} \times (0,T], 
    \label{eq:Poisson} \\
    &\bm{E}(\bm{x},t) = -\nabla_{\bm{x}} \phi, 
    \quad \Omega_{\bm{x}} \times (0,T]. 
    \label{eq:Efromphi}
\end{align}
The evolution of $f(\bm{x},\bm{v},t)$ by \cref{eq:Vlasov} induces a corresponding charge density $\rho(\bm{x},t)$, which generates a self-consistent electrostatic potential $\phi(\bm{x},t)$ via \cref{eq:Poisson}. Then, from \cref{eq:Efromphi}, we obtain the electric field $\bm{E}(\bm{x},t)$, which accelerates electrons to restore quasi-neutrality.

A wide range of numerical methods have been developed for the VP system, and, more generally, multi-scale kinetic models. Of particular relevance to this work are the semi-Lagrangian (SL) schemes \cite{sonnendrucker1999semi, qiu2010conservative, rossmanith2011positivity, qiu2011conservative, DimarcoReview2014, cho2024conservative} and the family of low-rank and adaptive-rank approaches \cite{einkemmer2025review}. The SL methodology combines an Eulerian mesh with characteristic tracing, which results in an integration scheme that permits large time step sizes and overcomes the CFL restriction. At each time step, solution values on grid points are updated by interpolating around the feet of the characteristics. For linear advection, this yields a straightforward and effective algorithm, and when diffusion and source terms are present, these contributions can be integrated along characteristics \cite{li2023high, ding2020semi}. 

Low-rank and adaptive-rank techniques have recently emerged as powerful tools to reduce computational complexity and mitigate the curse of dimensionality in high-dimensional kinetic problems.
These methods exploit separability arising from tensor-product structure between, and sometimes within, the spatial and velocity dimensions. Whether to further exploit low-rank structure within the spatial variables depends on the underlying geometry and the intrinsic low-rankness of the problem. Some notable examples include low-rank SL schemes \cite{kormann2015semi}, dynamical low rank approaches \cite{einkemmer2018low, einkemmer2021mass, dektor2025interpolatory}, Eulerian step-and-truncate formulations \cite{guo2022low,sands2025transport}, and sampling-based methods that select the most ``informative" rows and columns of matrices \cite{zheng2025semi, sands2026adaptive}. These sampling strategies have been applied to the Vlasov–Poisson system with local mass conservation \cite{zheng2025semi}, and, more recently, to the multi-scale BGK model \cite{sands2026adaptive}. In particular, \cite{sands2026adaptive} proposed an adaptive-rank SL method that preserves the local conservation laws for mass, momentum, and energy using an implicit, dynamic closure approach that extends the locally macroscopic conservative (LoMaC) projection technique in \cite{guo2024local}.

The central component of this paper is the integration of the hierarchical Tucker decomposition (HTD), originally developed for tensor compression \cite{hackbusch2009new}, into the SL finite-difference (FD) framework for the nonlinear evolution of the VP system. Existing algorithms for constructing HTDs include both classical approaches that require global access to tensor entries \cite{grasedyck2010hierarchical,kressner2012htucker} and heuristic, data-centric strategies such as the black box approximations based on cross approximations \cite{espig2009black,ballani2013black}. In contrast to the black box approximations that typically involve pseudoinverses or least-squares solves, the present work introduces a hierarchical Tucker adaptive cross approximation (HTACA), a recursive algorithm that extends the adaptive cross approximation (ACA) from matrices to tensors. HTACA builds HTDs from localized entry access 
{with successive rank-1 corrections,} and provides the methodological foundation for the scalable framework developed for a general high-dimensional VP system.  

The contributions of this paper are threefold. First, we propose a SL-FD method based on a unified high-dimensional quadratic reconstruction at the foot of characteristics. This reconstruction achieves third-order accuracy with a compact $\mathcal{O}(d^{2})$ stencil and leads to a dimension-agnostic SL-FD scheme that avoids dimensional splitting. Although the method does not enforce conservation, it offers a simple and systematic formulation that differs from existing SL approaches. Second, we introduce HTACA as a new tensor algorithm that generalizes ACA to high-order tensors and constructs HTDs from local entry access. This algorithm is memory-efficient and scalable, and serves as the core building block of our adaptive-rank framework. Third, we develop a low-complexity Poisson solver in the HTD format, which leverages HTACA to represent the solution in low-rank form and employs efficient batched FFT/IFFT transforms. A notable advantage of this solver is that it circumvents Hadamard products of HT tensors, which typically increase ranks \cite{kressner2012htucker}. Combining these components, the resulting SLAR framework admits a worst-case complexity of $\mathcal{O}\!\bigl(d^{4} N r^{3+\lceil \log_{2} d \rceil}\bigr)$, where $N$ is the grid resolution per dimension, $d$ is the problem dimension, and $r$ is the maximum tree rank.

The remainder of this paper is organized as follows. 
In \cref{sec:SL-FD-high-D}, we present a local third-order SL-FD scheme for high-dimensional linear advection problems. 
\Cref{sec:HTACA} describes the HTACA algorithm constructed via a recursive adaptive cross approximation (ACA) on a given dimension tree. In \cref{sec:SLAR}, we detail the proposed high-dimensional SLAR method for the nonlinear VP system, including a nonlinear characteristic tracing, a Poisson solver in HTD form, and an analysis of the overall algorithmic complexity. The effectiveness of the method is demonstrated on several benchmark problems in \cref{sec:numerical_tests}. Finally, \cref{sec:Conclusion} summarizes the findings and outlines potential directions for future work. 

\section{SL-FD scheme for linear advection equations}
\label{sec:SL-FD-high-D}

In this section, we extend the SL-FD solver previously developed for 2D problems in \cite{zheng2025semi} to problems posed in arbitrary dimensions using characteristic tracing together with high-order polynomial reconstruction. The proposed solver is designed for general high-dimensional problems and relies on unified reconstruction formulas that can be applied to arbitrary $d$-dimensional Cartesian grids without requiring dimension-specific adjustments. This structured formulation enables systematic implementation of the SL-FD scheme in high-dimensional problems and facilitates its integration into the adaptive-rank framework described in \cref{sec:SLAR}.

\subsection{An SL method for the linear advection equation}

We consider the linear advection equation in $d$ spatial dimensions:
\begin{equation}\label{eq:advection_ddim}
    \frac{\mathrm{d}f(\bm{x},t)}{\mathrm{d}t} + \bm{a}(\bm{x},t)\cdot\nabla_{\bm{x}}f(\bm{x},t) = 0,
\end{equation}
where $\bm{x} = (x^{(1)}, x^{(2)}, \dots, x^{(d)}) \in \mathbb{R}^d$ and $\bm{a}(\bm{x},t)$ is the velocity field. An important property is that the solution is constant along the characteristic curves, a fact that can be exploited to construct a high-order local solver for arbitrary $d$-dimensional problems.

The domain is discretized along each dimension with a uniform Cartesian mesh,
\begin{equation*}
    a^{(\mu)} = x^{(\mu)}_{\frac{1}{2}} < x^{(\mu)}_{\frac{3}{2}} < \ldots < x^{(\mu)}_{N_{\mu}+\frac{1}{2}} = b^{(\mu)}, 
    \quad \mu=1,2,\ldots,d.
\end{equation*}
For each $\mu$, with $\mu = 1, \dots, d$, we denote the cells, centers, and spacing as
\begin{equation*}
    I^{(\mu)}_{i_{\mu}} = \left[x^{(\mu)}_{i_{\mu}-\frac{1}{2}},x^{(\mu)}_{i_{\mu}+\frac{1}{2}}\right], 
    \quad x^{(\mu)}_{i_{\mu}} = \tfrac{1}{2}\bigl(x^{(\mu)}_{i_{\mu}-\frac{1}{2}}+x^{(\mu)}_{i_{\mu}+\frac{1}{2}}\bigr), 
    \quad \Delta x^{(\mu)} = x^{(\mu)}_{i_{\mu}+\frac{1}{2}}-x^{(\mu)}_{i_{\mu}-\frac{1}{2}}.
\end{equation*}
For a particular cell, we use multi-index notation: 
$
I_{i_1,i_2,\ldots,i_d} = I^1_{i_1}\times I^2_{i_2} \times \cdots\times I^{d}_{i_d}, 
\quad \bm{x}_{i_1,\ldots,i_d}=(x^{(1)}_{i_1},\ldots,x^{(d)}_{i_d}).
$
The solution to \cref{eq:advection_ddim} can be computed pointwise by the method of characteristics:
\begin{equation}
    f(\bm{x}_{i_1,\ldots,i_d},t^{n+1}) = f(\bm{x}^\star, t^n),
    \label{eq:SL_basic_formulation}
\end{equation}
where $\bm{x}^\star = (x^{(1),\star},\ldots,x^{(d),\star})$ is the characteristic foot, obtained by integrating the ODE system backward in time from $t^{n+1}$ to $t^{n}$:
\begin{equation}\label{eq:ODEs_characteristics}
    \dfrac{d\bm{x}}{dt} = \bm{a}(\bm{x}(t), t), \qquad    \bm{x}(t^{n+1}) = \bm{x}_{i_1,\ldots,i_d}.
\end{equation}
This ODE system can be solved using a high-order Runge--Kutta (RK) method. We remark that the locations of the characteristic feet generally do not coincide with the mesh. To evaluate 
$f(\bm{x}^\star,t^n)$, one can use various interpolation or reconstruction strategies. In this work, we employ a high-dimensional polynomial reconstruction technique, as described in the next subsection.

\subsection{High-dimensional local polynomial reconstruction}
\label{sec:poly_rec}

In \cref{eq:SL_basic_formulation}, we reconstruct the solution at the characteristic foot using a high-dimensional polynomial. The scheme achieves third-order spatial accuracy by constructing a local $P^2$ polynomial in $d$ variables using the cell-center values of $ f(\bm{x}, t^n) $ near the characteristic foot $ \bm{x}^\star $:
\begin{equation}\label{eq:SL-FD_solver}
    f^{n+1}_{i_1,\ldots,i_d} = p(\bm{x}^\star).
\end{equation}
Suppose the location of the characteristic foot satisfies $\bm{x}^\star \in I_{i_1^\star,i_2^\star,\ldots,i_d^\star}$. We reconstruct a $d$-dimensional quadratic polynomial of the form
\begin{equation}\label{eq:P2_ddim}
\begin{split}
    p(x^{(1)},\ldots,x^{(d)}) = a_0 &+ \sum_{\mu=1}^d a_{x^{(\mu)}} \, \xi^{(\mu)}(x^{(\mu)}) + \sum_{\mu=1}^d a_{x^{(\mu)} x^{(\mu)}} \, \bigl( \xi^{(\mu)}(x^{(\mu)}) \bigr)^2 \\
    &+ \sum_{1 \leq \mu < \nu \leq d} a_{x^{(\mu)}x^{(\nu)}} \, \xi^{(\mu)}(x^{(\mu)}) \, \xi^{(\nu)}(x^{(\nu)}),
\end{split}
\end{equation}
where the normalized local coordinate in direction $\mu$ is
\(
    \xi^{(\mu)}(x^{(\mu)}) = \frac{x^{(\mu)} - x^{(\mu)}_{i^\star_\mu}}{\Delta x^{(\mu)}}.
\)
This polynomial contains $1$ constant term, $d$ linear terms, $d$ pure quadratic terms, and $\binom{d}{2}$ mixed quadratic terms, giving $1 + 2d + \binom{d}{2}$ coefficients in total.

To determine the coefficients, we impose two sets of constraints. The first set enforces interpolation at the central cell \(\bm{x}_{i^\star_1,\ldots,i^\star_d}\) and its \(2d\) direct neighbors:
\begin{equation*}
    p(\bm{x}_{j_1,\ldots,j_d}) = f^n_{j_1,\ldots,j_d}, \quad \text{for } \bm{x}_{j_1,\ldots,j_d} \in \mathcal{S}_{\mathrm{inter}},
\end{equation*}
where the interpolation stencil is
\begin{equation*}
\mathcal{S}_{\mathrm{inter}}
= \Bigl\{ \bm{x}_{j_1,\ldots,j_d} :
\exists!\,\mu \ \text{s.t. } j_\mu=i^\star_\mu\pm 1, \ 
j_\nu=i^\star_\nu \ (\nu\neq\mu) \Bigr\}
\cup \{\bm{x}_{i^\star_1,\ldots,i^\star_d}\}.
\end{equation*}
This yields a linear system of size \(1 + 2d\), from which the constant, linear, and pure quadratic coefficients are obtained for all \( \mu = 1,\ldots,d \):
\begin{align*}
    a_0 &= f^n_{i^\star_1,\ldots,i^\star_d}, \quad a_{x^{(\mu)}} = \frac{1}{2} \left( f^n_{i^\star_1,\ldots,i^\star_\mu+1,\ldots,i^\star_d} - f^n_{i^\star_1,\ldots,i^\star_\mu-1,\ldots,i^\star_d} \right), \\
    a_{x^{(\mu)}x^{(\mu)}} &= \frac{1}{2} \left( f^n_{i^\star_1,\ldots,i^\star_\mu+1,\ldots,i^\star_d} + f^n_{i^\star_1,\ldots,i^\star_\mu-1,\ldots,i^\star_d} \right) - f^n_{i^\star_1,\ldots,i^\star_d}.
\end{align*}

With these lower-order coefficients fixed, the mixed quadratic terms \( a_{x^{(\mu)}x^{(\nu)}} \) are determined by a least squares condition over the first-indirect neighbor stencil:
\begin{equation*}
    \mathcal{S}_{\mathrm{least}} = \bigl\{ \bm{x}_{j_1,\ldots,j_d} : j_\mu = i^\star_\mu \pm 1,\ j_\nu = i^\star_\nu \pm 1,\ \mu < \nu,\ j_k = i_k^\star\ \text{for } k \ne \mu,\nu \bigr\}.
\end{equation*}
Here, we minimize
\vspace{-0.3cm}
\begin{equation*}
    \min_{\{a_{x^{(\mu)}x^{(\nu)}}\}} \sum_{\bm{x}_{\bm{j}} \in \mathcal{S}_{\mathrm{least}}} \left( p(\bm{x}_{\bm{j}}) - f^n_{\bm{j}} \right)^2,
\end{equation*}
\vspace{-0.1cm}
which admits a unique solution that provides the remaining mixed quadratic terms in concise closed form:
\begin{equation*}
    a_{x^{(\mu)}x^{(\nu)}} = \frac{1}{4} \big( f^n_{i^\star_\mu+1, i^\star_\nu+1} + f^n_{i^\star_\mu-1, i^\star_\nu-1} - f^n_{i^\star_\mu+1, i^\star_\nu-1} - f^n_{i^\star_\mu-1, i^\star_\nu+1} \big),
\end{equation*}
for all \( 1 \leq \mu < \nu \leq d \), with all omitted indices fixed to \( i^\star_k \) for \( k \ne \mu,\nu \).

The resulting reconstruction, when coupled to a third-order RK scheme for characteristic tracing, achieves third-order spatio-temporal accuracy, as stated in the following proposition.
\begin{proposition}\label{prop:P2_analysis}
If all $\{f^n_{j_1,\ldots,j_d}\}$ in the reconstruction stencil $\mathcal{S}_{\mathrm{inter}} \cup \mathcal{S}_{\mathrm{least}}$ are exact, and the time step satisfies $\Delta t = \mathcal{O}(h)$ with $h = \max_k \Delta x^{(k)}$, then the SL-FD solver \cref{eq:SL-FD_solver} with a third-order RK scheme for characteristic tracing, satisfies the one-step error estimate
\begin{equation*}
    f^{n+1}_{i_1,i_2,\ldots,i_d}-f(\bm{x}_{i_1,i_2,\ldots,i_d},t^{n+1})=\mathcal{O}(h^3).
\end{equation*}
\end{proposition}

\begin{proof}
    The detailed proof is provided in the Supplementary Material.
\end{proof}

We next discuss the computational complexity of the local solver \cref{eq:SL-FD_solver}. It incurs three main costs per grid point: characteristic tracing, stencil extraction, and local polynomial reconstruction. Among these, the latter dominates the algebraic complexity. The quadratic polynomial contains \(1 + 2d + \binom{d}{2} = \mathcal{O}(d^2)\) coefficients, computed from \(1 + 2d + 4\binom{d}{2} = \mathcal{O}(d^2)\) stencil values using centered finite difference formulas. For instance, in the 2D2V case (\(d=4\)), this involves 15 coefficients and 33 stencil points, while in the 3D3V case (\(d=6\)) it involves 28 coefficients and 73 stencil points.

Despite its locality and third-order accuracy, applying the SL-FD solver to the full phase-space grid is infeasible in high dimensions due to the exponential growth in the number of grid points. To address this, we introduce the HTACA -- a novel and effective data extraction and compression scheme in \cref{sec:HTACA}. By evaluating only selected entries using the local SL-FD solver, HTACA constructs a highly compressed tensor approximation to the underlying high order tensor. The unification of these ideas forms the foundation of the proposed high-dimensional SLAR framework discussed in \cref{sec:SLAR}.

\section{Hierarchical Tucker Adaptive Cross Approximation (HTACA)}
\label{sec:HTACA}

In this section, we introduce the HTACA, which constructs HTDs of high-dimensional tensors through a novel extension of the ACA for matrices. First, in \cref{sec:Preliminaries}, we establish the basic tensor notation system and review the fundamental concepts of the HTD. The notation system adopted in \cref{sec:Preliminaries} closely follows the conventions in \cite{grasedyck2010hierarchical,kressner2012htucker}. Then, in \cref{sec:Construction}, we present a detailed description of the proposed HTACA algorithm. Together, these form the basis for the HTACA construction.

\subsection{{Notations and} Preliminaries}\label{sec:Preliminaries}

An order $d$ complex-valued tensor $\mathcal{X} \in \mathbb{C}^{N_1\times\cdots\times N_d}$ is a $d$-dimensional array whose elements are complex numbers. The tensor $\mathcal{X}$ has $d$ modes, indexed by $\mu \in {1,\ldots,d}$, and we denote the size of each mode as $N_{\mu}$. The entries of the tensor along mode $\mu$ are associated with an index set $\mathbb{I}_\mu = \{1,\ldots,N_\mu\}$. In addition, an index set can be associated with a subset of modes. For example, if $\alpha \subseteq \{1,\ldots,d\}$, then the corresponding $\alpha$-mode index set is
$
\mathbb{I}_\alpha = \left\{1,\ldots,\prod_{\mu\in \alpha}N_{\mu}\right\}.
$
Using these concepts, we introduce two fundamental operations essential to the construction of the HTD: \emph{matricization}, which rearranges a tensor into a matrix, and \emph{vectorization}, which is the special case of matricization that produces a vector.

The process of matricization first splits the modes of the tensor into a pair of disjoint subsets. Then, to construct the corresponding matricization, the first subset is merged to form row indices while the other forms the column indices. To illustrate, suppose that we partition the modes $\{1,2,\ldots,d\}$ into two disjoint sets, namely $\alpha=\{\alpha_1,\alpha_2,\ldots,\alpha_k\} \subseteq \{1,2,\ldots,d\}$ and $\alpha^c = \{s_1,s_2,\ldots,s_{d-k}\} = \{1,2,\ldots,d\} \setminus \alpha$. Then, the $\alpha$-matricization of $\mathcal{X}$ is the matrix
$
X^{\alpha}\in \mathbb{C}^{\left(\prod_{\mu\in\alpha}N_{\mu}\right)\times \left(\prod_{\nu\in\alpha^c}N_{\nu}\right)},
$
with entries defined by
$$
X^{\alpha}(i^\alpha,i^{\alpha^c}) = \mathcal{X}(i_1,i_2,\ldots,i_d),\quad \forall (i^\alpha, i^{\alpha^c})\in \mathbb{I}_\alpha\times\mathbb{I}_{\alpha^c},
$$
where the composite indices \(i^\alpha\) and \(i^{\alpha^c}\) are computed by the following bijective maps (with the modes in $\alpha$ and $\alpha^c$ ordered as listed):
\begin{align}
i^\alpha &= 1 + \sum_{\alpha_\ell\in \alpha}(i_{\alpha_\ell}-1)J_{\alpha_\ell}, \quad J_{\alpha_\ell} = \prod_{m=1}^{\ell-1}N_{\alpha_m}, \label{eq:index_1to1_map_1} \\
i^{\alpha^c} &= 1 + \sum_{s_\ell\in \alpha^c}(i_{s_\ell}-1)L_{s_\ell}, \quad L_{s_\ell} = \prod_{m=1}^{\ell-1}N_{s_m}, \label{eq:index_1to1_map_2}
\end{align}
where, by convention, an empty product (when $\ell=1$) is taken to be 1. Vectorization of a tensor $\mathcal{X}$ is obtained from the $\alpha$-matricization in which $\alpha = \{1,\ldots,d\}$, so
$
\text{vec}(\mathcal{X}) = X^{\{1,2,\ldots,d\}}\in\mathbb{C}^{N_1N_2\cdots N_d}.
$

We distinguish two types of indexing: the bold symbol \(\bm{i}^{\alpha}=(i_{\alpha_1},\ldots,i_{\alpha_k})\) 
denotes a multi-index over the modes in \(\alpha\), while the plain symbol \(i^{\alpha}\) 
denotes its corresponding linearized index in \(X^{\alpha}(i^{\alpha},i^{\alpha^c})\), satisfying
$\mathcal{X}(\bm{i}^\alpha,\bm{i}^{\alpha^c})=X^{\alpha}(i^\alpha,i^{\alpha^c}).$
These indices are related via the bijective mappings in \eqref{eq:index_1to1_map_1}–\eqref{eq:index_1to1_map_2}, and we interchange them when the context is clear. When the tensor is a matrix or vector, the two notations are equivalent and may be used interchangeably depending on context.

The Kronecker product plays a central role in the HTD by linking hierarchical levels of matricizations and vectorizations. For matrices $A=[\bm{a}_1~\bm{a}_2~\ldots~\bm{a}_{m_2}]\in\mathbb{C}^{m_1\times m_2}$ and $B=[\bm{b}_1~\bm{b}_2~\ldots~\bm{b}_{n_2}]\in\mathbb{C}^{n_1\times n_2}$, the Kronecker product $A\otimes B\in\mathbb{C}^{m_1n_1\times m_2n_2}$ is
\begin{equation*}
    A\otimes B =
    \begin{bmatrix}
    a_{11}B & a_{12}B & \cdots & a_{1m_2}B\\
    a_{21}B & a_{22}B & \cdots & a_{2m_2}B\\
    \vdots  & \vdots  & \ddots & \vdots   \\
    a_{m_1 1}B & a_{m_1 2}B & \cdots & a_{m_1 m_2}B
    \end{bmatrix}.
\end{equation*}
Equivalently, each column of \(A \otimes B\) can be expressed as the Kronecker product of a column of \(A\) with a column of \(B\):
\(
A\otimes B =
\left[
\bm{a}_1\otimes\bm{b}_1~\bm{a}_1\otimes\bm{b}_2~\cdots~\bm{a}_{m_2}\otimes\bm{b}_{n_2}
\right].
\)
The following proposition states this relation in the special case of vectors, which will be useful for connecting Kronecker products with vector outer products in the construction of HTACA.

\begin{proposition}\label{prop:kronecker_outer_relation}
Let $\bm{u} \in \mathbb{C}^{m}$ and $\bm{v} \in \mathbb{C}^{n}$. Then the outer product $\bm{v}\bm{u}^{\mathrm{T}}$ is a matrix in $\mathbb{C}^{n \times m}$, and its vectorization satisfies
\(
\mathrm{vec}(\bm{v} \bm{u}^{\mathrm{T}}) = \bm{u} \otimes \bm{v}.
\)
\end{proposition}

The HTD constructs a hierarchical low-rank representation by recursively partitioning modes of a tensor. This is encoded in a binary structure called the \emph{dimension tree}, which defines the relationships between mode subsets. A formal definition for this can be stated as follows.

\begin{definition}
(Dimension tree). A binary tree $\mathcal{T}$, with each node represented by a subset of $\{1,2,\ldots,d\}$, is called a \textit{dimension tree} if it satisfies the following conditions: the root node is $\{1,2,\ldots,d\}$; each leaf node contains a single index; the two child nodes of every parent node are disjoint; and each parent node is the union of its two child nodes. Moreover, the set of leaf and non-leaf nodes are respectively denoted as $\mathcal{L}(\mathcal{T})$ and $\mathcal{N}(\mathcal{T}) = \mathcal{T} \setminus \mathcal{L}(\mathcal{T})$. For each non-leaf node $\alpha \in \mathcal{N}(\mathcal{T})$, we denote its left and right children by $\alpha_l$ and $\alpha_r$, respectively, and the level of a node is its distance from the root node. 
\end{definition}

To simplify notation, we assume that for each parent node, indices in the left child are smaller than those in the right. In general, a dimension tree for \(d\) modes contains exactly \(d-1\) non-leaf nodes \cite{grasedyck2010hierarchical_SIMAA}. \Cref{fig:dimension_trees} displays two dimension trees that can be used to represent 4D and 6D data, i.e., \(d=4\) and \(d=6\). The tree in \cref{fig:balanced_tree_4d} is balanced and contains two levels, while the tree in \cref{fig:unbalanced_tree_6d} is unbalanced and contains three levels.

\begin{figure}[t]
    \centering
    \subfigure[Balanced 4D tree]{
    \begin{tikzpicture}[scale=0.75]
        \node(1234) {$\left\{ 1,2,3,4 \right\}$}
            child{node(12) {$\left\{ 1, 2 \right\}$} 
                child{node(1) {$\left\{ 1 \right\}$}}
                child{node(2) {$\left\{ 2 \right\}$}}   
            }
            child{node(34) {$\left\{ 3, 4 \right\}$} 
                child{node(3) {$\left\{ 3 \right\}$}}
                child{node(4) {$\left\{ 4 \right\}$}}
            };
    \label{fig:balanced_tree_4d}
    \end{tikzpicture}
    }
    \subfigure[Unbalanced 6D tree]{
    \begin{tikzpicture}[scale=0.75]
        \node (123456) {$\{1,2,3,4,5,6\}$}
            child {node (1234) {$\{1,2,3,4\}$}
                child {node (12) {$\{1,2\}$}
                    child {node (1) {$\{1\}$}}
                    child {node (2) {$\{2\}$}}
                }
                child {node (34) {$\{3,4\}$}
                    child {node (3) {$\{3\}$}}
                    child {node (4) {$\{4\}$}}
                }
            }
            child {node (56) {$\{5,6\}$}
                child {node (5) {$\{5\}$}}
                child {node (6) {$\{6\}$}}
            };
    \label{fig:unbalanced_tree_6d}
    \end{tikzpicture}
    }
    \caption{Two candidate dimension trees for 4D (left) and 6D (right) tensors.}
    \label{fig:dimension_trees}
\end{figure}

The notion of rank also extends to tensors in the HTD and is defined as follows.
\begin{definition}
(Hierarchical rank). Given a dimension tree \(\mathcal{T}\), the \textit{hierarchical rank} of a tensor \(\mathcal{X}\) is defined as
\[
(r_{\alpha})_{\alpha\in\mathcal{T}} = (\mathrm{rank}(X^{\alpha}))_{\alpha\in\mathcal{T}}.
\]
The set of tensors with hierarchical ranks bounded by \((r_\alpha)_{\alpha\in\mathcal{T}}\) is denoted by
\[
\mathcal{H}\text{-Tucker}((r_\alpha)_{\alpha\in\mathcal{T}}) = \left\{ \mathcal{X} \in \mathbb{C}^{N_1\times \cdots \times N_d} \,\middle|\, \mathrm{rank}(X^{\alpha}) \leq r_\alpha,\ \forall \alpha \in \mathcal{T} \right\}.
\]
\end{definition}
\noindent The hierarchical rank \((r_\alpha)_{\alpha\in\mathcal{T}}\) consists of \(2d-1\) values, with the rank at the root node fixed to 1. 

A dimension tree not only defines the hierarchical grouping of modes but also guides the construction of the basis and the so-called transfer matrices (tensors) in the HTD. This is motivated by the following observation: For each non-leaf node \(\alpha \in \mathcal{N}(\mathcal{T})\) with children \(\alpha_l\) and \(\alpha_r\), the column space of \(X^{\alpha}\) lies within that of the Kronecker product \(X^{\alpha_r} \otimes X^{\alpha_l}\). A formal statement of this property is given below.

\begin{lemma}\label{lem:HTD_foundation}
(\cite{grasedyck2010hierarchical,kressner2012htucker})  
For $\alpha = \alpha_l \cup \alpha_r \in \mathcal{T}$, the column space of $X^{\alpha}$ {is a subspace} of the Kronecker product of its children's matricizations:
\[
{\mathrm{span}(X^{\alpha}) \subset \mathrm{span}(X^{\alpha_r} \otimes X^{\alpha_l}),}
\]
where $\mathrm{span}(\cdot)$ denotes the column space.
\end{lemma}

Building on \cref{lem:HTD_foundation}, let \(\mathcal{X} \in \mathcal{H}\text{-Tucker}((r_\alpha)_{\alpha\in\mathcal{T}})\) denote a tensor in HTD format. We define the frame tree, \(\alpha\)-frame, and transfer matrix (tensor) as follows.
\begin{definition}
(Frame tree, \(\alpha\)-frame, transfer matrix).    
A matrix \(U_\alpha \in \mathbb{C}^{\mathbb{I}_\alpha \times r_\alpha}\) satisfying \(\mathrm{span}(U_\alpha) = \mathrm{span}(X^{\alpha})\) is called an \(\alpha\)-frame of \(\mathcal{X}\).  
A collection \((U_\alpha)_{\alpha \in \mathcal{T}}\) of such frames is called a frame tree.  
According to \cref{lem:HTD_foundation}, for any non-leaf node \(\alpha\) with children \(\alpha_l\) and \(\alpha_r\), there exists a matrix \(B_\alpha \in \mathbb{C}^{r_{\alpha_l}r_{\alpha_r} \times r_\alpha}\) such that
\begin{equation}\label{eq:HTD_foundation}
    U_\alpha = (U_{\alpha_r} \otimes U_{\alpha_l}) B_\alpha,
\end{equation}
where $B_\alpha$ is referred to as the transfer matrix at node \(\alpha\).
\end{definition}

By unifying these concepts, we may now define the HTD of a tensor as follows.
\begin{definition}
\label{def:HTD}
(HTD).  
Given \(\mathcal{X} \in \mathcal{H}\text{-Tucker}((r_\alpha)_{\alpha \in \mathcal{T}})\), the HTD of \(\mathcal{X}\) consists of a set of transfer matrices \((B_\alpha)_{\alpha \in \mathcal{N}(\mathcal{T})}\) and leaf frames \((U_\alpha)_{\alpha \in \mathcal{L}(\mathcal{T})}\), such that the recursive relation~\eqref{eq:HTD_foundation} holds for all non-leaf nodes.  
We denote this decomposition as  
\[
\mathcal{X} = ((B_\alpha)_{\alpha\in\mathcal{N}(\mathcal{T})},\ (U_\alpha)_{\alpha\in\mathcal{L}(\mathcal{T})}).
\]
\end{definition}

As seen in \cref{def:HTD}, the HTD stores the transfer tensors at the non-leaf nodes of the dimension tree, as well as the bases at the leaves of the dimension tree. The following proposition relates the dimension, mode sizes, and hierarchical rank with the storage complexity of this format.
\begin{proposition}\label{prop:HTD_storage}
Let $N = \max\limits_{1 \leq k \leq d} \{N_k\}$ and $r = \max\limits_{\alpha \in \mathcal{T}} \{r_\alpha\}$. Then the storage complexity of $\mathcal{X} \in \mathbb{C}^{N_1 \times \cdots \times N_d}$ for the HTD with hierarchical rank $(r_{\alpha})_{\alpha\in\mathcal{T}}$ is
\begin{equation}
    dNr + (d - 2)r^3 + r^2.
\end{equation}
\end{proposition} 
\vspace{-0.5cm} An attractive feature of the HTD is that $d$ no longer appears in the exponent of $N$. Consequently, when
\(
    d N r + (d-2)r^3 + r^2 \ll N^d,
\)
the HTD mitigates the storage costs associated with the curse of dimensionality. Furthermore, we can define the compression ratio for the HTD as follows.
\begin{definition}[Compression ratio]
\label{def:compression}  
Let $\mathcal{X} \in \mathbb{C}^{N_1 \times \cdots \times N_d}$ and let
$\widehat{\mathcal{X}}$ denote the corresponding HTD with hierarchical ranks $(r_\alpha)_{\alpha\in\mathcal{T}}$.
The compression ratio is defined as the dimensionless ratio
\begin{equation*}
    \mathrm{Compression~ratio~}(\widehat{\mathcal{X}}) :=
\frac{\texttt{ndofs}(\widehat{\mathcal{X}})}{\texttt{ndofs}(\mathcal{X})}.
\end{equation*}
Here, $\texttt{ndofs}(\mathcal{X}) = \prod_{\mu=1}^d N_{\mu}$ denotes the number of degrees of freedom in the full tensor representation, while $\texttt{ndofs}(\widehat{\mathcal{X}})$ is given by the total number of stored
entries in all transfer tensors and leaf-node basis matrices of the HTD.
\end{definition}

The HTD permits various types of operations including arithmetic, orthogonalization, tensor contraction, rank truncation, and more. We refer the reader to \cite{kressner2012htucker} and references therein for a comprehensive treatment of these algorithms. Since many fundamental operations increase the hierarchical ranks, truncation is essential for maintaining computational efficiency.

It is important to distinguish between two different truncation settings. First, when a full tensor is given, one seeks to construct a compressed HTD representation. In this case, Grasedyck \cite{grasedyck2010hierarchical_SIMAA} describes two general classes of algorithms. The first proceeds from the \emph{root to the leaves}, applying successive projections to the tensor, at a computational cost of $\mathcal{O}(N^{3d/2})$. The second proceeds from the \emph{leaves to the root}, reducing the cost to $\mathcal{O}(N^{d-1})$. Both approaches require access to the full tensor and are therefore most relevant when constructing an HTD from dense data \cite{grasedyck2010hierarchical_SIMAA,kressner2012htucker}. In contrast, the HTACA method presented in this paper offers an alternative approach to construct an HTD of high-order tensors which only requires access to entries of the tensor.

In contrast, when a tensor is already available in HTD format and one seeks to reduce its rank, a significantly more efficient truncation procedure can be used. This approach relies on a recursive computation of Gramian matrices associated with the hierarchical structure, followed by local singular value decompositions at each node. The resulting algorithm has computational complexity $\mathcal{O}(d N r^2 + d r^4)$, where $N = \max_{\mu} N_\mu$ and $r = \max_{\alpha\in\mathcal{T}} r_{\alpha}$. This setting is particularly relevant in our work, where truncation is applied repeatedly to HTDs resulting from basic operations.

\subsection{Construction}\label{sec:Construction}

This section presents the construction of the HTACA which generalizes the matrix ACA algorithm used by the SLAR method in \cite{zheng2025semi} to tensors. The core idea is to reduce the cost of computing low-rank approximations of a tensor by exploiting the recursive structure of HTD along with local access to its entries. We begin with a key relation for the \emph{root node} of the dimension tree, which links the tensor data at this node to a matrix factorization involving its left and right child frames. The same form will later be applied to other non-leaf nodes when they are regarded as the root of a corresponding subtree. This relation is formalized in the following proposition.

\begin{proposition}\label{prop:HTACA_foundation}
    For the root node $\alpha = \alpha_l \cup \alpha_r \in \mathcal{T}$, the relation
    \begin{equation}\label{eq:HTACA_foundation_1}
        \mathrm{vec}(\mathcal{X}) = U_{\alpha} = (U_{\alpha_r} \otimes U_{\alpha_l}) B_{\alpha}
    \end{equation}
    is equivalent to
    \begin{equation}\label{eq:HTACA_foundation_final}
        X^{\alpha_l} = U_{\alpha_l} \mathcal{B}_{\alpha} (U_{\alpha_r})^{\mathrm{T}}
    \end{equation}
    where $U_{\alpha_l}$ and $U_{\alpha_r}$ are the $\alpha_l$- and $\alpha_r$-frames of $\mathcal{X}$, respectively. Additionally, the matrix $B_{\alpha} \in \mathbb{C}^{r_{\alpha_l}r_{\alpha_r}\times 1}$ is the root transfer matrix, which is essentially a vector for the root, and $\mathcal{B}_{\alpha} \in \mathbb{C}^{r_{\alpha_l} \times r_{\alpha_r}}$ satisfies the identity $\mathrm{vec}(\mathcal{B}_{\alpha}) = B_{\alpha}$.
\end{proposition}

\begin{proof}
This follows directly from \cref{prop:kronecker_outer_relation}, which relates the Kronecker product to the vectorization of outer products.
\end{proof}

\noindent \Cref{prop:HTACA_foundation} shows that, at the root node (or analogously at the root of any subtree), the associated tensor block can be represented as a low-rank matrix factorization whose left and right factors are the orthonormal frame matrices inherited from its two children. This matrix perspective naturally connects the HTD representation to the ACA framework, which we review in \cref{sec:ACA} before extending it to the hierarchical setting in \cref{subsec:HTACA}.

\subsubsection{Adaptive Cross Approximation for Matrices}\label{sec:ACA}

The ACA algorithm constructs a rank-\( k \) approximation of a given a matrix \( A \in \mathbb{C}^{N_1 \times N_2} \) using a small number of selected rows and columns. It can be regarded as a practical method for computing a CUR decomposition of the form
\begin{equation}\label{eq:CUR}
    A_k = A(:, \mathcal{J})\, A(\mathcal{I}, \mathcal{J})^{-1}\, A(\mathcal{I}, :),
\end{equation}
where \( A_k \) denotes the rank-\( k \) approximation of \( A \), and \( \mathcal{I} \subset \mathbb{I}_1 \), \( \mathcal{J} \subset \mathbb{I}_2 \) are index sets for the selected rows and columns. Here, the MATLAB-style notation is used: \( A(:, \mathcal{J}) \) and \( A(\mathcal{I}, :) \) are submatrices formed by selecting the specified columns and rows, respectively, and \( A(\mathcal{I}, \mathcal{J}) \) is their intersection. Although the optimal index sets in CUR correspond to the so-called \emph{maximum volume} (maxvol) submatrix, finding them is NP-hard in general \cite{civril2007finding}. ACA avoids this difficulty by using a greedy, pivoting strategy that iteratively builds the approximation while accessing only a small portion of the matrix, which is not stored. This locality makes ACA particularly effective for problems where only local matrix entries can be evaluated, such as the large matricizations in HTD representations.

Throughout the ACA construction, we use \(\mathcal{I} = \{i_1, \dots, i_k\}\) and \(\mathcal{J} = \{j_1, \dots, j_k\}\) for the sets of selected row and column indices, respectively. The approximation \(A_k \in \mathbb{C}^{N_1 \times N_2}\) is initialized as \(A_0 = (0)_{N_1 \times N_2}\), and the residual after \(k\) steps is \(R_k = A - A_k\). At step \(k\), the algorithm operates on \(R_{k-1}\) and proceeds in two phases, illustrated schematically in \cref{fig:ACA_illustration}.

\begin{figure}[htb]
	\begin{center}
    \scalebox{0.7}{
		\begin{overpic}[scale=0.68]{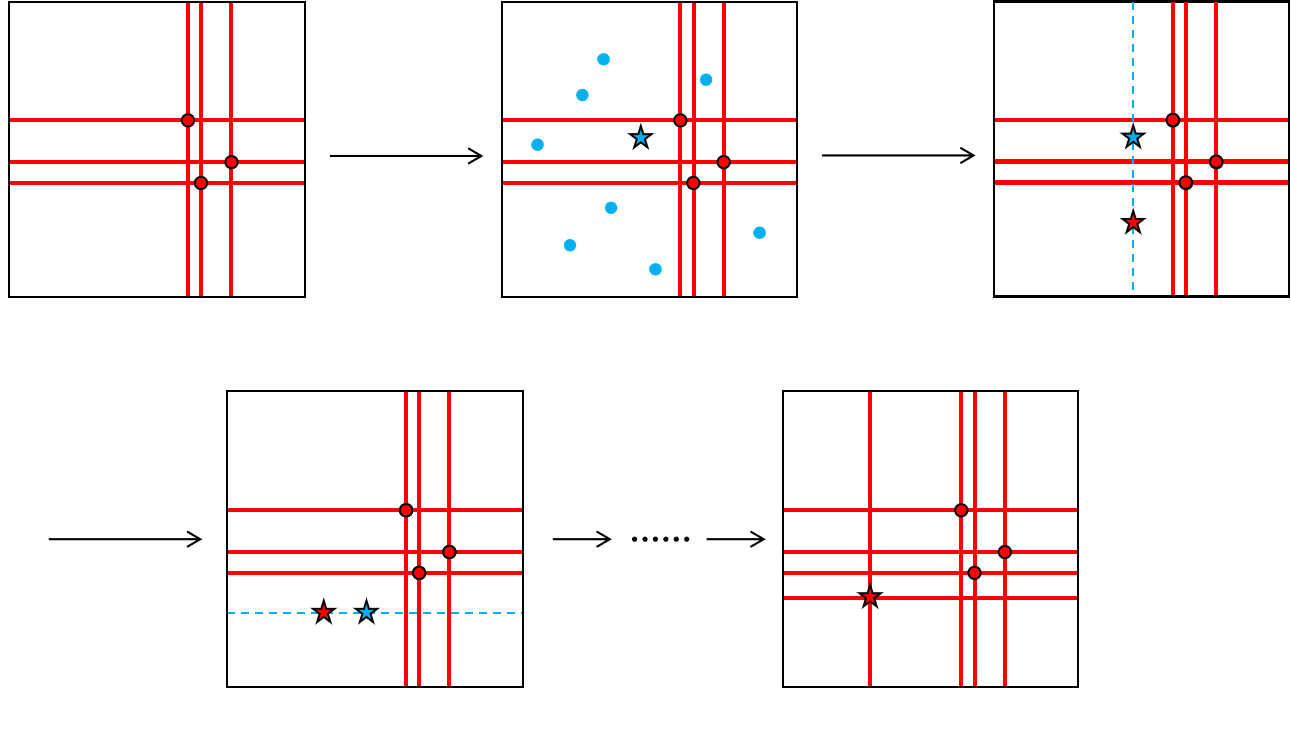}
                \put(24.5,46.5) {initial sample}
                \put(62,46.5) {column search}
                \put(4,17) {row search}
                \put(43.2,17) {rook condition}
                \put(41.2,13.2) {or repeat $m$ times}
                \put(3,32) {$R_{k-1}=A-A_{k-1}$}
                \put(86.5,32) {$j^\ast$}
                \put(15,9) {$i^\ast$}
                \put(66,2.5) {$j_k$}
                \put(58,10) {$i_k$}
                \put(61,-0.2) {$A_k = A_{k-1} + \text{Correction}(i_k,j_k)$}
		\end{overpic}
    }
	\end{center}
	\caption{Schematic illustration of the ACA algorithm for a matrix.}
	\label{fig:ACA_illustration}
\end{figure}

\paragraph{Phase I: Pivot search via residual sampling and greedy refinement}
At step \( k \), the goal is to determine the next pivot position \( (i_k, j_k) \) by locating a locally dominant entry in the residual \(R_{k-1}\). The current approximation \(A_{k-1}\) is constructed from the previously selected pivot rows and columns (red horizontal and vertical lines in the top-left panel of \cref{fig:ACA_illustration}), with the pivots marked as red circles.

We first sample a small set of residual entries from the Cartesian product of the unused index sets $\mathbb{I}_1\setminus\mathcal{I}\times\mathbb{I}_2\setminus\mathcal{J}$ (top-center panel), denoted by \(\mathcal{C} = \{(i^\ell, j^\ell)\}_{\ell=1}^s\). The entry of largest magnitude in this set is chosen as the initial estimate:
\[
(i^\ast, j^\ast) = \arg\max_{(i,j) \in \mathcal{C}} |R_{k-1}(i,j)|.
\]
This estimate (blue star) is refined by a column search: fixing \(j^\ast\), we scan the \(j^\ast\)-th column (blue dashed line in top-right panel) and update \(i^\ast\) to the row index of the largest entry in that column (red star). A subsequent row search (bottom-left panel) fixes this \(i^\ast\) and scans the corresponding row to update \(j^\ast\) (red star). The column/row search alternates until either the pivot location remains unchanged in two successive steps (rook condition) or the maximum number of refinement rounds \(m\) is reached. The final pivot \((i_k, j_k)\) (bottom-right panel) is added to \(\mathcal{I}\) and \(\mathcal{J}\).

\paragraph{Phase II: Rank-one correction}
After determining the pivot \( (i_k, j_k) \), the approximation \( A_k \) is updated via a rank-one outer product correction:
\begin{equation}\label{eq:matrix_aca_rank_one_correction}
    A_k = A_{k-1} + \frac{1}{R_{k-1}(i_k, j_k)} R_{k-1}(:, j_k) R_{k-1}(i_k, :),
\end{equation}
where the residual row and column corresponding to the pivot location are extracted from \( R_{k-1} \). This update, illustrated in the bottom-right panel of \cref{fig:ACA_illustration}, completes the step from \( A_{k-1} \) to \( A_k \).  
At every iteration, this update is mathematically equivalent to constructing a CUR decomposition of the form \cref{eq:CUR} using the current pivot sets \(\mathcal{I},\mathcal{J}\), as rigorously established in~\cite{shi2026distributed}.  
Consequently, \(A_k\) inherits the interpolation property of CUR, satisfying
\(
A_k(i,:) = A(i,:),\quad A_k(:,j) = A(:,j), \quad \forall\, i \in \mathcal{I},\ j \in \mathcal{J},
\)
which ensures that \(A_k\) exactly matches the original matrix on all selected rows and columns. Even though this is a heuristic approach, in practice, if the target matrix can be well approximated by a low-rank decomposition, and that relevant rows and columns are correctly identified, then the pivot value tends to zero and $A_k$ well approximates $A$, as iteration continues. The recursion terminates when the magnitude of the most recent pivot falls below a user-defined tolerance:
\[
|R_{k-1}(i_k, j_k)| < \varepsilon_{\mathrm{C}}.
\]
Because the pivot is the largest residual entry detected by the greedy search, this condition approximately controls the residual maximum norm, \(\|A - A_{k-1}\|_{\max}\).

The initial sample size $s$ used in Phase~I is a tunable parameter. A larger $s$ generally improves robustness by better identifying dominant residual entries, especially for large matrices that arise from matricizations used to build HTD representations. The same logic applies to the parameter $m$, which specifies the maximum number of pivot refinement rounds. In our application, accessing tensor entries requires the evaluation of a local PDE solver, so both $s$ and $m$ are adaptively set according to the size of the corresponding matricization.

\subsubsection{HTACA Construction}\label{subsec:HTACA}

We now present the HTACA algorithm, which enables memory-efficient tensor compression using only localized access to tensor entries. The algorithm is recursive and naturally extends to higher-order tensors and arbitrary (binary) dimension trees. To formalize this recursive process, we next define the notion of subtrees and introduce the concept of a subtree pool. 

\begin{definition}
(Subtree and subtree pool). Let \(\mathcal{T}\) be a dimension tree associated with a tensor \(\mathcal{X} \in \mathbb{C}^{N_1 \times \cdots \times N_d}\), where each node \(\alpha \in \mathcal{T}\) corresponds to a subset of modes \(\alpha \subseteq \{1,\dots,d\}\). A \emph{subtree} \(\mathcal{T}_\alpha\) is defined as the full subtree rooted at node \(\alpha\), containing all its descendants \(\beta \subseteq \alpha\). If \(\alpha\) is a leaf node (i.e., \(\alpha = \{\mu\}\)), we call \(\mathcal{T}_\alpha\) \emph{degenerate} and view it as a trivial single-node subtree. The \emph{subtree pool} is the collection of all such subtrees,
\(
\mathcal{P}_\mathcal{T} := \{\mathcal{T}_\alpha \mid \alpha \in \mathcal{T}\},
\)
i.e., one rooted subtree for each node of \(\mathcal{T}\).
\end{definition}

We begin the construction from the root node \(\alpha = \text{root}(\mathcal{T}) = \{1,\dots,d\}\), where the goal is to approximate the full tensor \(\mathcal{X}\) via a sequence of low-rank corrections. By \cref{prop:HTACA_foundation}, this task is equivalent to constructing a low-rank factorization of the matricization \(X^{\alpha_l} \in \mathbb{C}^{\prod_{\mu \in \alpha_l} N_\mu \times \prod_{\nu \in \alpha_r} N_\nu},\)
where \(\alpha_l\) and \(\alpha_r\) denote the left and right children of \(\alpha\). The matrix \(X^{\alpha_l}\) encodes all information of $\mathcal{X}$ and serves as the first target of approximation. In HTACA, neither full rows nor full columns of \(X^{\alpha_l}\) are available directly. Instead, each row or column corresponds to the vectorization of a smaller subtensor, which is itself defined over a subtree in \( \mathcal{P}_\mathcal{T}\). This recursive structure allows us to repeatedly apply \cref{prop:HTACA_foundation} on smaller subtrees until the relevant fiber information can be fully accessed at the leaf level. Before describing the main HTACA algorithm, we first describe how pivot entries are recursively identified using this residual-driven strategy in Phase~I.

\paragraph{Phase I: Recursive Pivot Search}  
To enable a recursive pivot search, we define the residual tensor \(\mathcal{R}_\alpha := \mathcal{X}_\alpha - \widetilde{\mathcal{X}}_\alpha\), where \(\widetilde{\mathcal{X}}_\alpha\) denotes the current HTACA of a tensor $\mathcal{X}_\alpha$ on $\mathcal{T}_\alpha$. While \(\mathcal{X}_\alpha\) is assumed to support localized access by design, it is important to note that \(\widetilde{\mathcal{X}}_\alpha\), as an HTD, also permits local access to its entries. In particular, evaluating a single entry of \(\widetilde{\mathcal{X}}_\alpha\) costs at most \(\mathcal{O}(dr^3)\) operations, where \(r\) is the maximum hierarchical rank. Consequently, the residual \(\mathcal{R}_\alpha\) can be evaluated on demand, at arbitrary indices, without explicitly forming the full tensor. This property allows the pivot search procedure to operate in a fully localized manner, guided by the magnitude of the residual. The complete pseudocode for the recursive pivot search is provided in \cref{alg:recursive_pivot_search}.

\begin{algorithm}[htb]
\caption{\texttt{recursive\_pivot\_search}: Residual-guided pivot search}
\label{alg:recursive_pivot_search}
\begin{algorithmic}[1]

\Require Residual accessor $\mathcal{R}_\alpha$; node $\alpha \in \mathcal{T}$; subtree $\mathcal{T}_\alpha$;
index sets $\mathcal{I}^{\alpha_l}, \mathcal{I}^{\alpha_r}$;
candidate pivot $\bm{i}^\alpha$ (optional); recursion depth $n$

\Ensure Pivot value $p$; pivot index $\bm{i}^\alpha \in \prod_{\mu \in \alpha}\mathbb{I}_\mu$;
updated index sets $\mathcal{I}^{\alpha_l}, \mathcal{I}^{\alpha_r}$

\If{depth$(\mathcal{T}_\alpha) > 0$}
  \State Sample $s(\alpha)$ candidate index pairs
  \Comment{see \cref{eq:sampling_params}}
  \State $\mathcal{C} \gets \{(i_\ell^{\alpha_l}, i_\ell^{\alpha_r})\}_{\ell=1}^{s(\alpha)}
          \subset (\mathbb{I}_{\alpha_l}\setminus\mathcal{I}^{\alpha_l})
          \times (\mathbb{I}_{\alpha_r}\setminus\mathcal{I}^{\alpha_r})$

  \If{prescribed $\bm{i}^\alpha$ exists}
    \State Decompose $\bm{i}^\alpha$ into $(\bm{i}^{\alpha_l}, \bm{i}^{\alpha_r})$
    \State Map to matricized indices $(i^{\alpha_l}, i^{\alpha_r})$
    \State $\mathcal{C} \gets \mathcal{C}\cup\{(i^{\alpha_l}, i^{\alpha_r})\}$
  \EndIf

  \State Evaluate residuals and select initial pivot:
  \State $(\bm{i}^{\alpha_l,*}, \bm{i}^{\alpha_r,*}) \gets
         \arg\max_{(i^{\alpha_l},i^{\alpha_r})\in\mathcal{C}}
         |\mathcal{R}_\alpha(\bm{i}^{\alpha_l}, \bm{i}^{\alpha_r})|$

  \For{$k = 1:m(\alpha)$}
    \Comment{see \cref{eq:sampling_params}}
    \State $\mathcal{R}_{\alpha_l} \gets \mathcal{R}_\alpha(:, \bm{i}^{\alpha_r,*})$
    \State $[p, \bm{i}^{\alpha_l,*}, \sim, \sim] \gets
           \texttt{recursive\_pivot\_search}\left(\mathcal{R}_{\alpha_l}, \alpha_l,
             \mathcal{T}_{\alpha_l}, \emptyset, \emptyset, \bm{i}^{\alpha_l,*}, n+1\right)$
    \State $\mathcal{R}_{\alpha_r} \gets \mathcal{R}_\alpha(\bm{i}^{\alpha_l,*}, :)$
    \State $[p, \bm{i}^{\alpha_r,*}, \sim, \sim] \gets
           \texttt{recursive\_pivot\_search}\left(\mathcal{R}_{\alpha_r}, \alpha_r,
             \mathcal{T}_{\alpha_r}, \emptyset, \emptyset, \bm{i}^{\alpha_r,*}, n+1\right)$

    \If{pivot indices unchanged}
      \State \textbf{break} \Comment{rook condition met}
    \EndIf
  \EndFor

  \State Assemble pivot multi-index:
  $\bm{i}^\alpha \gets (\bm{i}^{\alpha_l,*}, \bm{i}^{\alpha_r,*})$

  \If{$n=0$}
    \State Update index sets:
    $\mathcal{I}^{\alpha_l}\gets\mathcal{I}^{\alpha_l}\cup\{i^{\alpha_l}\}$,
    $\mathcal{I}^{\alpha_r}\gets\mathcal{I}^{\alpha_r}\cup\{i^{\alpha_r}\}$
  \Else
    \State $\mathcal{I}^{\alpha_l}\gets\emptyset$, $\mathcal{I}^{\alpha_r}\gets\emptyset$
  \EndIf

\Else
  \State $\bm{i}^\alpha \gets \arg\max_i |\mathcal{R}_\alpha(i)|$
  \State $p \gets \mathcal{R}_\alpha(\bm{i}^\alpha)$
  \State $\mathcal{I}^{\alpha_l}\gets\emptyset$, $\mathcal{I}^{\alpha_r}\gets\emptyset$
\EndIf

\end{algorithmic}
\end{algorithm}

The recursive pivot search closely follows Phase~I of the classical matrix ACA algorithm with some additional modifications for tensors. In particular, we adapt the initial sample size \( s(\alpha) \) and the maximum number of refinement rounds \(m(\alpha)\) from the matrix ACA algorithm to account for the order of \(\mathcal{T}_{\alpha}\). Specifically, we define
\begin{equation}
    s(\alpha) := 3^{\text{order}(\alpha)}, \quad m(\alpha) := \text{order}(\alpha) - 1, \label{eq:sampling_params}
\end{equation}
where order($\cdot$) represents the number of modes in a node $\alpha$. This heuristic is informed by empirical observations and provides a practical balance between sampling efficiency and refinement accuracy. We do not pursue optimal tuning of these parameters, as their best values may be application-dependent.

The algorithm first checks whether the current subtree $\mathcal{T}_{\alpha}$ has depth greater than zero, i.e., whether $\mathcal{R}_{\alpha}$ corresponds to a subtensor of order at least two. If $\text{depth}(\mathcal{T}_\alpha) = 0$, then the current node is a leaf node, and the pivot index is identified by a maximum search. Otherwise, if $\text{depth}(\mathcal{T}_\alpha) > 0$, the algorithm samples $s(\alpha)$ index pairs (see \eqref{eq:sampling_params}) from the Cartesian product of the unused index sets $(\mathbb{I}_{\alpha_l}\setminus \mathcal{I}^{\alpha_l}) \times (\mathbb{I}_{\alpha_r}\setminus \mathcal{I}^{\alpha_r})$, which are stored in a set $\mathcal{C}$. This candidate set $\mathcal{C}$ can also be augmented with a prescribed pivot from an outer recursive call, if desired. The residual is evaluated at all candidate pairs in $\mathcal{C}$, and the index pair with the largest absolute residual value is selected as the initial pivot.

Once the initial pivot is selected, we perform up to $m(\alpha)$ refinement rounds (see \eqref{eq:sampling_params}) to improve upon this initial pivot. In each round, the residual is contracted along one subset of modes (either $\alpha_l$ or $\alpha_r$) to form a corresponding subtensor accessor. For example, fixing $\bm{i}^{\alpha_r,*}$ produces the left accessor $\mathcal{R}_{\alpha_l} = \mathcal{R}_{\alpha}(:, \bm{i}^{\alpha_r,*})$, which is supported on $\mathcal{T}_{\alpha_l}$. The recursive pivot search is then applied on this subtensor to refine the left pivot index. Once a left pivot index has been identified, we repeat this process symmetrically for $\alpha_r$, producing $\mathcal{R}_{\alpha_r} = \mathcal{R}_{\alpha}(\bm{i}^{\alpha_l,*}, :)$ along with an analogous refined right pivot index. If both indices remain unchanged during a refinement round, the rook condition is satisfied and the refinement loop terminates early. This alternating refinement of left and right pivot indices emulates the classical column/row alternating search in matrix ACA algorithm, with \cref{prop:HTACA_foundation} providing the theoretical connection between them.

After the refinement terminates, the multi-index of the pivot is assembled as $\bm{i}^{\alpha} = (\bm{i}^{\alpha_l,*},\bm{i}^{\alpha_r,*})$, which corresponds to row and column indices in the $\alpha_{l}$-matricization of the residual $\mathcal{R}_{\alpha}$. In order to prepare for subsequent iterations, the current index sets $\mathcal{I}^{\alpha_l}$ and $\mathcal{I}^{\alpha_r}$ should be updated at the root node of the tree, for which the recursion depth $n = 0$. For deeper recursive calls ($n > 0$), the index sets are discarded, as these contracted subtensors are not revisited.

\begin{remark}[Direction-adaptive pivot refinement]
The refinement loop in \cref{alg:recursive_pivot_search} is presented with a fixed left-then-right search order for clarity of exposition. In practice, this fixed order introduces a systematic bias, since there is no a priori reason to prefer searching $\alpha_l$ before $\alpha_r$. To address this, we employ a direction-adaptive variant in our implementation. Prior to entering the refinement loop, we simultaneously probe both subtrees by forming $\mathcal{R}_{\alpha_l} = \mathcal{R}_\alpha(:, \bm{i}^{\alpha_r,*})$ and $\mathcal{R}_{\alpha_r} = \mathcal{R}_\alpha(\bm{i}^{\alpha_l,*}, :)$ and applying \texttt{recursive\_pivot\_search} to each independently, yielding candidate values $p_l$ and $p_r$. The direction with the larger $|p|$ is committed to first, and a follow-up search in the complementary direction completes the first refinement round. Subsequent rounds alternate in the order determined by this initial comparison.
\end{remark}

\paragraph{Phase II: Recursive Hierarchical Residual Correction} The second phase concerns an iterative refinement of a hierarchical approximation \(\widetilde{\mathcal{X}}\) through a sequence of localized ``rank-one" corrections applied to subtensors that are recursively approximated on the corresponding subtrees. Together with the recursive pivot search of Phase~I, this phase defines the HTACA algorithm, whose complete pseudocode is given in \cref{alg:HTACA}.

At a given subtree $\mathcal{T}_{\alpha}$, the algorithm first checks if the node is a leaf. In that case, the local tensor is simply vectorized and returned. Otherwise, the construction begins with a zero-initialized HTD and performs a sequence of rank-one updates to reduce the residual error $\mathcal{R}_{\alpha} := \mathcal{X}_{\alpha} - \widetilde{\mathcal{X}}_{\alpha}$. At each iteration, a dominant pivot location \(\bm{i}^\alpha\) is identified using the recursive pivot search described in \cref{alg:recursive_pivot_search}. The residual is then contracted to form two lower-dimensional subtensors, \(\mathcal{R}_{\alpha_l}\) and \(\mathcal{R}_{\alpha_r}\) which are recursively approximated as HTDs on the left and right subtrees. These contracted residuals are never stored explicitly but can be evaluated using functional accessors. Each of these recursive calls returns a pair of hierarchical representations, which are merged into an HTD \(\widetilde{\mathcal{R}}_{\alpha}\) over the subtree $\mathcal{T}_{\alpha}$ whose transfer matrix at its root contains the pivot weight. In particular, we define the scalar transfer matrix as
\(
        \widetilde{B}_\alpha = 1/(p + \operatorname{sign}(p + 0^+)\cdot 10^{-15}),
\)
where \(p\) is the residual pivot. Here, \(\operatorname{sign}(p+0^+)\) is understood as a modified sign function that takes the value \(1\) when \(p\ge 0\) and \(-1\) when \(p<0\). Thus, the perturbation term \(\operatorname{sign}(p + 0^+) \cdot 10^{-15}\) shifts the denominator slightly away from zero in the direction determined by the sign of \(p\), with the zero case treated as positive, thereby serving as a safeguard against division by zero. With this final assignment, the HTD \(\widetilde{\mathcal{R}}_\alpha\) on \(\mathcal{T}_\alpha\) is fully defined. This update mirrors the matrix ACA rank-one correction step~\eqref{eq:matrix_aca_rank_one_correction}. For example, the \(\alpha_l\)-matricization of \(\widetilde{\mathcal{R}}_{\alpha}\) yields
\begin{equation*}
\widetilde{R}_{\alpha}^{\alpha_l} = \frac{1}{p + \operatorname{sign}(p + 0^+)\cdot 10^{-15}}\, \text{vec}(\widetilde{\mathcal{R}}_{\alpha_l}) \, \text{vec}(\widetilde{\mathcal{R}}_{\alpha_r})^\mathrm{T}.    
\end{equation*}
After forming the rank-one correction \(\widetilde{\mathcal{R}}_{\alpha}\), we update the current approximation $\widetilde{\mathcal{X}}_{\alpha}$ using tensor addition. However, the addition of two HTDs is known to cause a rapid increase in intermediate hierarchical ranks, as the result of the addition cannot be represented exactly with the same ranks as the summands \cite{kressner2012htucker}. Therefore, to avoid uncontrolled rank growth, we immediately apply a compression step:
\[
\widetilde{\mathcal{X}}_{\alpha} \gets \texttt{truncate} \Big(\widetilde{\mathcal{X}}_{\alpha} + \widetilde{\mathcal{R}}_{\alpha},~10^{-14} \max_n \| X^{\alpha}(:,n) \|,~r_{\min}(\mathcal{T}_\alpha),~r_{\max}(\mathcal{T}_\alpha) \Big).
\]
Here, the truncation procedure removes numerically redundant components while keeping hierarchical ranks within user-prescribed bounds. We follow the standard hierarchical singular value decomposition (HSVD) approach proposed in~\cite[Algorithm~6]{kressner2012htucker}, originally based on the theory in~\cite[Section~4]{grasedyck2010hierarchical}. The call to \texttt{truncate} represents a function that takes as input: (1) the current HTD to be compressed; (2) an absolute Frobenius-norm tolerance that controls truncation error; (3) a minimal and maximal hierarchical rank for the current subtree. Unlike the final global compression step, which will be discussed later, this local truncation step is not designed to control the global approximation error. Instead, we conservatively set the relative truncation tolerance to a small constant \(10^{-14}\), which ensures that only negligible components are discarded. This prevents uncontrolled rank growth during recursive updates and avoids the amplification of round-off errors. In estimating the normalization factor \(\max_n \| X^{\alpha}(:,n) \|\), we face the challenge that the original tensor \(\mathcal{X}\) is not globally accessible. Therefore, we approximate it using the Frobenius norm of the first rank-one correction encountered in the current subtree. Although this underestimates the true maximum, the resulting truncation is stricter, which empirically helps preserve useful structure.

After all residual-guided rank-one corrections have been assembled through the recursive process, we apply a final truncation step to the entire HTD structure \(\widetilde{\mathcal{X}}_{\alpha}\) at the root node $\alpha = \text{root}(\mathcal{T})$. This ensures that accumulated numerical noise and unnecessary components from intermediate contractions are removed in a globally consistent manner. We denote this process as
\begin{equation*}
    \widetilde{\mathcal{X}}_{\alpha} \gets \texttt{truncate}\Big(\widetilde{\mathcal{X}}_{\alpha}, ~\varepsilon_{\text{Base}}\|\widetilde{\mathcal{X}}_{\alpha}\|, ~r_{\min}(\mathcal{T}), ~r_{\max}(\mathcal{T})\Big).
\end{equation*}
Compared to the conservative local truncation tolerance \(10^{-14} \max_k \| X^{\alpha}(:,k) \|\) used during each intermediate update, this root-level truncation applies a coarser but well-calibrated absolute threshold proportional to the final approximation norm. This design matches the ACA + SVD strategy adopted in our previous work~\cite{zheng2025semi,sands2026adaptive}, which used the SVD to improve the robustness of the algorithm.

\paragraph{Local Pivoting Tolerances and Error Control}
Unlike the HSVD-based truncation tolerance, which controls the low-rank compression error after each correction, the following pivoting tolerance \(\varepsilon_C\) is designed to determine the termination of residual-driven updates throughout the HTACA procedure.

To ensure convergence of the root-level ACA iteration, we assign for each subtree \(\mathcal{T}_\alpha\) a local pivoting tolerance \(\varepsilon_C(\varepsilon_{\text{Base}}, \gamma, \mathcal{T}_\alpha)\), defined as
\begin{equation}\label{eq:hierach_pivoting}
\varepsilon_C(\varepsilon_{\text{Base}}, \gamma, \mathcal{T}_\alpha) = \gamma^{\text{depth}(\mathcal{T}) - \text{depth}(\mathcal{T}_{\alpha})}\varepsilon_{\text{Base}}\|\mathcal{X}\|_{\max}.
\end{equation}
This exponential decay schedule ensures that subtrees closer to the leaves are constructed more accurately. The rationale is that their approximations will be recursively used to build larger residuals at lower levels; hence they must be sufficiently accurate to avoid error accumulation. Since the true maximum entry norm \(\|\mathcal{X}\|_{\max}\) is generally unavailable in our setting, we approximate it by the absolute value of the first pivot encountered during the root-level ACA. Although this underestimates the true value, the resulting conservative tolerance promotes numerical robustness and consistency with our truncation heuristics. In our numerical experiments, we set \(\gamma = 1/10\). Based on empirical observations, more complex or higher-rank tensors benefit from more aggressive decay factors (smaller \(\gamma\)), whereas extremely low-rank tensors may tolerate \(\gamma = 1\) without compromising accuracy or convergence.

\begin{remark}[Rank control parameters in HTACA]
Throughout the HTACA construction, we use four types of rank-related parameters:
\begin{itemize}
  \item \( r_{\min}(\mathcal{T}_\alpha) \), \( r_{\max}(\mathcal{T}_\alpha) \): user-defined lower and upper bounds on the \emph{hierarchical ranks} for each node in the subtree \(\mathcal{T}_\alpha\). These bounds are enforced during the local truncation step to limit the rank growth of the approximation.
  
  \item \( r_{\#,\min}(\mathcal{T}_\alpha) \), \( r_{\#,\max}(\mathcal{T}_\alpha) \): minimal and maximal number of recursive \emph{ACA correction steps} performed during Phase~II on the subtree \(\mathcal{T}_\alpha\). These ensure that a minimum number of updates are performed regardless of pivot size, and that the iteration terminates after a reasonable number of steps even if the pivot threshold is not met.
\end{itemize}
Unless otherwise specified, \( r_{\min}(\mathcal{T}) \), \( r_{\max}(\mathcal{T}) \), \( r_{\#,\min}(\mathcal{T}) \), and \( r_{\#,\max}(\mathcal{T}) \) are inherited by all subtrees in a uniform fashion.
\end{remark}

\begin{remark}[Matrix branch in the implementation]
In the implementation, the matrix case, where \(\mathcal{T}_\alpha\) consists of two leaf nodes, is handled by a separate ACA branch. In this branch, the accumulated rank-one corrections are updated using the matrix ACA update formula given in Corollary~2.2 of \cite{zheng2025semi}, so no truncation is required after each rank-one correction. The local truncation step in Algorithm~\ref{alg:HTACA} is only applied in the higher-order HTD construction.
\end{remark}

\begin{algorithm}[h!]
\caption{HTACA for tensors}
\label{alg:HTACA}
\begin{algorithmic}[1]

\Require Tensor $\mathcal{X}_\alpha$; node $\alpha$; subtree $\mathcal{T}_\alpha$;
baseline tolerance $\varepsilon_{\text{Base}}$; decay factor $\gamma$

\Ensure Approximation
$\widetilde{\mathcal{X}}_\alpha =
((B_\beta)_{\beta \in \mathcal{N}(\mathcal{T}_\alpha)},
 (U_\beta)_{\beta \in \mathcal{L}(\mathcal{T}_\alpha)})$

\If{$\text{depth}(\mathcal{T}_\alpha) > 0$}
  \State Initialize zero HTD:
  $\widetilde{\mathcal{X}}_\alpha \gets ((B_\beta)_{\beta \in \mathcal{N}(\mathcal{T}_\alpha)},
  (U_\beta)_{\beta \in \mathcal{L}(\mathcal{T}_\alpha)})$
  \State Initialize: $p \gets \infty$,
  $\mathcal{I}^{\alpha_l}\gets\emptyset$,
  $\mathcal{I}^{\alpha_r}\gets\emptyset$, $k \gets 1$

  \While{$|p| > \varepsilon_C(\varepsilon_{\text{Base}}, \gamma, \mathcal{T}_\alpha)$
  or $k < r_{\#,\min}(\mathcal{T}_\alpha)$} \Comment{see \cref{eq:hierach_pivoting}}
    \State Assemble residual:
    $\mathcal{R}_\alpha \gets \mathcal{X}_\alpha - \widetilde{\mathcal{X}}_\alpha$
    \State $[p, \bm{i}^\alpha, \mathcal{I}^{\alpha_l}, \mathcal{I}^{\alpha_r}] \gets
      \texttt{recursive\_pivot\_search}\left(\mathcal{R}_\alpha, \alpha, \mathcal{T}_\alpha,
      \mathcal{I}^{\alpha_l}, \mathcal{I}^{\alpha_r}, \sim, 0\right)$

    \State Decompose $\bm{i}^\alpha$ into $(\bm{i}^{\alpha_l}, \bm{i}^{\alpha_r})$
    \State Form contracted residuals:
    $\mathcal{R}_{\alpha_l} \gets \mathcal{R}_\alpha(:, \bm{i}^{\alpha_r})$,
    $\mathcal{R}_{\alpha_r} \gets \mathcal{R}_\alpha(\bm{i}^{\alpha_l}, :)$

    \State
    $\widetilde{\mathcal{R}}_{\alpha_l} = ((\widetilde{B}_\beta)_{\beta\in\mathcal{N}(\mathcal{T}_{\alpha_l})},
       (\widetilde{U}_\beta)_{\beta\in\mathcal{L}(\mathcal{T}_{\alpha_l})})\gets
      \texttt{HTACA}\left(\mathcal{R}_{\alpha_l}, \alpha_l, \mathcal{T}_{\alpha_l},
      \varepsilon_{\text{Base}}, \gamma\right)$
    \State
    $\widetilde{\mathcal{R}}_{\alpha_r}=((\widetilde{B}_\beta)_{\beta\in\mathcal{N}(\mathcal{T}_{\alpha_r})},
       (\widetilde{U}_\beta)_{\beta\in\mathcal{L}(\mathcal{T}_{\alpha_r})}) \gets
      \texttt{HTACA}\left(\mathcal{R}_{\alpha_r}, \alpha_r, \mathcal{T}_{\alpha_r},
      \varepsilon_{\text{Base}}, \gamma\right)$

    \State Set $\widetilde{B}_\alpha \gets 1 / (p + \text{sign}(p)\cdot 10^{-15})$
    \State Assemble rank-one correction:
    $\widetilde{\mathcal{R}}_\alpha \gets
      ((\widetilde{B}_\beta)_{\beta\in\mathcal{N}(\mathcal{T}_\alpha)},
       (\widetilde{U}_\beta)_{\beta\in\mathcal{L}(\mathcal{T}_\alpha)})$

    \State Update approximation:
     \begin{equation*}
         \widetilde{\mathcal{X}}_\alpha \gets
    \texttt{truncate}\left(\widetilde{\mathcal{X}}_\alpha + \widetilde{\mathcal{R}}_\alpha,
      10^{-14}\max_n \|X^\alpha(:,n)\|,
      r_{\min}(\mathcal{T}_\alpha), r_{\max}(\mathcal{T}_\alpha)\right)
     \end{equation*}
    \If{$k+1 > r_{\#,\max}(\mathcal{T}_\alpha)$}
      \State \textbf{break}
    \EndIf

    \State $k \gets k+1$
  \EndWhile

  \If{$\alpha = \text{root}(\mathcal{T})$}
    \State $\widetilde{\mathcal{X}}_\alpha \gets
      \texttt{truncate}\left(\widetilde{\mathcal{X}}_\alpha,
      \varepsilon_{\text{Base}}\|\widetilde{\mathcal{X}}_\alpha\|,
      r_{\min}(\mathcal{T}), r_{\max}(\mathcal{T})\right)$
  \EndIf

\Else
  \State Leaf-level approximation:
  $\widetilde{\mathcal{X}}_\alpha = U_\alpha \gets \text{vec}(\mathcal{X}_\alpha)$
\EndIf

\end{algorithmic}
\end{algorithm}

\section{SLAR Method for the Nonlinear VP System}\label{sec:SLAR}

We begin by defining the SLAR method for linear advection equations, which combines the local SL--FD solver of \cref{sec:SL-FD-high-D} with the HTACA compression technique introduced in \cref{sec:HTACA}.  
Given a high-dimensional linear advection equation~\eqref{eq:advection_ddim} with velocity field $\bm{a}(\bm{x}, t)$, the SLAR method evolves the solution in the HTD format from time $t^n$ to $t^{n+1}$.  
We denote this evolution as
\begin{equation*}
    \mathcal{F}^{n+1} = \text{SLAR}(\mathcal{F}^n, t^n, t^{n+1}, \bm{a}(\bm{x}, t), \varepsilon_{\text{Base}}),
\end{equation*}
where localized access to future-time tensor values is provided by the local SL--FD solver, and the HTD approximation $\mathcal{F}^{n+1}$ is constructed using the HTACA algorithm with user-defined relative tolerance $\varepsilon_{\text{Base}}$.  
Although some of the HTACA-related parameters are omitted here for conciseness, the complete parameter settings will be specified in the tests presented in \cref{sec:numerical_tests}.

To solve the nonlinear VP system, we adopt a third-order RK exponential integrator (RKEI) \cite{cai2021high}, which decomposes the nonlinear characteristic tracing into a sequence of linear advection steps, each using a frozen velocity field.  
This approach allows for a straightforward generalization of our earlier 1D1V method~\cite{zheng2025semi} to 2D2V and 3D3V problems.  
At each stage, the corresponding linear advection problem is solved using the SLAR method defined above:
\begin{equation}\label{eq:CF3}
\begin{aligned}
    \mathcal{F}^{(1)} &= \text{SLAR}(\mathcal{F}^n, t^n, t^{n+1}, \tfrac{1}{3}\bm{v}, \tfrac{1}{3}\bm{E}^n(\bm{x}), \varepsilon_{\text{Base}}), \\
    \mathcal{F}^{(2)} &= \text{SLAR}(\mathcal{F}^n, t^n, t^{n+1}, \tfrac{2}{3}\bm{v}, \tfrac{2}{3}\bm{E}^{(1)}(\bm{x}), \varepsilon_{\text{Base}}), \\
    \mathcal{F}^{n+1,\star} &= \text{SLAR}(\mathcal{F}^{(1)}, t^n, t^{n+1}, \tfrac{2}{3}\bm{v}, -\tfrac{1}{12}\bm{E}^{n}(\bm{x}) + \tfrac{3}{4}\bm{E}^{(2)}(\bm{x}), \varepsilon_{\text{Base}}),
\end{aligned}
\end{equation}
where $\bm{E}^{n}$, $\bm{E}^{(1)}$, and $\bm{E}^{(2)}$ are electric fields computed from $\mathcal{F}^{n}$, $\mathcal{F}^{(1)}$, and $\mathcal{F}^{(2)}$, respectively, using the Poisson solver described in the next subsection, which directly operates on HTD-formatted distributions.

\subsection{FFT-Based Field Solver for HT-Formatted Distributions}\label{sec:FFT_solver}

The self-consistent electric field at each stage of \cref{eq:CF3} requires the gradient of a scalar potential, obtained by solving the Poisson equation:
\begin{equation}
\label{eq:Poisson equation for phi}
    -\Delta_{\bm{x}} \phi = \rho(\bm{x}, t), 
    \quad 
    \rho(\bm{x}, t) = \rho_0 - \int_{\mathbb{R}^{d_v}} f(\bm{x}, \bm{v}, t) \,\mathrm{d}\bm{v}.
\end{equation}
We consider problems defined on periodic domains, where direct Fourier methods are a natural choice and are fully compatible with the HTACA framework.  
Here, the distribution function $f(\bm{x}, \bm{v})$ is represented as an HTD:
\[
\mathcal{F} = \bigl((B_\alpha)_{\alpha \in \mathcal{N}(\mathcal{T})}, (U_\alpha)_{\alpha \in \mathcal{L}(\mathcal{T})}\bigr).
\]

To evaluate the integral in \cref{eq:Poisson equation for phi}, we apply midpoint quadrature in the velocity dimensions by summing over the rows of each velocity leaf-node basis and multiplying by the corresponding mesh size.  
This contraction reduces the velocity mode sizes to~1, after which these degenerate modes are removed via dimension squeezing operations supported by standard HTD algorithms~\cite{kressner2012htucker}.  
The result is a low-rank HTD approximation of the charge density,
$
\bm{\rho} = \bigl((B_\alpha)_{\alpha \in \mathcal{N}(\mathcal{T}_{\bm{x}})}, (U_\alpha)_{\alpha \in \mathcal{L}(\mathcal{T}_{\bm{x}})}\bigr),
$
where 
$\mathcal{T}_{\bm{x}}$ is the restriction of the original dimension tree to spatial variables.  
Because $f(\bm{x}, \bm{v})$ decays exponentially as $|\bm{v}| \to \infty$, the quadrature error is dominated by domain truncation~\cite[Chap.~17]{boyd2001chebyshev} rather than discretization error.

Suppose the numerical solution is defined on a periodic domain with lengths $L_{x^{(1)}}, L_{x^{(2)}}, L_{x^{(3)}}$ in each spatial direction.  
Let $\widehat{\phi}$ denote the discrete Fourier transform (DFT) of $\phi$, and let $k_{x^{(1)},i}, k_{x^{(2)},j}, k_{x^{(3)},k}$ be the discrete wavenumbers associated with the domain lengths and grid resolution.  
The DFT diagonalizes the Poisson equation~\cref{eq:Poisson equation for phi}, yielding for each nonzero mode $(i,j,k)$:
\begin{equation}
    \label{eq:FFT Poisson solution}
    \widehat{\phi}_{i,j,k} = \frac{1}{k_{x^{(1)},i}^{2} + k_{x^{(2)},j}^{2} + k_{x^{(3)},k}^{2}} \, \widehat{\rho}_{i,j,k}, 
    \quad (i,j,k) \neq (0,0,0).
\end{equation}
Since only the gradient of $\phi$ is needed, we enforce $\widehat{\phi}_{0,0,0} = 0$ to ensure a zero-mean solution.  

As the spatial dependence is encoded in the leaf-node bases of the HTD, the DFT can be applied mode-wise to these matrices. The Fourier-transformed HTD of the charge density has the form
$
\widehat{\bm{\rho}} = \bigl((B_\alpha)_{\alpha \in \mathcal{N}(\mathcal{T}_{\bm{x}})}, (\widehat{U}_\alpha)_{\alpha \in \mathcal{L}(\mathcal{T}_{\bm{x}})}\bigr),
$
where each $\widehat{U}_\alpha$ is obtained by applying a 1D FFT to the columns of $U_\alpha$.  
Since $U_\alpha \in \mathbb{C}^{N_\alpha \times r_\alpha}$, the FFT is applied along rows using batched FFTs. This mode-wise FFT achieves complexity $\mathcal{O}(d_x r N \log N)$, compared with the full-grid cost $\mathcal{O}(N^{d_x} \log N)$, and is well suited to parallelization.

To construct the HTD of the frequency-domain solution $\widehat{\phi}(\bm{k})$, we apply HTACA to the tensor defined by \cref{eq:FFT Poisson solution}. The electric field components are then obtained in Fourier space as
\[
\widehat{E}_{x^{(1)}} = -i k_{x^{(1)}} \widehat{\phi}, 
\quad 
\widehat{E}_{x^{(2)}} = -i k_{x^{(2)}} \widehat{\phi}, 
\quad 
\widehat{E}_{x^{(3)}} = -i k_{x^{(3)}} \widehat{\phi},
\]
where $i=\sqrt{-1}$ is the imaginary unit. Finally, inverse FFTs applied to the leaf-node bases of $\mathcal{T}_{\bm{x}}$ yield the HTD representations of $\bm{E}$ in physical space.

\begin{remark}[Pivot initialization in the frequency domain]
When HTACA is applied to the Fourier-transformed tensor in \eqref{eq:FFT Poisson solution}, the entries may be highly nonuniform because, for \(\bm{k}\ne 0\), \(\widehat{\phi}_{\bm{k}}=\widehat{\rho}_{\bm{k}}/|\bm{k}|^2\). Thus the lowest nonzero modes may dominate the potential, while high-frequency modes are strongly damped. To avoid missing these modes in the initial pivot search, the random candidate set is augmented with a small number of prescribed (super)diagonal low-frequency candidate indices in the Fourier-mode ordering, excluding the zero mode. In our numerical experiments, 16 such indices are included. The final pivot is still selected by residual maximization and the usual recursive refinement. The purpose of these additional candidates is to improve the robustness of the initial sampling used for the Fourier-space field solver.
\end{remark}

\begin{remark}[HTACA for entry-wise tensor definitions]
A notable advantage of HTACA is its ability to compress tensors given by explicit entry-wise formulas without enumerating all entries. For instance, the frequency-domain solution in \cref{eq:FFT Poisson solution} can be directly compressed without computing a Hadamard product between a pair of HTDs. This allows seamless integration with solvers and avoids explicitly assembling large tensors.
\end{remark}

\begin{remark}[Complex-valued HTD and real-valued recovery]
HTACA supports compression of complex-valued tensors, as required for frequency-domain calculations.  
Although the inverse DFT of $\widehat{\bm{\phi}}$ or $\widehat{\bm{E}}$ yields real-valued tensors, the intermediate HTDs may contain small imaginary parts due to roundoff.  
These are of size $\mathcal{O}(\varepsilon_{\text{Base}})$ and can be discarded by taking the real part when extracting physical quantities.
\end{remark}

\begin{remark}[Off-grid electric field evaluation]
Although $\bm{E}^n(\bm{x})$, $\bm{E}^{(1)}(\bm{x})$, and $\bm{E}^{(2)}(\bm{x})$ are stored in HTD format at grid points, the RKEI scheme requires values at off-grid locations during characteristic tracing.  
We evaluate these by applying the same $P^2$ polynomial reconstruction procedure described in \cref{sec:poly_rec}.
\end{remark}

\subsection{Complexity estimate for the SLAR method}\label{sec:SLAR_complexity}

We provide a rough upper bound on the computational complexity of the SLAR method by focusing on two key components: the number of original tensor entries accessed and the dominant cost of evaluating each entry. Assume that the tensor \(\mathcal{F} \in \mathbb{R}^{N_1 \times \cdots \times N_d}\) has uniform mode size \(N = N_1 = \cdots = N_d\) and uniform hierarchical rank \(r\). In the worst case, the HTACA compression step may request up to \(\mathcal{O}(d N r^{\lceil \log_2 d \rceil})\) entries from the input tensor. This bound accounts for all recursive sampling paths in a balanced dimension tree, assuming no early termination due to residual convergence.

Each such entry must be locally reconstructed by the SL-FD solver using stencil values from the underlying HTD, typically requiring $\mathcal{O}(d^2)$ points per evaluation. The cost of evaluating a single point from the HTD is \(\mathcal{O}(d r^3)\) in the worst case, leading to an overall per-entry evaluation cost of \(\mathcal{O}(d^3 r^3)\). Combining these estimates, the total complexity of SLAR is $\mathcal{O}(d^4 N r^{3 + \lceil \log_2 d \rceil}).$
We omit the complexity of other HTACA components, as our numerical experiments suggest that their cost is negligible compared to that of SL-FD–based entry evaluations. Deriving a complete and rigorous complexity analysis for the full HTACA procedure is nontrivial and beyond the scope of this paper.

The above estimate offers two important insights. First, the presence of the exponential factor \(r^{\lceil \log_2 d \rceil}\) indicates that HTACA may become inefficient for very high-dimensional tensors. However, for kinetic applications with \(d = 6\), the growth remains acceptable. Second, the high cost \(\mathcal{O}(d^3 r^3)\) of evaluating a single tensor entry from a local PDE solver suggests considerable potential for parallelization. Since each entry evaluation is a heavy, independent task, HTACA-based PDE solvers seem well-suited for parallel acceleration in high-performance computing environments. The algorithms proposed in this paper have been implemented in MATLAB with a serial execution model for testing and validation purposes. However, in future work, we shall consider implementations which can better exploit hardware capabilities of many-core computing platforms.

\section{Numerical tests}\label{sec:numerical_tests}

Next, we evaluate the performance of the proposed SLAR method in high-dimensional settings by applying the method to several benchmark problems, namely the Landau damping and the two-stream instability. For each test, we apply periodic boundary conditions in all spatial dimensions and truncate the velocity domain to a symmetric box \([ -v_{\max}, v_{\max} ]^{d_v}\), where \(v_{\max}\) denotes the maximum velocity cutoff. In all examples, the electric field is obtained using the Poisson solver presented in \cref{sec:SLAR}, with a truncation tolerance fixed at $0.1\varepsilon_{\text{Base}}$, where $\varepsilon_{\text{Base}}$ is a prescribed relative tolerance for the kinetic solver. The time step size is determined by
\begin{equation}\label{eq:time_step}
\Delta t = \text{CFL}/\left(\sum_{\mu=1}^{d_x} \frac{v_{\max}}{\Delta x^{(\mu)}} + \sum_{\mu=1}^{d_v} \frac{\max_{\bm{x}} |E_{x^{(\mu)}}(\bm{x})|}{\Delta v^{(\mu)}}\right).
\end{equation}
Since SL schemes are not sensitive to time step sizes, we shall approximate each term \(\max_{\bm{x}} |E_{x^{(\mu)}}(\bm{x})|\) by \(\|\bm{E}_{x^{(\mu)}}\|/\sqrt{\prod_{\mu=1}^{d_x} N_{x^{(\mu)}}}\), where \(\bm{E}_{x^{(\mu)}}\) is the HTD approximation of \(E_{x^{(\mu)}}(\bm{x})\), and \(\|\cdot\|\) denotes the Frobenius norm of high-dimensional tensors, which can be efficiently evaluated in the HTD format~\cite{kressner2012htucker}. Unless otherwise stated, no maximum rank limitation is imposed in our simulations.

The VP system possesses a number of physical invariants. To facilitate their definitions, let 
\[
    \Big\langle \cdot \Big\rangle_{\Omega_{\bm{x}}} := \int_{\Omega_{\bm{x}}} (\cdot) \,d\bm{x}, \quad
    \Big\langle \cdot \Big\rangle_{\Omega_{\bm{v}}} := \int_{\Omega_{\bm{v}}} (\cdot) \,d\bm{v}, \quad 
    \Big\langle \cdot \Big\rangle_{\Omega} := \int_{\Omega_{\bm{x}}}\int_{\Omega_{\bm{v}}} (\cdot) \,d\bm{v}\,d\bm{x}
\]
denote integration over $\Omega_{\bm{x}}$, $\Omega_{\bm{v}}$, and $\Omega_{\bm{x}}\times\Omega_{\bm{v}}$, respectively. The number density, momentum density, and kinetic energy density are then defined as the moments of $f(\bm{x}, \bm{v},t)$:
\begin{equation*}
    n(\bm{x},t) = \Big\langle f \Big\rangle_{\Omega_{\bm{v}}}, \quad 
    n\bm{u}(\bm{x},t) = \Big\langle \bm{v} f \Big\rangle_{\Omega_{\bm{v}}}, \quad 
    \kappa(\bm{x},t) = \Big\langle \tfrac{1}{2}\left\lvert \bm{v} \right\rvert^{2} f \Big\rangle_{\Omega_{\bm{v}}}.
\end{equation*}
We additionally define the electrostatic energy density as $e(\bm{x},t) = \tfrac{1}{2}\left\lvert \bm{E}(\bm{x},t) \right\rvert^{2}$. On periodic domains, the VP system preserves the total particle count $\langle n \rangle_{\Omega_{\bm{x}}}$, total momentum $\langle n\bm{u}\rangle_{\Omega_{\bm{x}}}$, and total energy $\langle \kappa + e\rangle_{\Omega_{\bm{x}}}$. While the SLAR method proposed in this work does not explicitly enforce the conservation of these quantities, we monitor their evolution to assess the quality of the numerical solution.

All numerical experiments were performed on the DARWIN high-performance
computing system at the University of Delaware. The jobs were run on standard
compute nodes equipped with two AMD EPYC 7502 processors, providing 64 CPU
cores and 512 GiB of memory per node. Each MATLAB job was submitted as a
single-task Slurm job requesting one CPU core and 8 GB of memory. The algorithms
were implemented in MATLAB R2023b. For each numerical example, we report representative wall-clock  
times measured using MATLAB's built-in timing functions. The HTD-based operations used in the
implementation are provided by our lightweight functional HT toolbox, which is
inspired by the HTucker toolbox \cite{kressner2012htucker} but reimplements the
core routines and extends the tree structure for the needs of the HTACA
algorithm.  The code used to
produce the results in this paper is available at \href{https://github.com/Nanyheng/highD_SLAR}{\texttt{https://github.com/Nanyheng/highD\_SLAR}}.

\begin{example}[Landau damping]
Consider the VP system in 2D2V and 3D3V phase space with initial condition
$$f(\bm{x},\bm{v},0) 
    = \frac{1}{(2\pi)^{d_x/2}}
      \left(1+\alpha\sum_{\mu=1}^{d_x}\cos(k x^{(\mu)})\right)
      \exp\!\left(-\frac{|\bm{v}|^2}{2}\right),
$$
with $k=0.5$. We test weak damping with $\alpha=0.01$, and strong damping with $\alpha=0.5$ (2D2V) or $\alpha=1/3$ (3D3V), the latter chosen to avoid negative initial distributions.

\Cref{fig:SLD_tree_comp} compares the rank histories of the SLAR method for two mode orderings in the dimension tree (\cref{fig:balanced_tree_4d}): $\{x^{(1)}, x^{(2)}, v^{(1)}, v^{(2)}\}$ (left) and $\{x^{(1)}, v^{(1)}, x^{(2)}$, $v^{(2)}\}$ (right). 
In the $\{x^{(1)}, x^{(2)}, v^{(1)}, v^{(2)}\}$ ordering, the two non-leaf ranks $r_{12}$ and $r_{34}$ grow rapidly, eventually exceeding $100$, indicating a substantial increase in computational complexity. 
In contrast, the $\{x^{(1)}, v^{(1)}, x^{(2)}, v^{(2)}\}$ ordering maintains all ranks at significantly lower levels throughout the simulation. 
This motivates the use of same-dimension pairing for the Landau damping test. For the 3D3V case, we adopt the unbalanced dimension tree (\cref{fig:unbalanced_tree_6d}) to avoid splitting a spatial–velocity pair at the root node.

\begin{figure}[htb]
    \centering
    \subfigure{
        \includegraphics[width=0.325\linewidth]{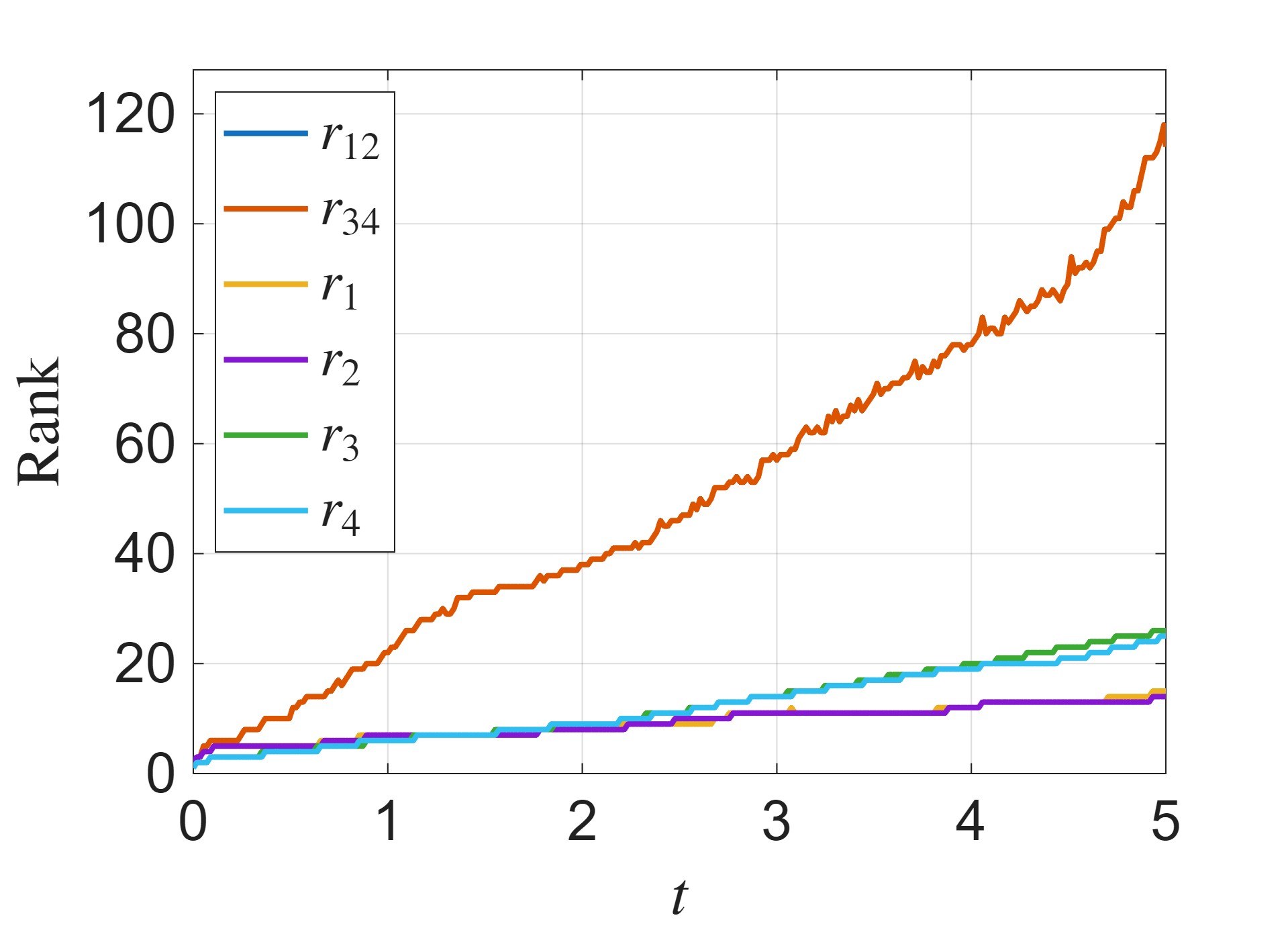}
    }
    \subfigure{
        \includegraphics[width=0.325\linewidth]{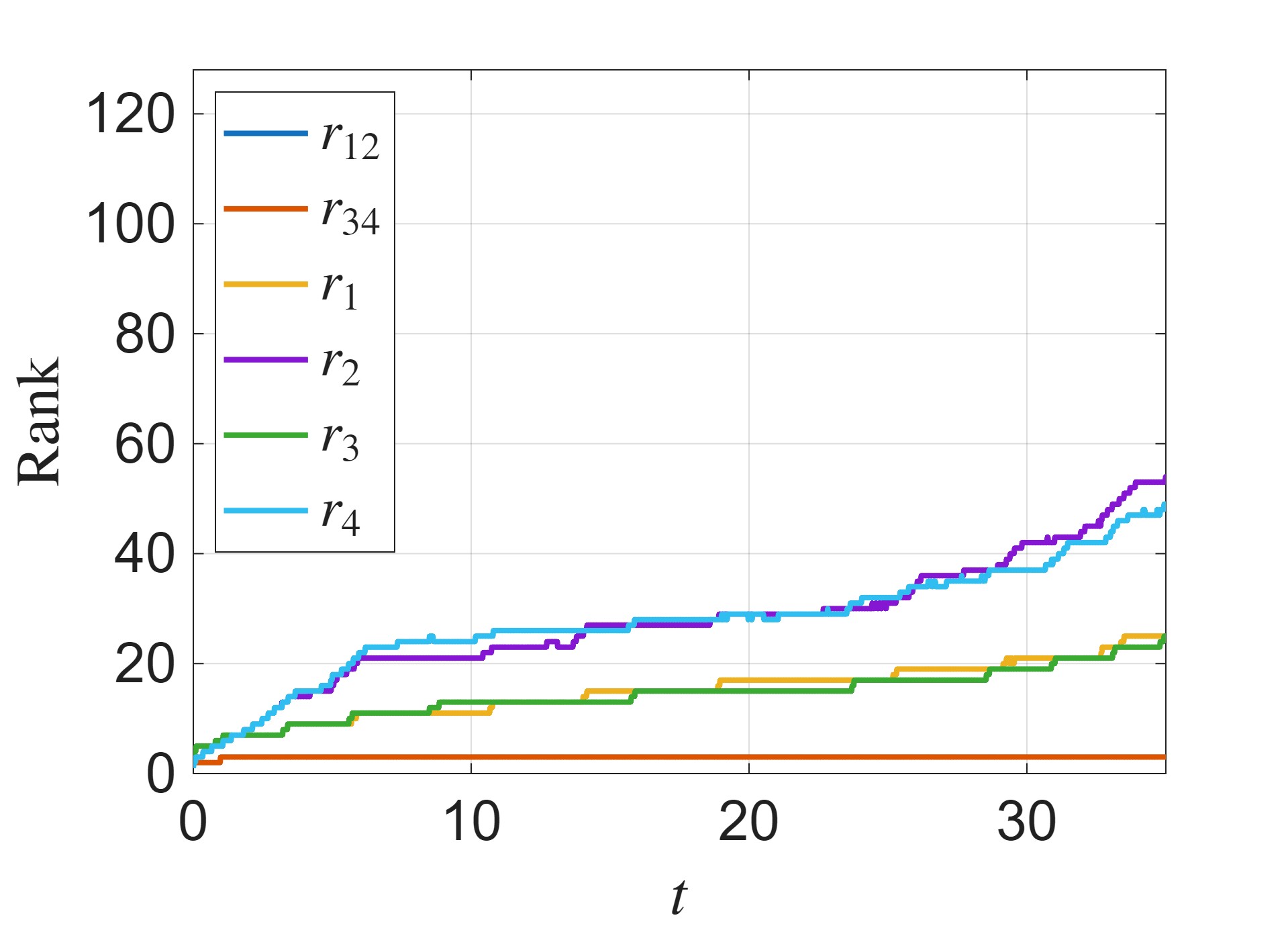}
    }
    \caption{(2D2V strong Landau damping). Rank history of the SLAR method with the dimension tree in \cref{fig:balanced_tree_4d}: Left $-$ modes ordered as $\{x^{(1)}, x^{(2)}, v^{(1)}, v^{(2)}\}$; Right $-$ modes ordered as \{$x^{(1)}$, $v^{(1)}$, $x^{(2)}$, $v^{(2)}$\}. Simulation settings: $v_{\max}=2\pi$, mesh $256^4$, CFL $=5$, $\epsilon_{\text{Base}} = 10^{-3}$, and no rank limitations.}
    \label{fig:SLD_tree_comp}
\end{figure}

\Cref{fig:SLD_accuracy} shows spatial (left) and temporal (right) refinement results for the 2D2V strong Landau damping. 
The reversibility of the VP system is used to assess accuracy by evolving to time $t = T$, reversing velocity, and evolving the solution backward in time until $t = 0$. Ideally, the solution should return to the initial condition, which can be used as a reference solution.
As shown in the figure, third-order accuracy is observed in both space and time, confirming \Cref{prop:P2_analysis}. For the spatial refinement, different truncation tolerances are employed for different meshes. 
This adjustment is necessary because, for coarse meshes, the discretization error exceeds the truncation tolerance, causing the tensor spectrum to decay rapidly to the discretization-error level and then flatten. 
In this regime, the so-called flat spectrum prevents HTACA from reaching the prescribed tolerance and thus leads to non-termination. This phenomenon will be examined in future work.

\begin{figure}[htb]
    \centering
    \subfigure{
        \includegraphics[width=0.325\linewidth]{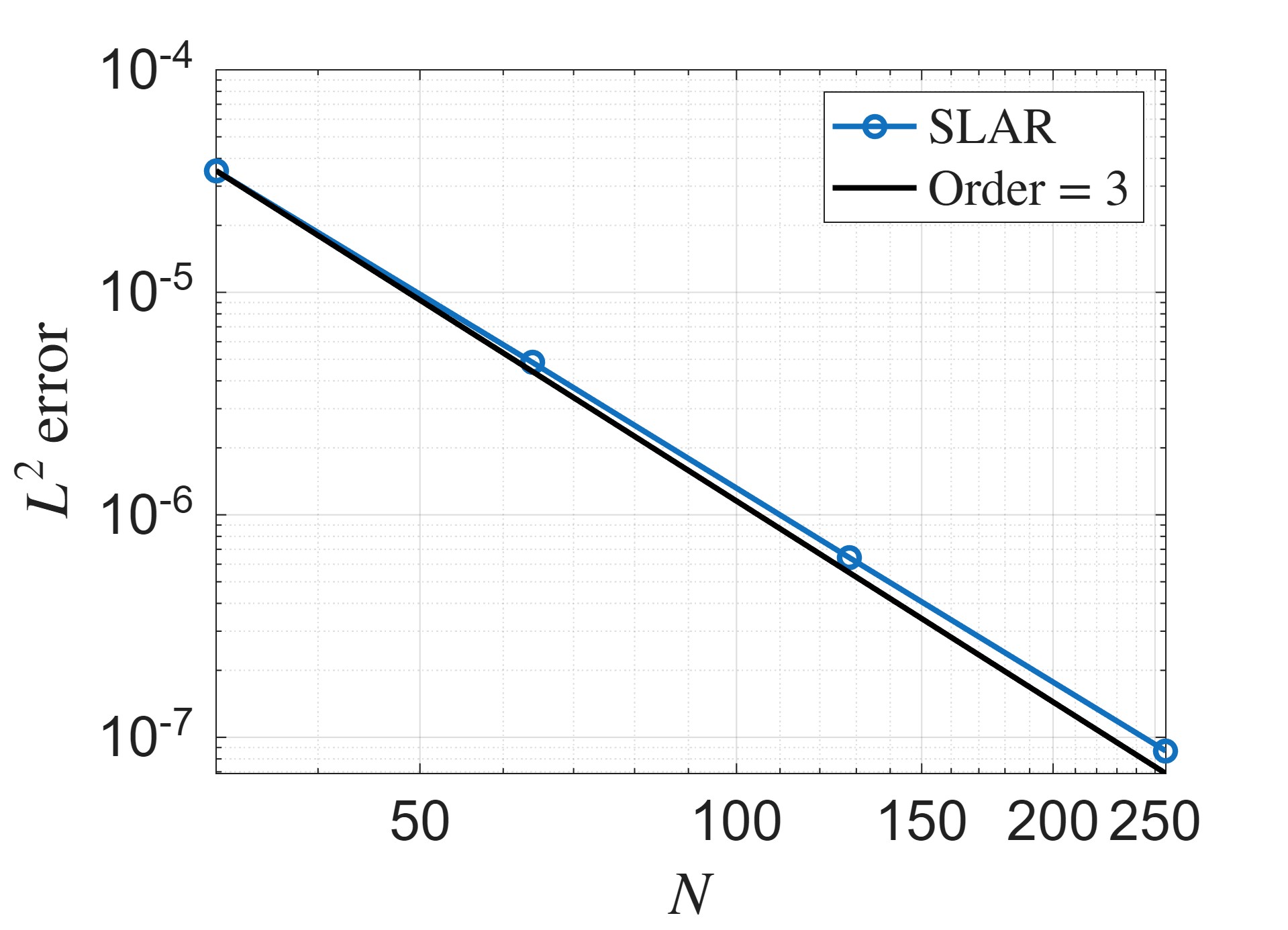}
    }
    \subfigure{
        \includegraphics[width=0.325\linewidth]{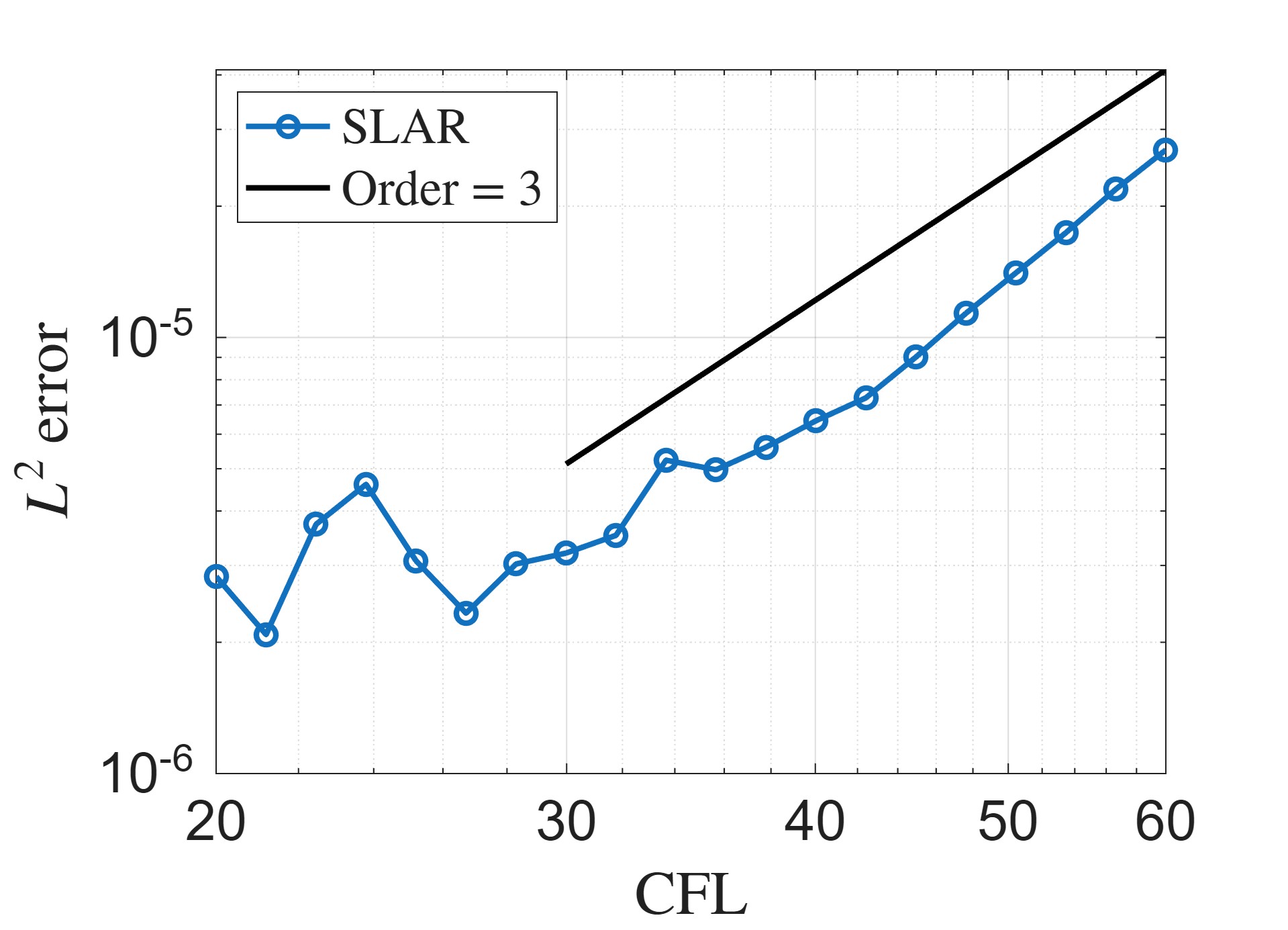}
    }
    \caption{(2D2V strong Landau damping). Left: log-log plot of grid points per dimension versus $L^2$ error at $T=0.5$ with $\mathrm{CFL}=1$ and truncation tolerances $\varepsilon_{\text{Base}}=10^{-4},10^{-5},10^{-6},10^{-7}$ for $N=32,64,128,256$, respectively. Right: log-log plot of $\mathrm{CFL}$ versus $L^2$ error at $T=0.5$ with fixed tolerance $10^{-4}$ and mesh $256^4$ for all simulations. No rank limitation is imposed in either test.}
    \label{fig:SLD_accuracy}
\end{figure}

\Cref{fig:2d2v_visualizations} presents a collection of visualizations corresponding to the strong Landau damping test case at time \( t = 35 \), illustrating the behavior of the residual-guided sampling process in a 2D2V phase space. Panel~(a) shows the contour plot of the \(\{1,2\}\)-matricization of the numerical solution (a $256^2\times 256^2$ matrix). Panels~(b) and (c) display contour plots of the first selected residual subtensors in mode sets \(\{1,2\}\) and \(\{3,4\}\), respectively, where strong localized structures are visible. Panels~(d)–(f) highlight the sampled rows and columns used in the HTACA algorithm. In particular, panel~(d) visualizes the selected index sets of the full \(\{1,2\}\)-matricization, with the red row and column indicating the first selected pivot indices. Panels~(e) and (f) depict the corresponding sampling patterns for the first selected residual subtensors in mode sets \(\{1,2\}\) and \(\{3,4\}\), respectively. The contour plots in panels~(a)--(c) illustrate the dominant coupling structures revealed by different matricizations and residual subtensors of the numerical distribution. The dominant slices \(\{1,2\}\) and \(\{3,4\}\) are accurately identified during the first recursive pivot search, as shown by the sharp features in the contour plots. The sparse selection of informative rows and columns confirms that HTACA efficiently captures low-rank structure with limited tensor access.

\begin{figure}[htb]
    \centering
    \subfigure[]{
        \includegraphics[width=0.299\linewidth]{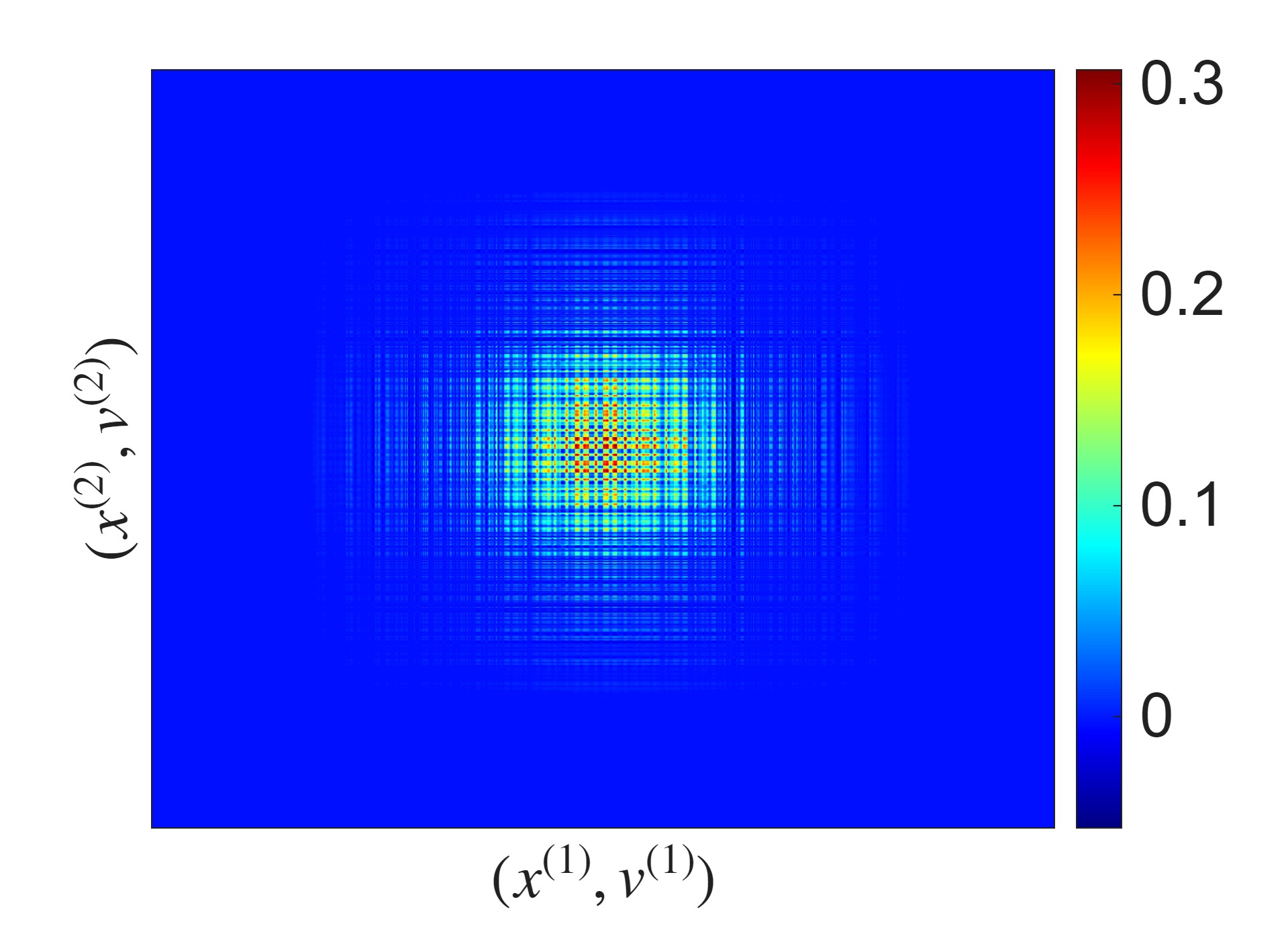}
    }
    \subfigure[]{
        \includegraphics[width=0.299\linewidth]{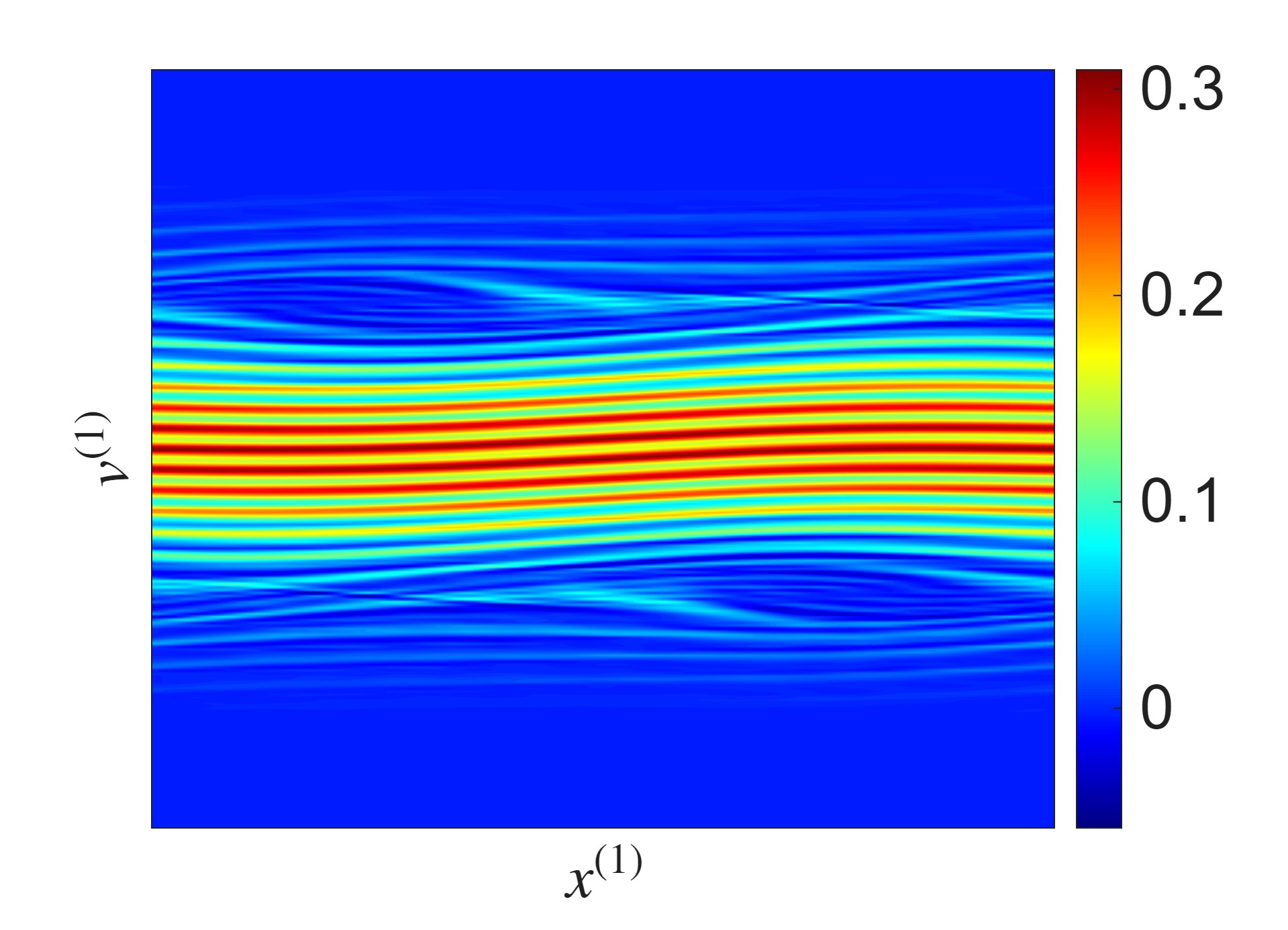}
    }
    \subfigure[]{
        \includegraphics[width=0.299\linewidth]{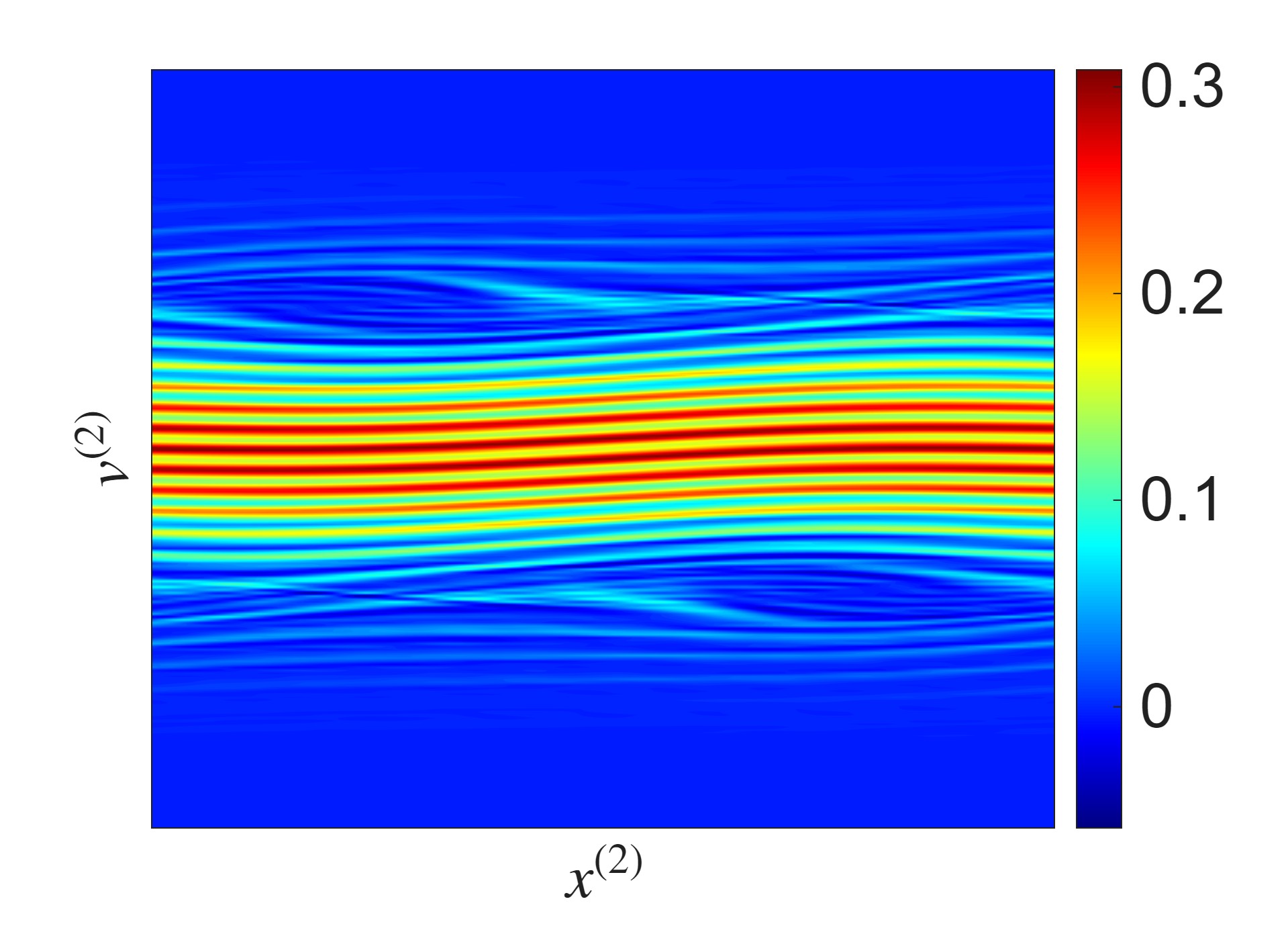}
    }

    \subfigure[]{
        \includegraphics[width=0.299\linewidth]{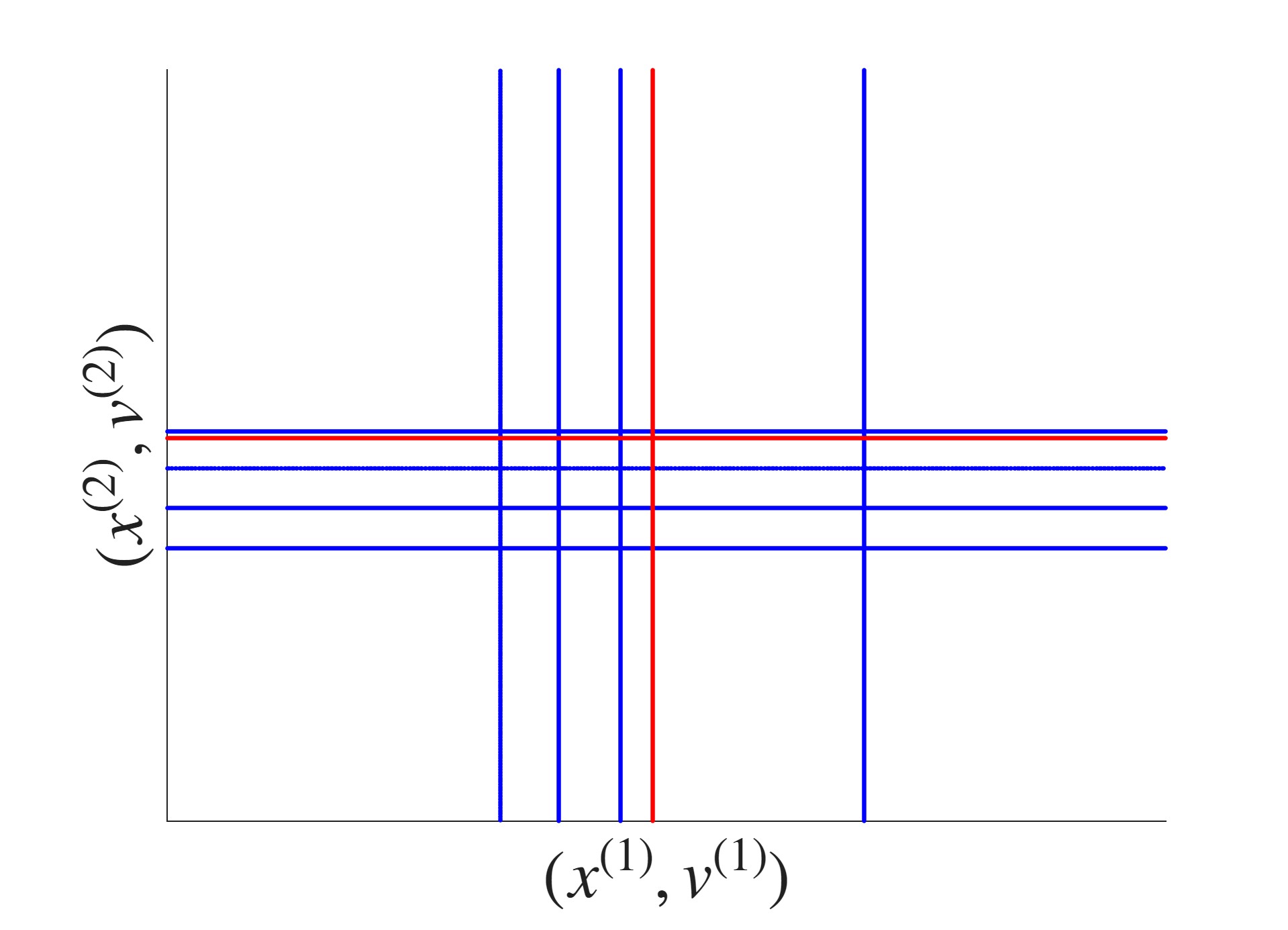}
    }
    \subfigure[]{
        \includegraphics[width=0.299\linewidth]{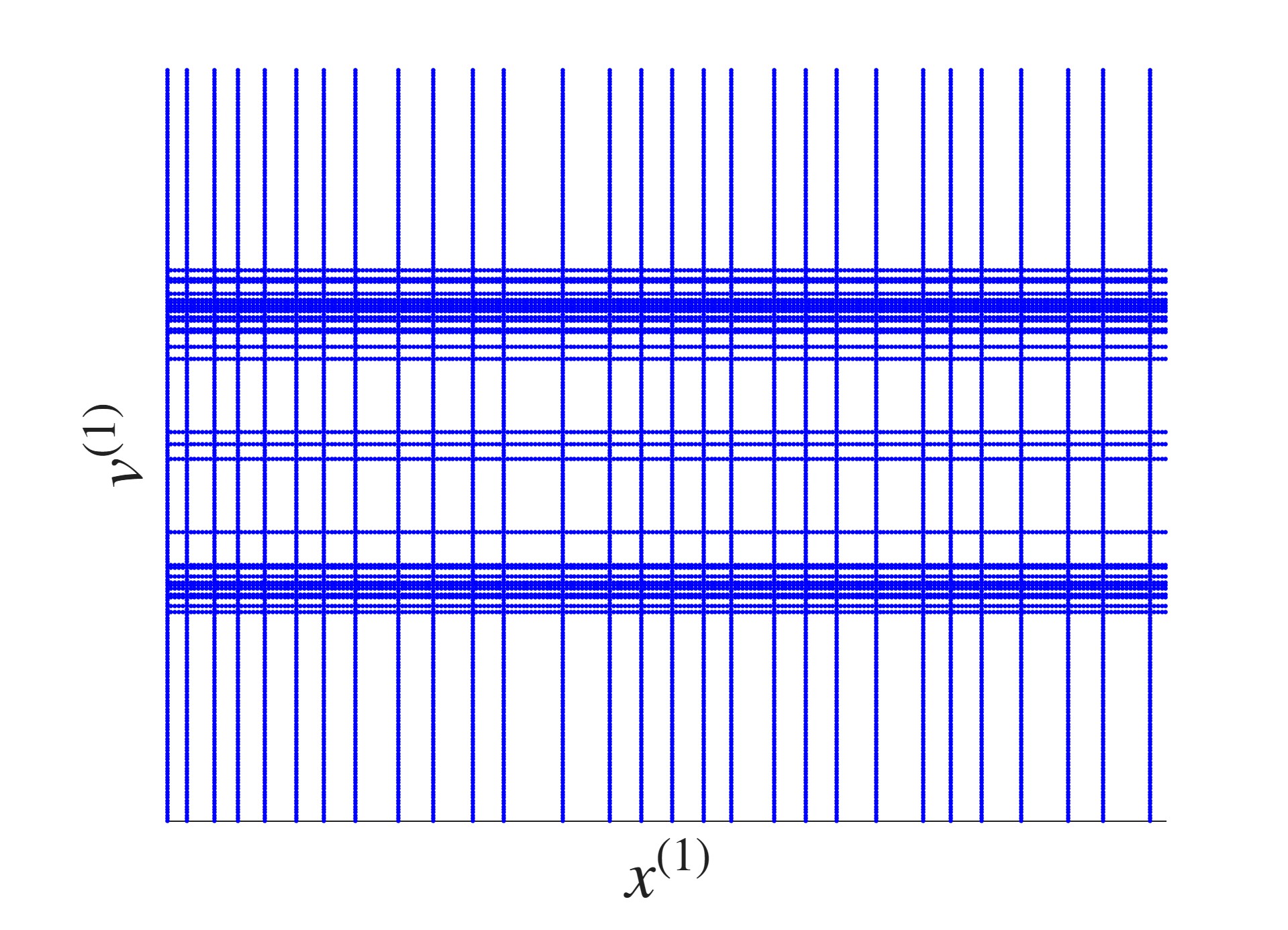}
    }
    \subfigure[]{
        \includegraphics[width=0.299\linewidth]{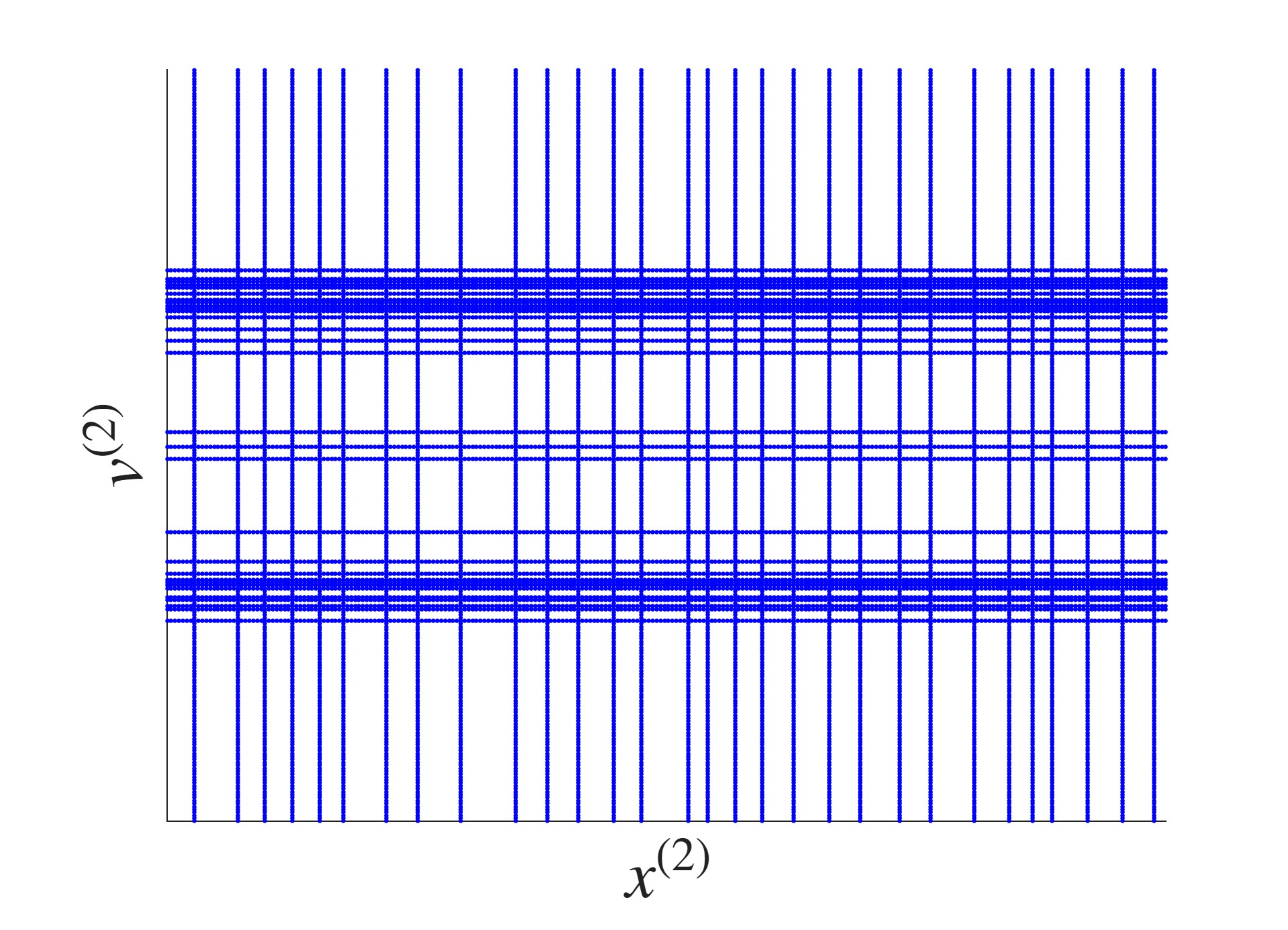}
    }
    \caption{(2D2V strong Landau damping). Contour plots (a)-(c) and selected row/column visualizations (d)-(f) of the numerical distribution at \(t = 35\).  
(a) The \(\{1,2\}\)-matricization.  
(b) The first selected \(\{1,2\}\) residual subtensor.  
(c) The first selected \(\{3,4\}\) residual subtensor.  
(d) Selected rows and columns of the \(\{1,2\}\)-matricization (the red row and column are the first selected ones).  
(e) Selected rows and columns of the first selected \(\{1,2\}\) residual subtensor.  
(f) Selected rows and columns of the first selected \(\{3,4\}\) residual subtensor.  
Simulation settings: \(v_{\max} = 2\pi\), CFL = 5, \(\varepsilon_{\text{Base}} = 5 \times 10^{-4}\), and no rank limitations.
}
    \label{fig:2d2v_visualizations}
\end{figure}

\Cref{fig:WLD_N_vs_time} illustrates the computational scaling of the SLAR method for the 3D3V weak Landau damping case using the unbalanced dimension tree in \cref{fig:unbalanced_tree_6d}. The left panel confirms that, with fixed hierarchical ranks (leaf = 3, non-leaf = 2), the simulation time per time step grows approximately linearly with the number of grid points per dimension $N$. The right panel investigates the dependence on hierarchical ranks for a fixed grid size of $64^6$. The leaf ranks are $3\times$ the \textit{rank scaling factor}, and the non-leaf ranks are $2\times$ the factor. The measured slope of $3.15$ in the log–log plot is below the theoretical scaling exponent $3 + \lceil \log_2 6 \rceil = 6$ predicted by the complexity estimate $\mathcal{O}(d^4 N r^{\,3 + \lceil \log_2 d \rceil})$. This is likely influenced by MATLAB’s built-in vectorization optimizations as observed in \cite{sands2025transport}. A detailed investigation of this effect is nontrivial and is left for future work.

\begin{figure}[htb]
    \centering
    \subfigure{
        \includegraphics[width=0.325\linewidth]{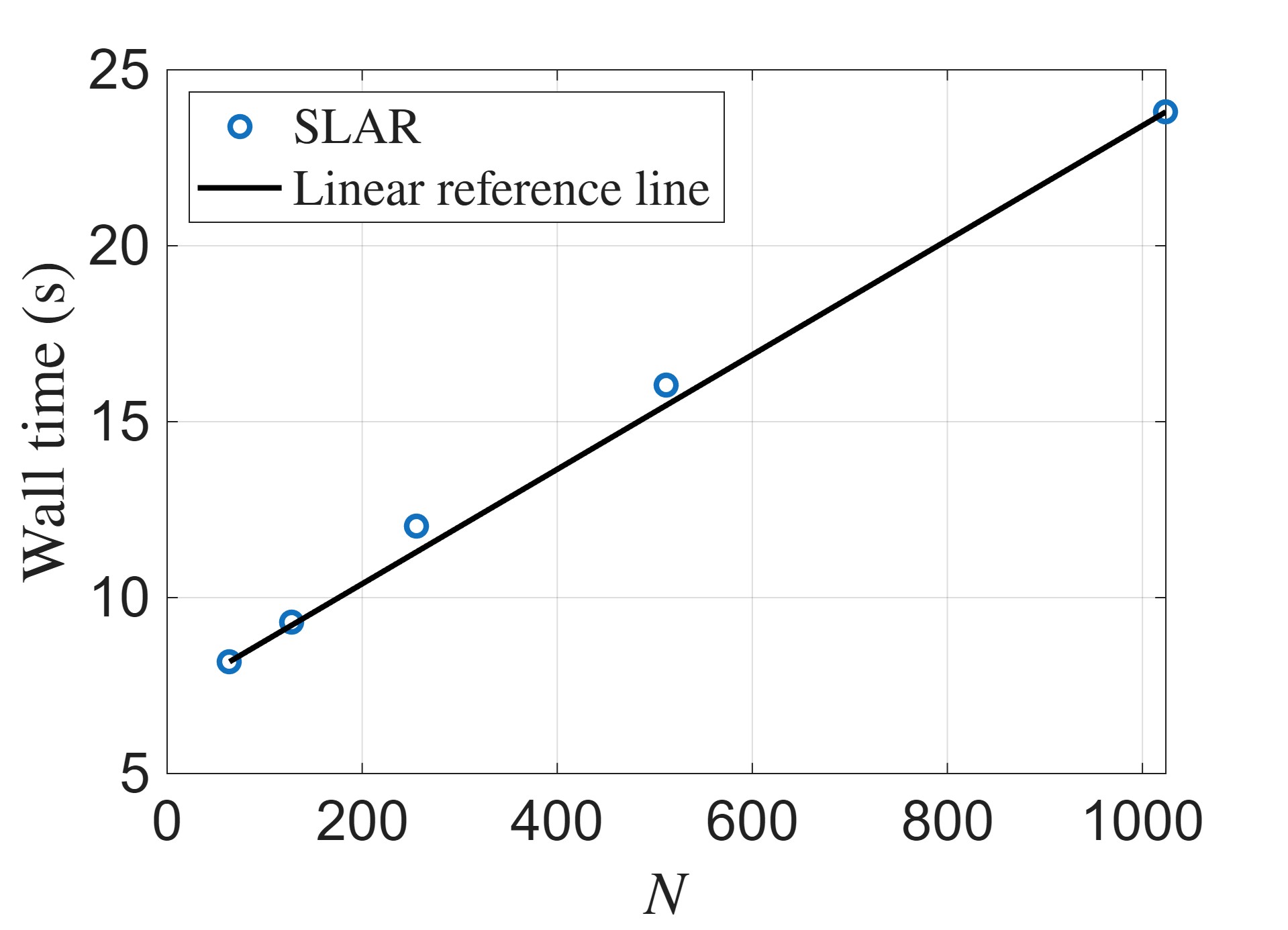}
    }
    \subfigure{
        \includegraphics[width=0.325\linewidth]{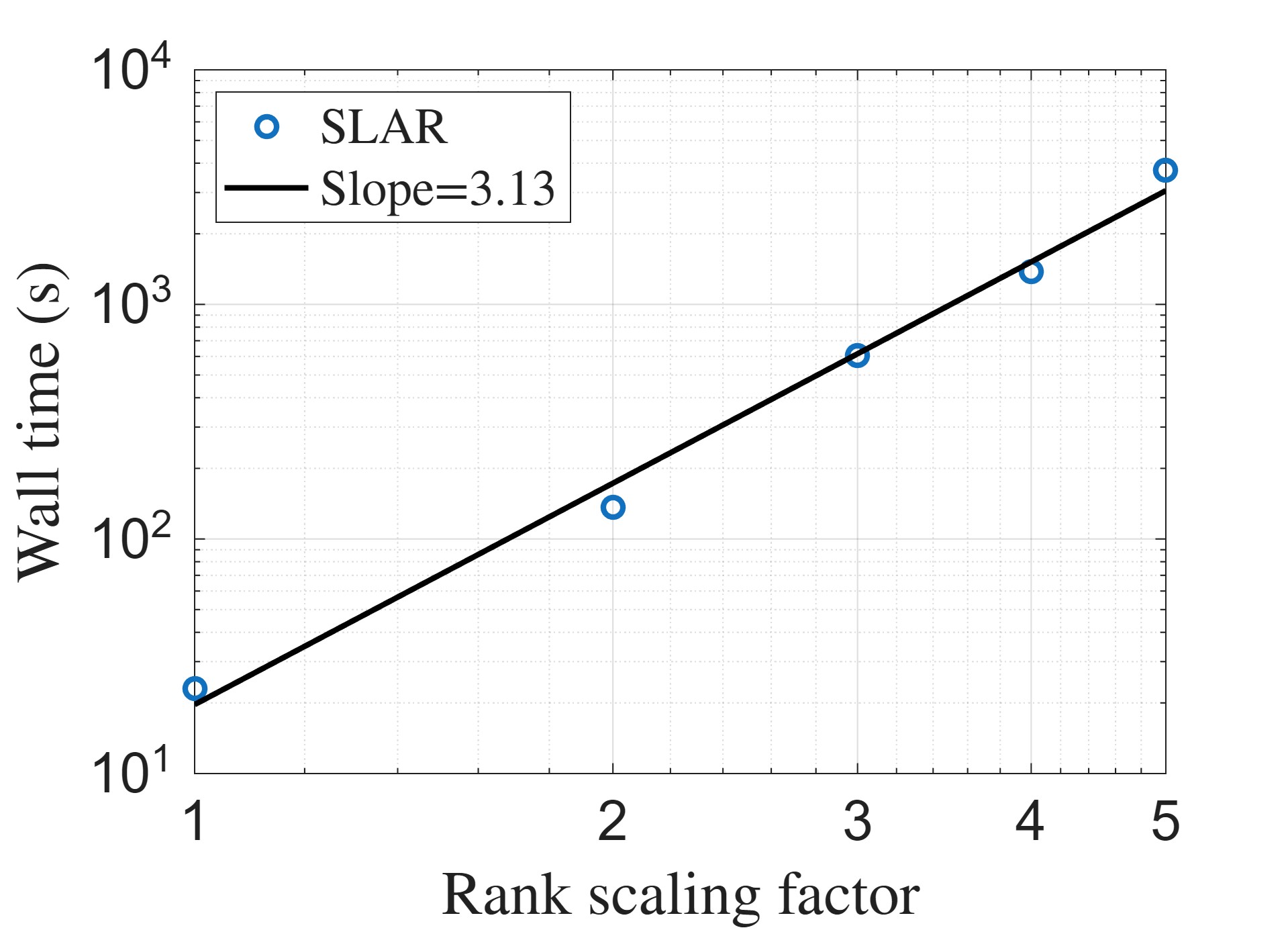}
    }
    \caption{(3D3V weak Landau damping). Left: grid points per dimension versus simulation time per time step, with a fixed hierarchical rank (leaf ranks = 3, non-leaf ranks = 2) and $N \in \{64, 128, 256, 512, 1024\}$. Right: log–log plot of rank scaling factor versus simulation time per time step, with a fixed grid size of $64^6$. Leaf ranks are $3\times$ the rank scaling factor, and non-leaf ranks are $2\times$ the factor. Shared settings: $v_{\max} = 2\pi$, $\epsilon_{\text{Base}} = 10^{-5}$.}
    \label{fig:WLD_N_vs_time}
\end{figure}

\Cref{fig:3d3v_weak_histories} presents the simulation results for the 3D3V weak Landau damping problem. Panel~(a) shows the decay of electric energy, where all resolutions produce nearly overlapping curves throughout the simulation, indicating resolution-independent behavior for this quantity under the present settings. 
Panel~(b) reports the compression ratio according to \Cref{def:compression} (lower is better). 
The compression ratio remains essentially constant over time for each resolution, with higher resolutions achieving significantly smaller ratios due to the curse of dimensionality. 
Panels~(c)--(e) display the relative deviation of mass, the deviation of momentum in $v_x$, and the relative deviation of energy, respectively. 
It should be noted that the present SLAR method does not enforce the conservation of any physical invariants, so the deviations observed in panels~(c)--(e) are purely numerical artifacts. 
Among these, mass and energy deviations remain at the order of $10^{-6}$, while momentum deviations are at the order of $10^{-4}$. In contrast to full-rank solvers, higher spatial resolution does not necessarily improve conservation here. This is because finer grids imply smaller grid spacing, and therefore, according to the time-step formula in~\eqref{eq:time_step}, smaller time steps. When the accuracy of the SLAR method is controlled by a fixed truncation tolerance~$\varepsilon_{\text{Base}}$, a larger number of time steps leads to greater cumulative truncation errors, making conservation harder to maintain.
Panel (f) provides a contour plot of the distribution slice $f(x^{(1)},v^{(1)})$ at $(x^{(2)}, v^{(2)}, x^{(3)}, v^{(3)}) = (2\pi, 0, 2\pi, 0)$ at time $t = 35$. The plot highlights the predominantly smooth, symmetric structures that persist after the collective damping of the electric field. These bands correspond to free-streaming particle motion due to a weakly nonlinear interaction and can be captured efficiently by the proposed SLAR method. We also observe some fine-scale filamentation structures, which will eventually become underresolved by the mesh. The total wall-clock times for the three simulations with resolutions $64^6$, $128^6$, and $256^6$ were 5206 s, 13406 s, and 31579 s,
respectively, corresponding to approximately 1.4, 3.7, and 8.8 hours.

\begin{figure}[htb]
    \centering
    \subfigure[]{
        \includegraphics[width=0.299\linewidth]{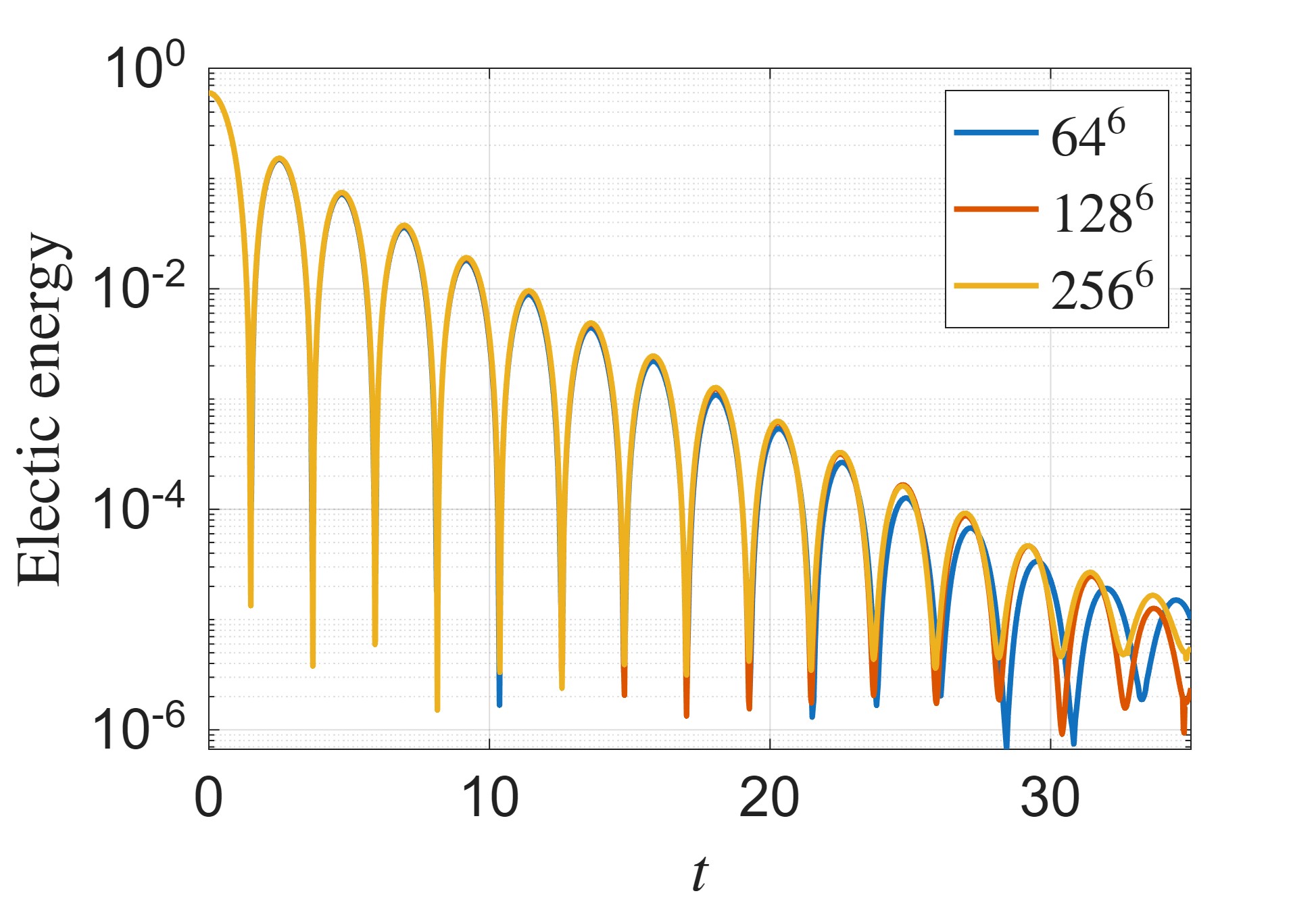}
    }
    \subfigure[]{
        \includegraphics[width=0.299\linewidth]{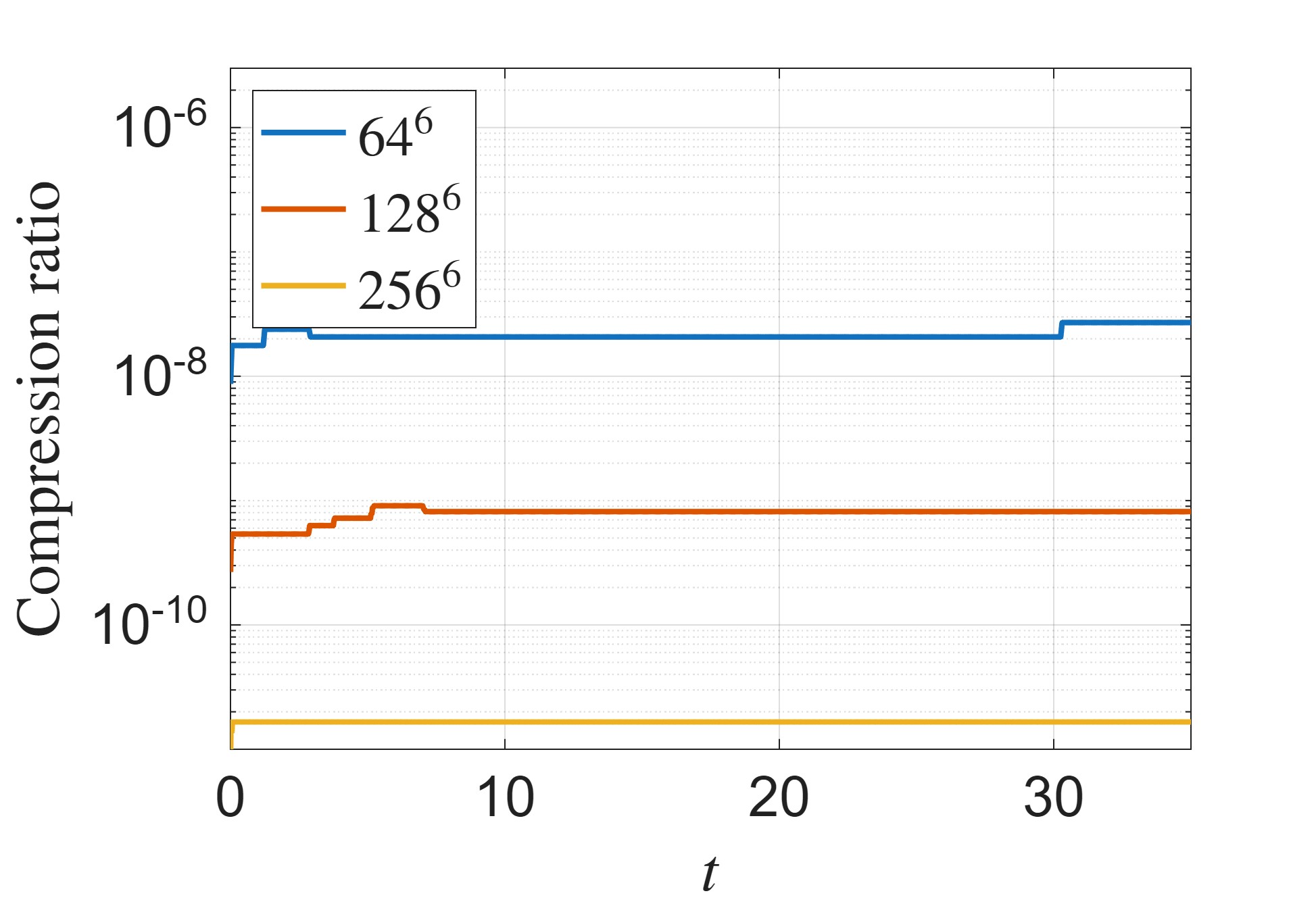}
    }
    \subfigure[]{
        \includegraphics[width=0.299\linewidth]{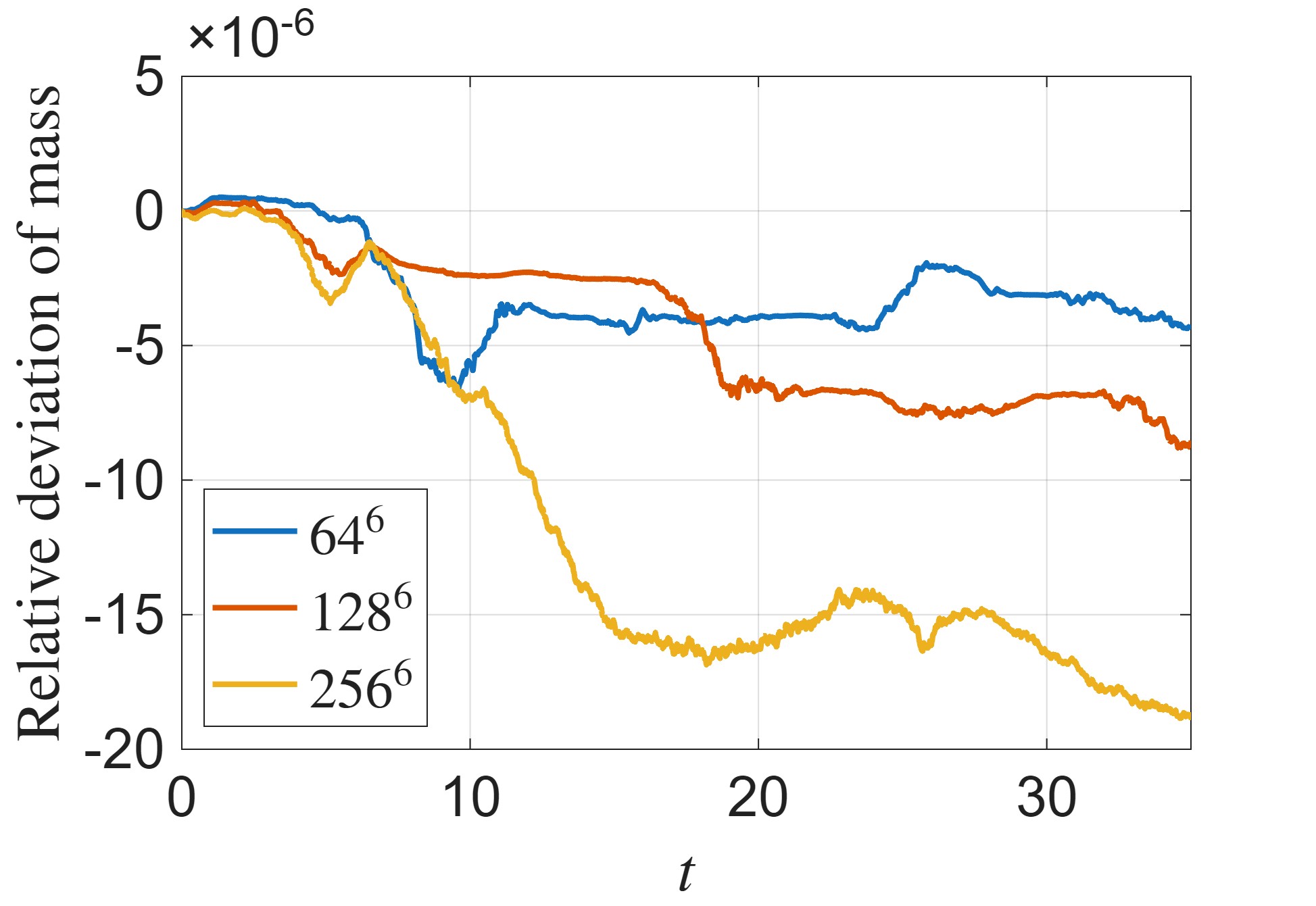}
    }

    \subfigure[]{
        \includegraphics[width=0.299\linewidth]{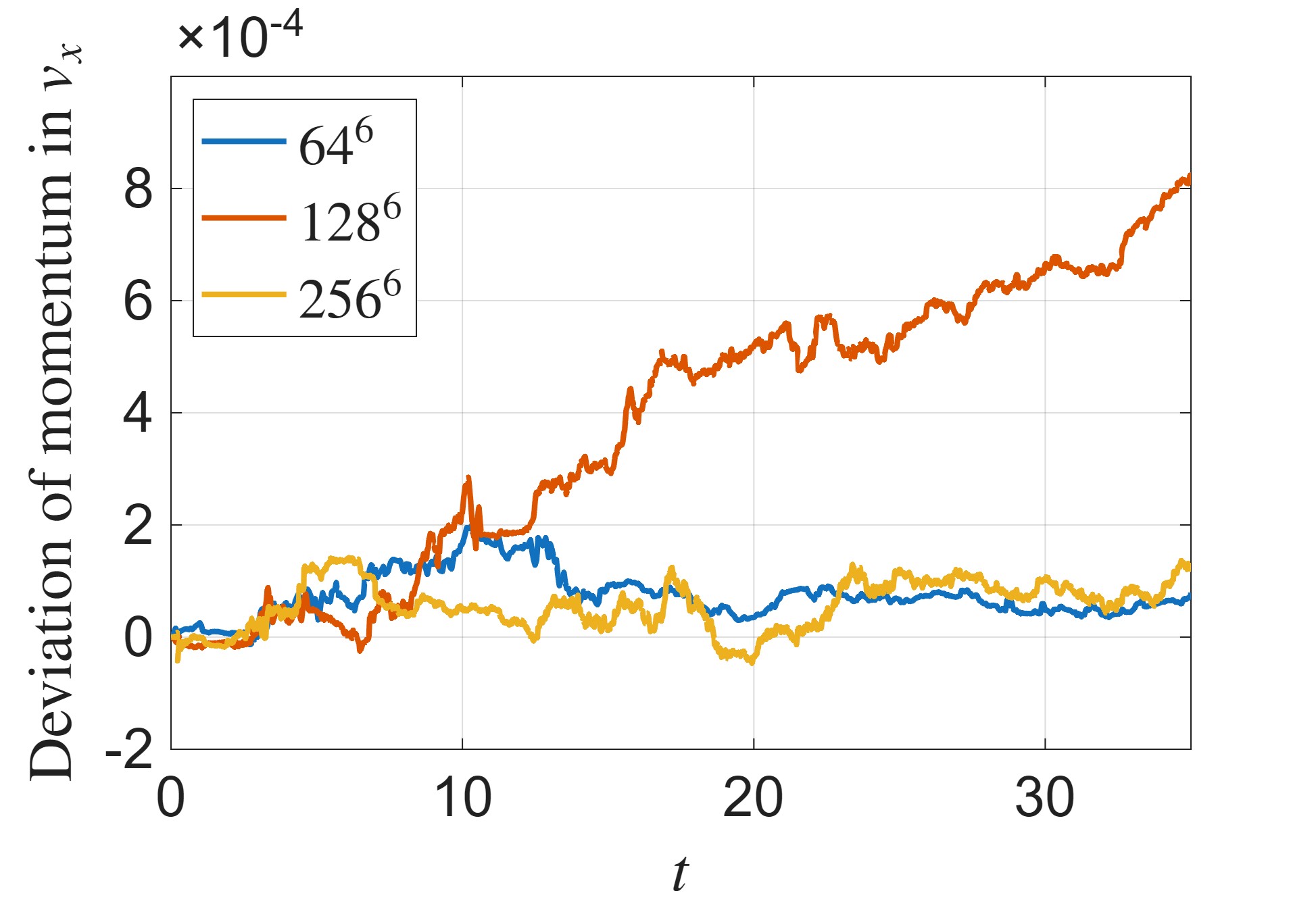}
    }
    \subfigure[]{
        \includegraphics[width=0.299\linewidth]{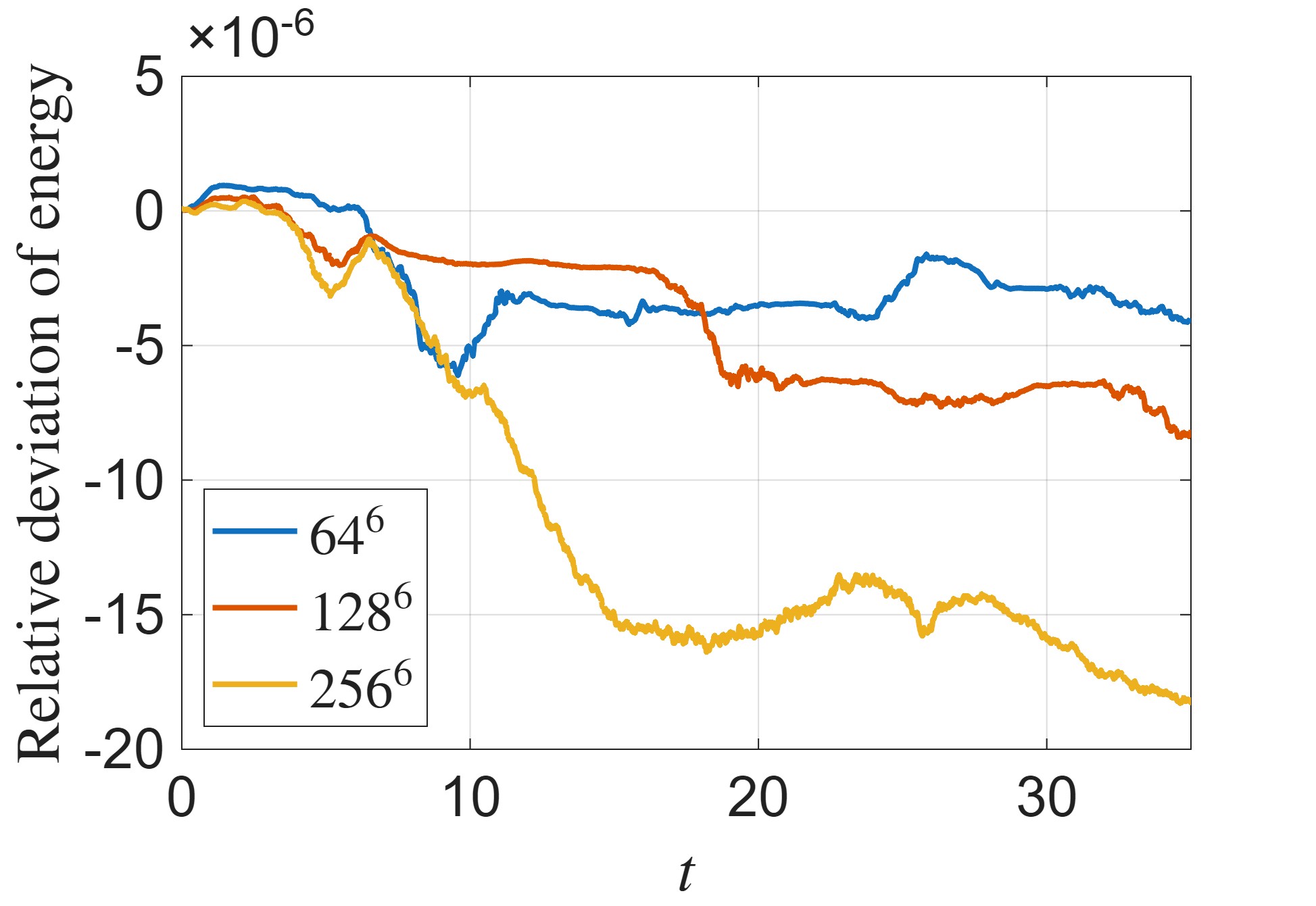}
    }
    \subfigure[]{
        \includegraphics[width=0.299\linewidth]{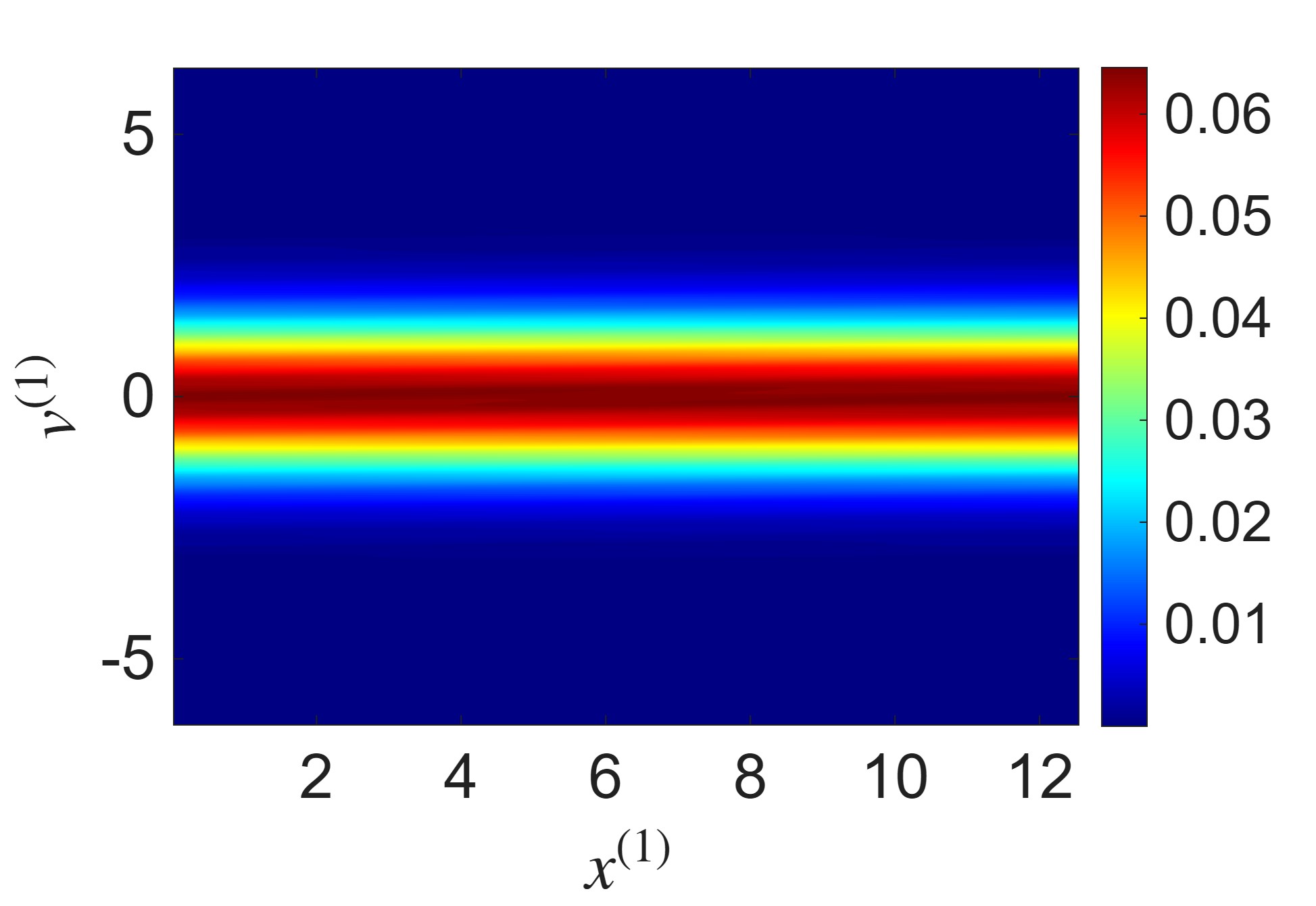}
    }
    \caption{(3D3V weak Landau damping). Selected results of the SLAR method: (a)–(e) Time histories of electric energy, compression ratio, relative deviation of mass, deviation of momentum in $v_x$, and relative deviation of energy, for resolutions $64^6$, $128^6$, and $256^6$. (f) Contour plot of $f(x^{(1)}, v^{(1)})$ on the slice $(x^{(2)}, v^{(2)},x^{(3)}, v^{(3)})=(2\pi,0,2\pi,0)$ at $t=35$ for resolution $256^6$. Simulation settings: $v_{\text{max}}=2\pi$, CFL = 5, $\epsilon_{\text{Base}} = 10^{-5}$, and no rank limitations.}
    \label{fig:3d3v_weak_histories}
\end{figure}

Similarly, \Cref{fig:3d3v_strong_histories} presents the simulation results for the 3D3V strong Landau damping problem. 
Panel~(a) shows the temporal evolution of the electric energy. Compared with the weak damping case, the curves for different spatial resolutions start to diverge much earlier, with the coarsest $64^6$ grid showing visible deviations for $t \gtrsim 22$. 
Panel~(b) displays the compression ratio. Unlike the weak damping case, the growth of the hierarchical ranks reflects the presence of stronger nonlinear effects which generate fine-scale structures. This results in a noticeable increase in the storage cost, which, in turn, increases the compression ratio of the scheme. 
Panels~(c)--(e) show the relative deviations of mass, the deviation of momentum in $v_x$, and the relative deviation of energy, respectively. The deviation magnitudes are substantially larger than those in the weak damping case: mass and energy deviations increase from $\mathcal{O}(10^{-6})$ to $\mathcal{O}(10^{-3})$, while momentum deviations grow from $\mathcal{O}(10^{-4})$ to $\mathcal{O}(10^{-1})$. Higher spatial resolution does not yield significant improvement in preserving these invariants, which is consistent with the observation in the weak damping case. The poor conservation performance observed here highlights a limitation of the current SLAR framework in handling strongly nonlinear problems. In future work, we will address this issue by generalizing the LoMaC framework proposed in~\cite{sands2026adaptive} to 3D3V problems.
Panel (f) provides a contour plot of the distribution slice $f(x^{(1)},v^{(1)})$ at $(x^{(2)}, v^{(2)}, x^{(3)}, v^{(3)}) = (2\pi, 0, 2\pi, 0)$ at time $t = 35$. The plot captures the fine-scale filamentation structures generated by the residual perturbation after the collective damping of the electric field. These filaments, which are more pronounced than the weak case, correspond to free-streaming particle dynamics that are continually stretched and deformed by nonlinear phase mixing and will cascade to progressively smaller scales. The resulting dynamics drive substantial rank growth, which increases the computational complexity of the adaptive-rank representation. The wall-clock times for the three simulations with resolutions
$64^6$, $128^6$, and $256^6$ were 31037 s, 62767 s, and 137438 s, respectively, corresponding to approximately 8.6, 17.4, and 38.2 hours.

\begin{figure}[htb]
    \centering
    \subfigure[]{
        \includegraphics[width=0.299\linewidth]{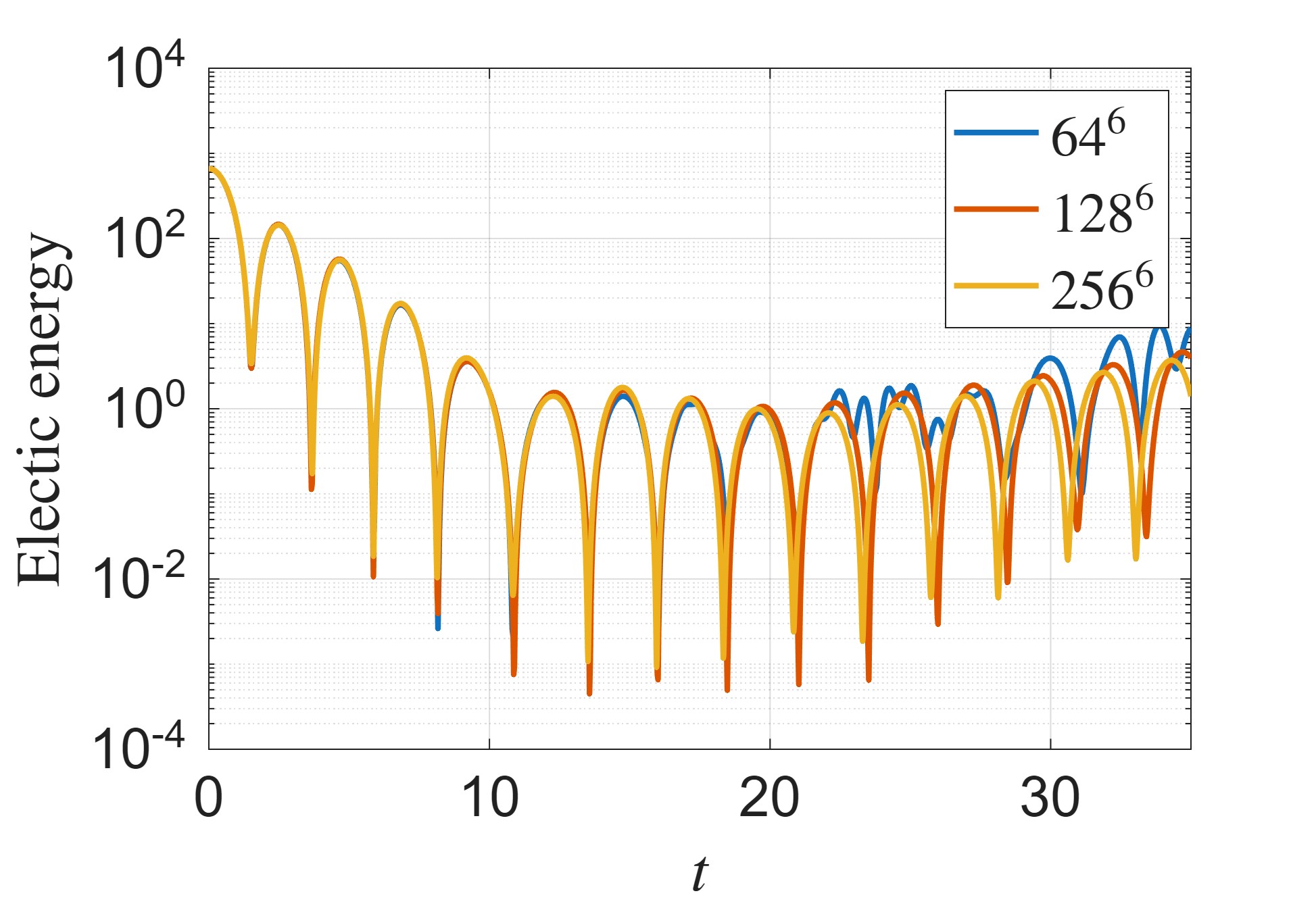}
    }
    \subfigure[]{
        \includegraphics[width=0.299\linewidth]{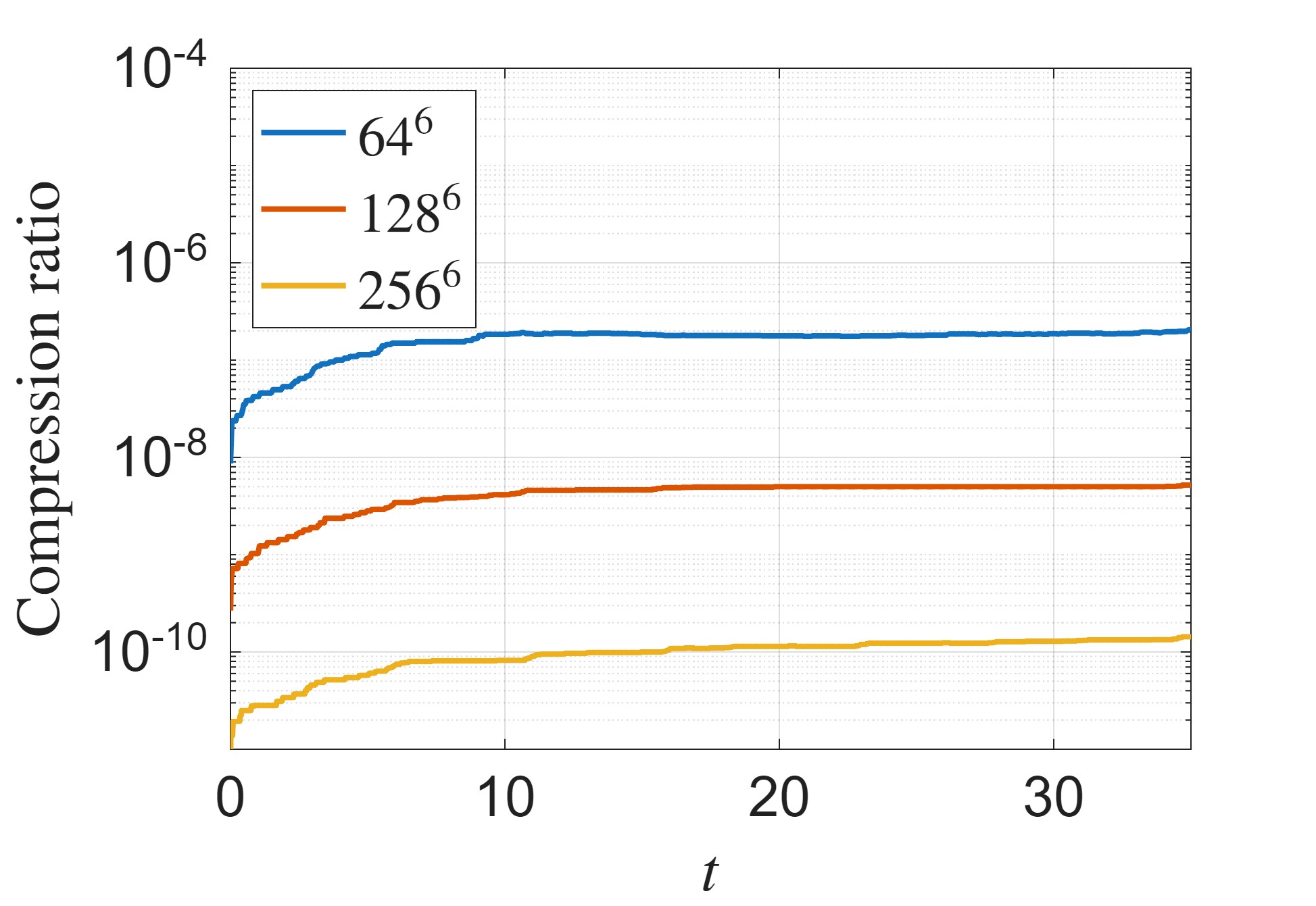}
    }
    \subfigure[]{
        \includegraphics[width=0.299\linewidth]{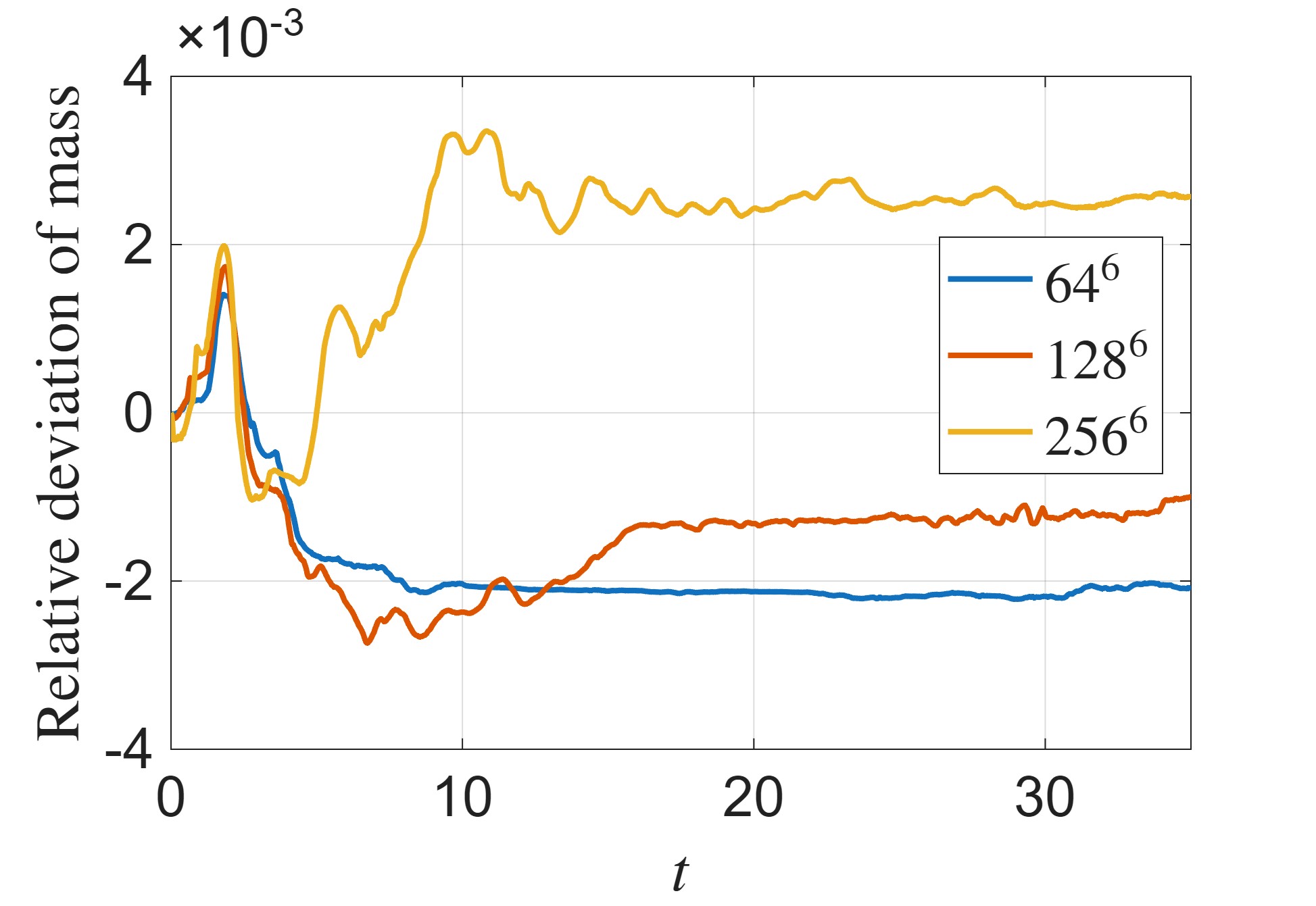}
    }

    \subfigure[]{
        \includegraphics[width=0.299\linewidth]{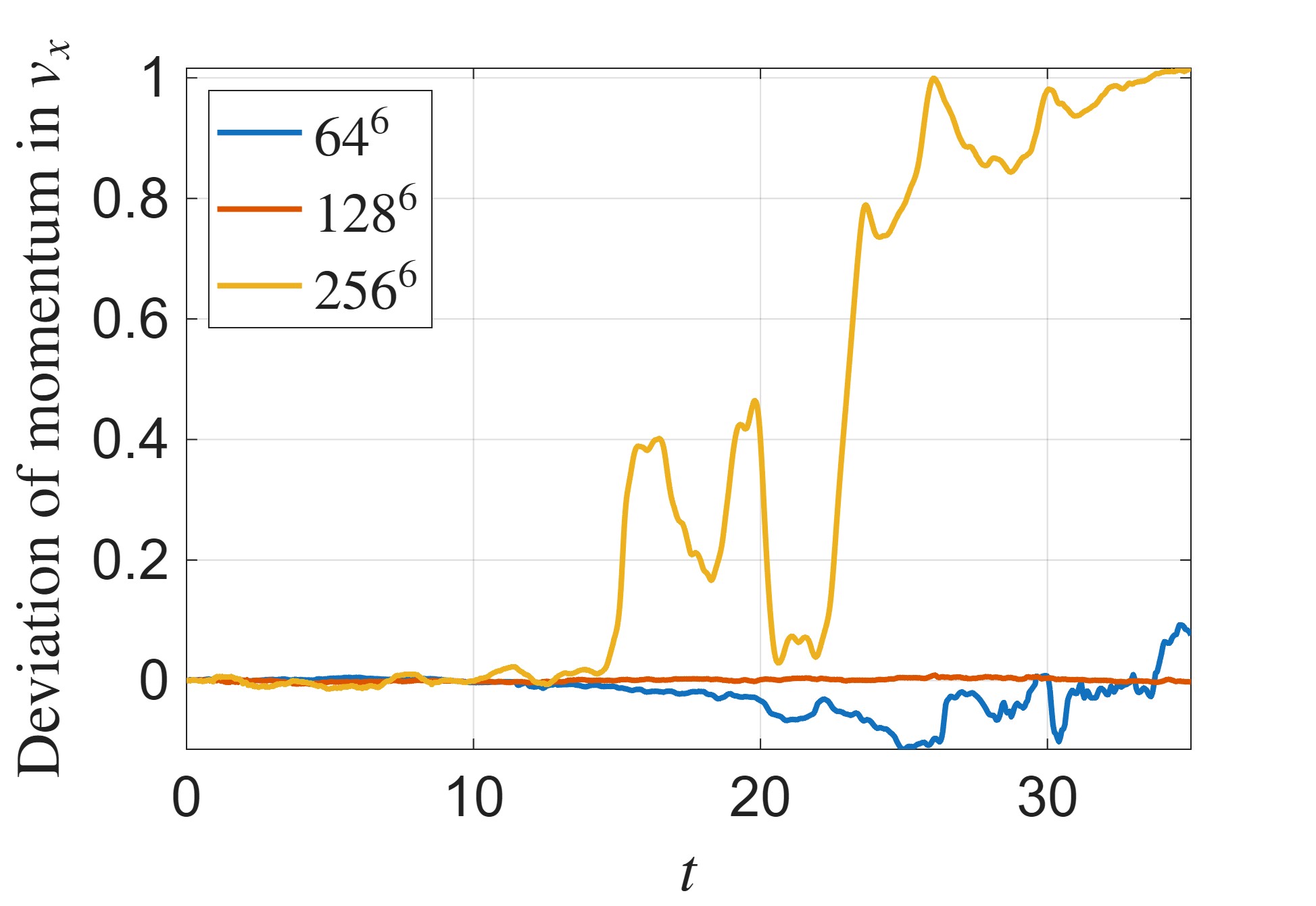}
    }
    \subfigure[]{
        \includegraphics[width=0.299\linewidth]{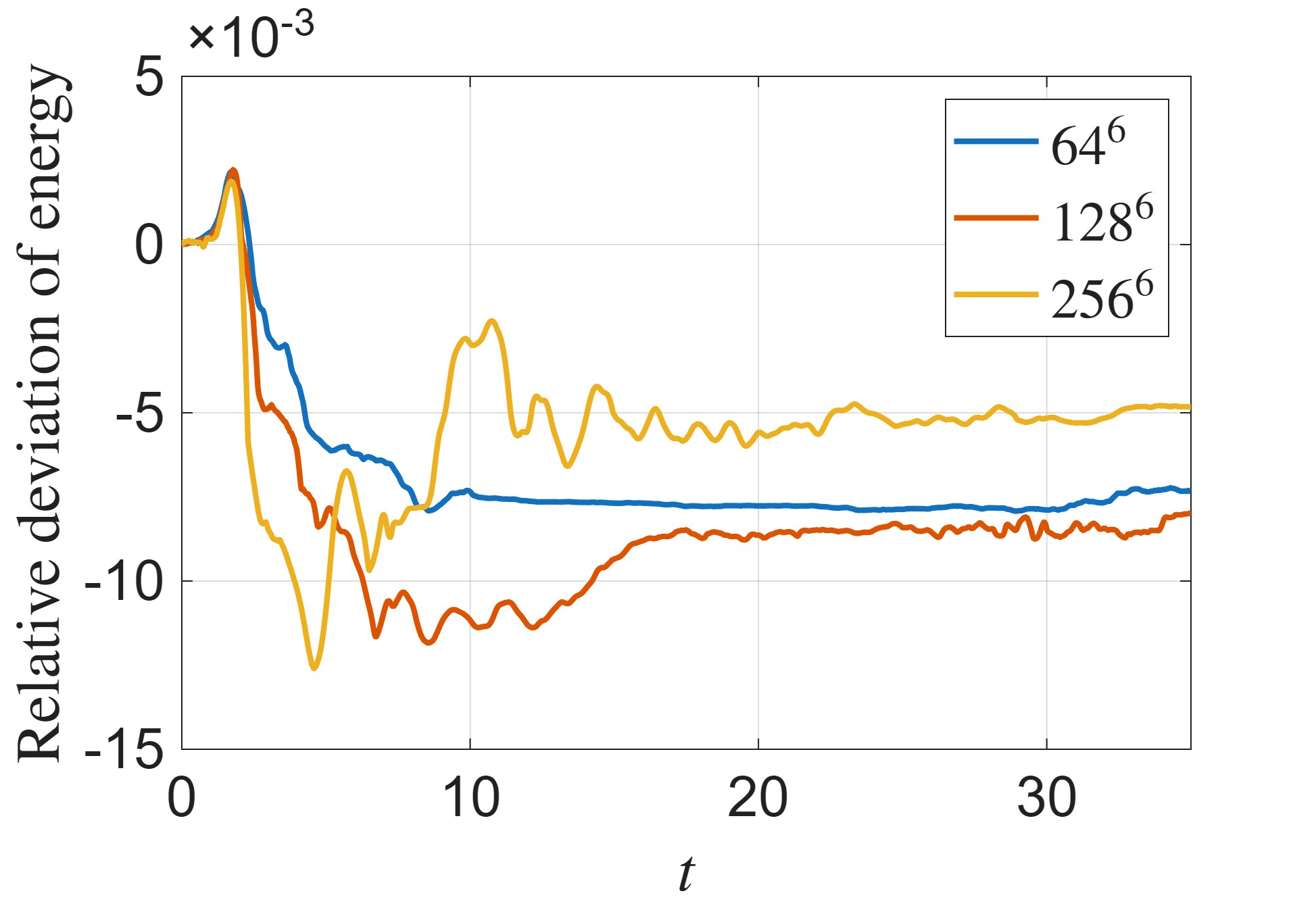}
    }
    \subfigure[]{
        \includegraphics[width=0.299\linewidth]{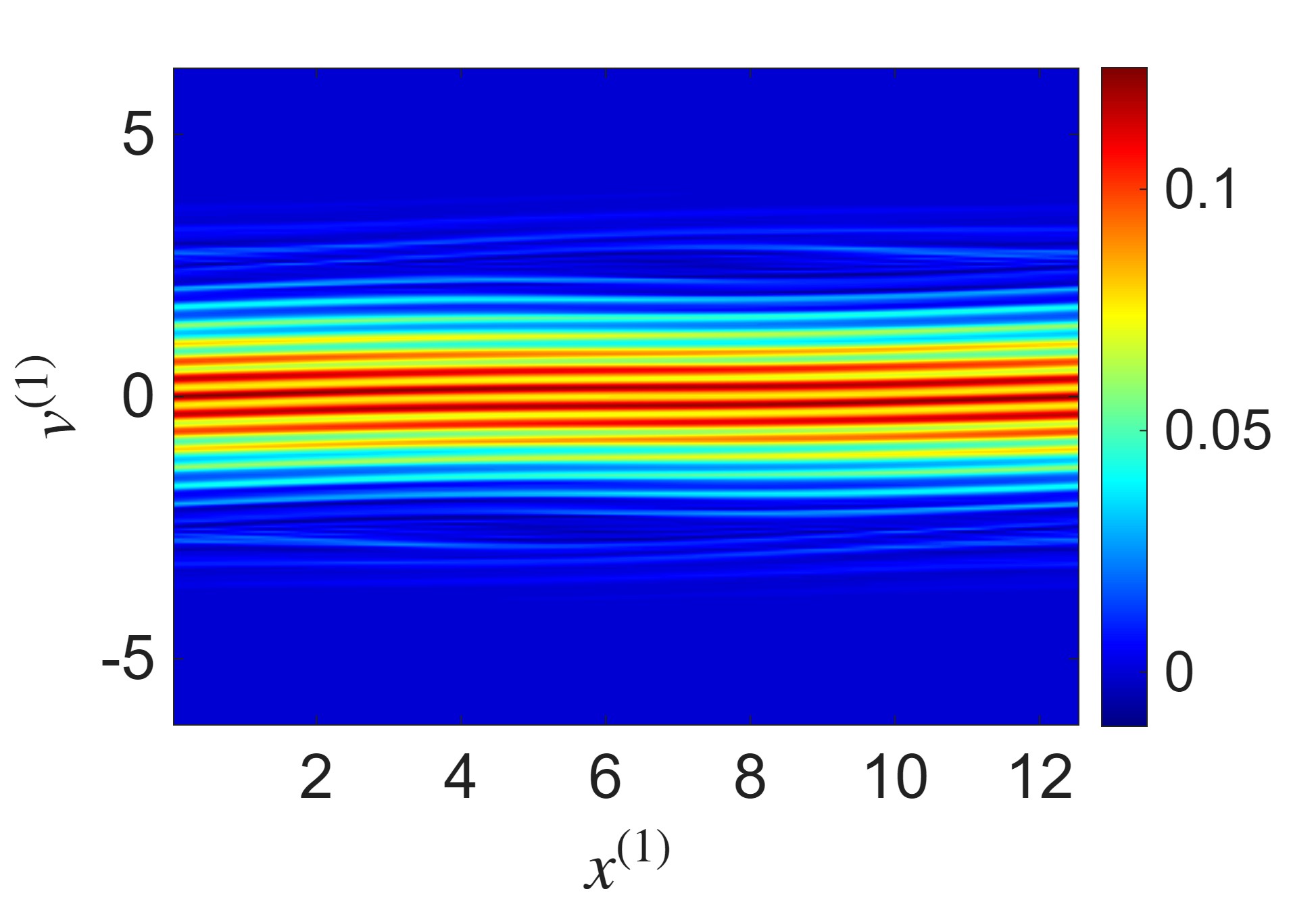}
    }
    \caption{(3D3V strong Landau damping). Selected results of the SLAR method: (a)–(e) Time histories of electric energy, compression ratio, relative deviation of mass, deviation of momentum in $v_x$, and relative deviation of energy, for resolutions $64^6$, $128^6$, and $256^6$. (f) Contour plot of $f(x^{(1)}, v^{(1)})$ on the slice $(x^{(2)}, v^{(2)}, x^{(3)}, v^{(3)})=(2\pi,0,2\pi,0)$ at $t=35$ for resolution $256^6$. Simulation settings: mesh $128^6$, $v_{\text{max}}=2\pi$, CFL = 5, $\epsilon_{\text{Base}} = 5\times10^{-4}$, and no rank limitations.}
    \label{fig:3d3v_strong_histories}
\end{figure}

\end{example}

\begin{example}[Two-stream instability]
Consider the 3D3V VP system with the initial condition $f(\bm{x},\bm{v},0)$
\[
    \frac{1}{2^{d_x}(2\pi)^{d_x/2}}
    \left(1 + \alpha \sum_{\mu=1}^{d_x} \cos\!\big(k x^{(\mu)}\big)\right) 
    \prod_{\mu=1}^{d_v}
    \Bigl[
        \exp\!\left(-\tfrac{(v^{(\mu)}-v_0)^2}{2}\right)
        + \exp\!\left(-\tfrac{(v^{(\mu)}+v_0)^2}{2}\right)
    \Bigr],
\]
with $\alpha = 0.001$, $v_0 = 2.4$, and $k = 0.2$.  
This describes two counter-propagating populations of warm electrons drifting with velocities $\pm v_0$ in each velocity coordinate. A small spatial perturbation is applied to initiate the instability.

Figure~\ref{fig:TSI_tree_comp} shows the rank histories for the 2D2V two-stream
instability with the two dimension orderings considered above. The observation
is similar to the Landau damping case. In particular, the ordering
\(\{x^{(1)},x^{(2)},v^{(1)},v^{(2)}\}\) leads to more pronounced rank growth,
whereas the same-dimension pairing
\(\{x^{(1)},v^{(1)},x^{(2)},v^{(2)}\}\) keeps all hierarchical ranks at lower
levels throughout the simulation.

\begin{figure}[htb]
    \centering
    \subfigure{
        \includegraphics[width=0.325\linewidth]{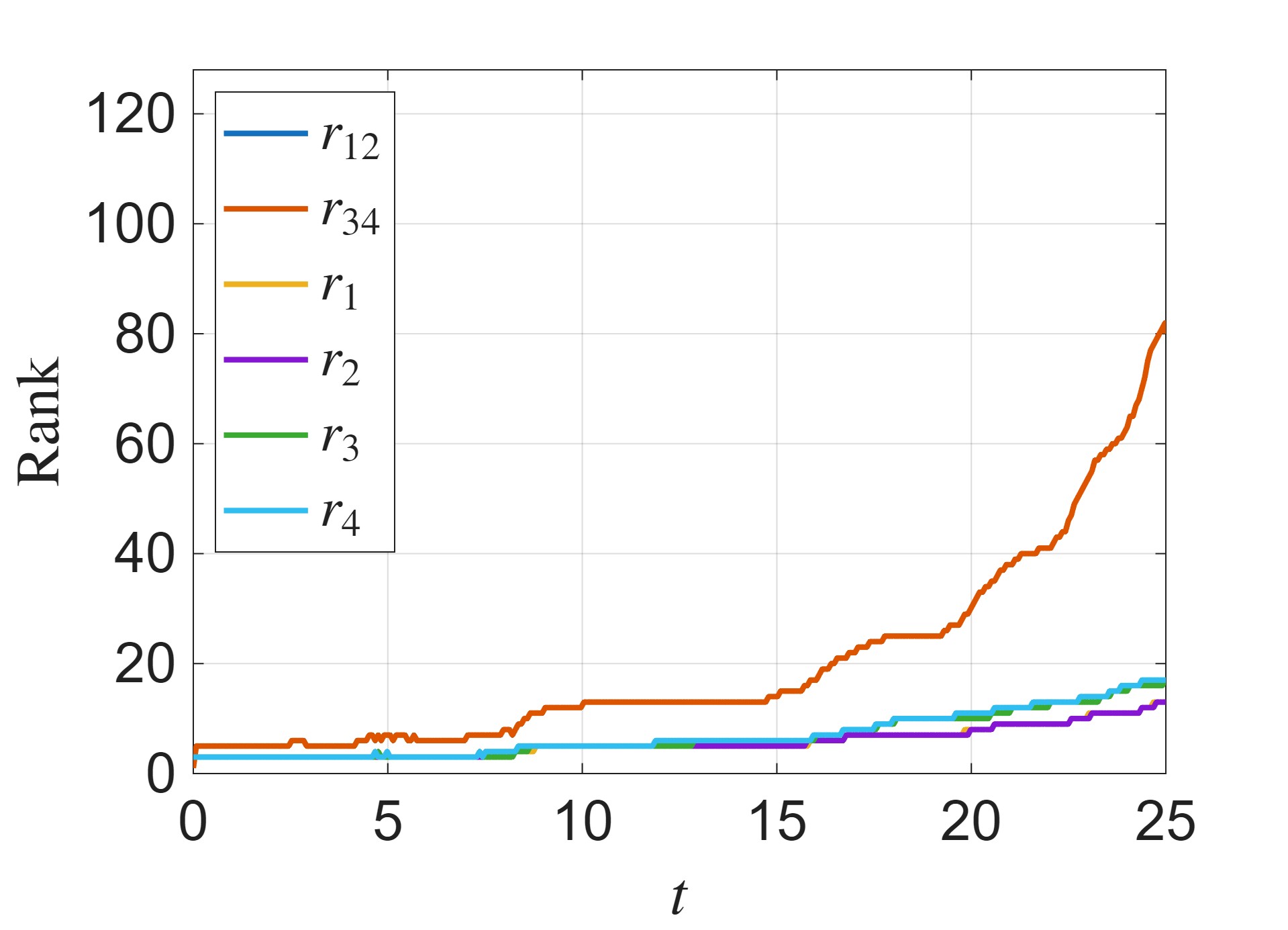}
    }
    \subfigure{
        \includegraphics[width=0.325\linewidth]{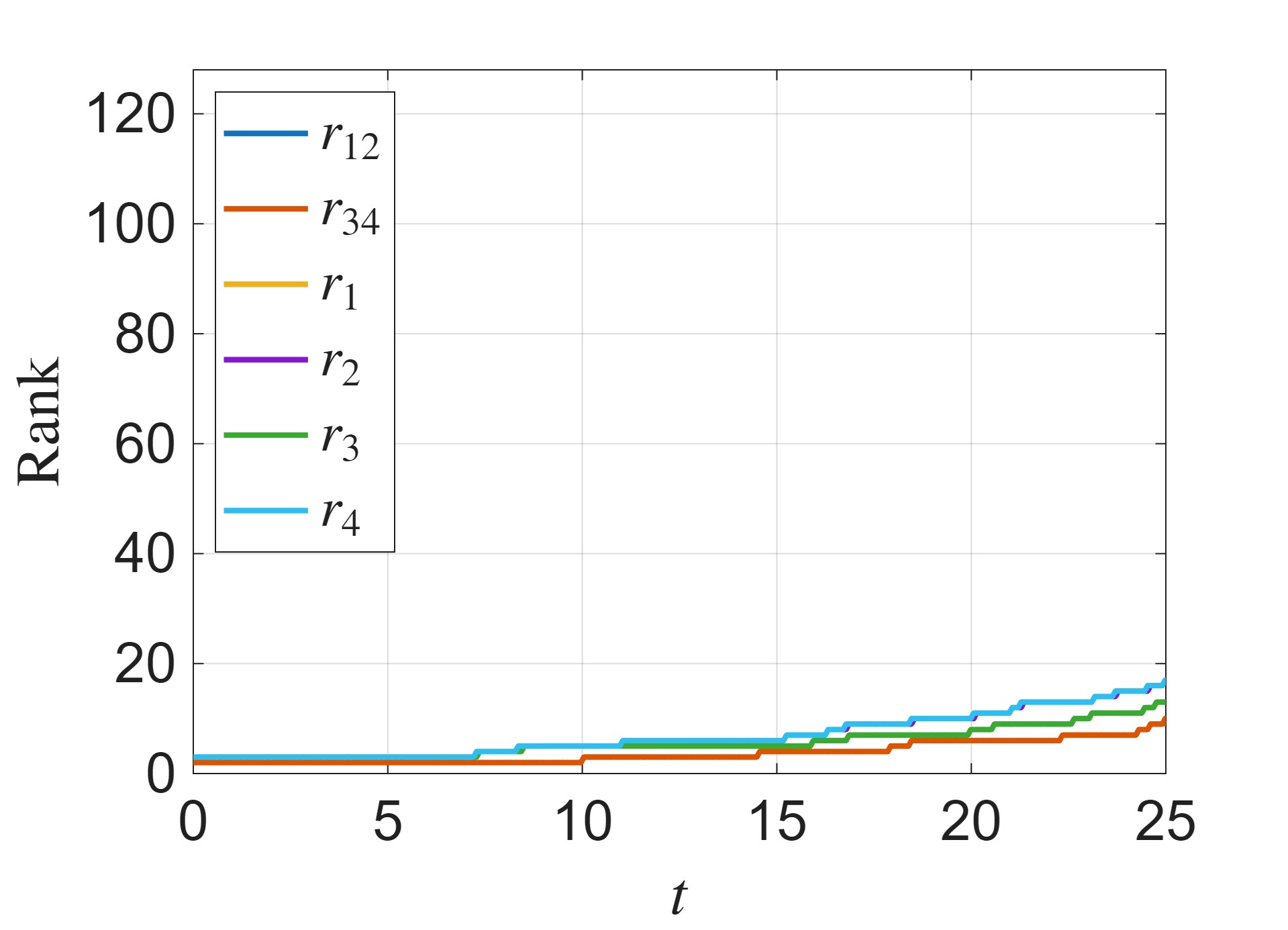}
    }
    \caption{(2D2V two-stream instability). Rank history of the SLAR method with the dimension tree in \cref{fig:balanced_tree_4d}: Left $-$ modes ordered as $\{x^{(1)}, x^{(2)}, v^{(1)}, v^{(2)}\}$; Right $-$ modes ordered as \{$x^{(1)}$, $v^{(1)}$, $x^{(2)}$, $v^{(2)}$\}. Simulation settings: $v_{\max}=8$, mesh $128^4$, CFL $=5$, $\epsilon_{\text{Base}} = 10^{-6}$, with minimum leaf ranks fixed at 3 and unrestricted non-leaf ranks.}
    \label{fig:TSI_tree_comp}
\end{figure}

\Cref{fig:3d3v_two_stream_histories} presents the simulation results for the 3D3V two-stream instability problem. Panel~(a) shows the expected growth in electric energy during the linear phase of the instability and is followed by the onset of nonlinear saturation around $t \approx 35$. The results from all three resolutions ($64^6$, $128^6$, and $256^6$) are nearly indistinguishable, indicating that the bulk electric energy dynamics are not sensitive to spatial resolution in this case. 
Panel~(b) plots the compression ratio (lower is better), which remains relatively steady during the linear stage, but increases significantly once nonlinear structures develop due to rank growth. Higher resolutions maintain lower compression ratios throughout the simulation due to the curse of dimensionality, and all cases display step-wise increases corresponding to the addition of new rank components in the adaptive representation. 
Panels~(c)–(e) show the deviations in mass, momentum, and total energy. For $t \lesssim 25$, the relative mass and energy deviations stay within $10^{-6}$, but both grow rapidly to the order of $10^{-3}$ in the later nonlinear stage ($t > 30$). The momentum deviations in $v_x$ [panel~(d)] are negligible in the early stage but increase noticeably after $t \approx 20$.
Panel (f) provides a contour plot of the distribution slice $f(x^{(1)},v^{(1)})$ at $(x^{(2)}, v^{(2)}, x^{(3)}, v^{(3)}) = (0,0,0,0)$ at time $t = 35$. The plot displays the characteristic vortex structures of the two-stream instability during the nonlinear saturation phase, arising from electrons being trapped by the electric field. The separatrix, which delineates the boundary of this vortex, is also visible and marks the transition between regions of trapped and untrapped particles. These features demonstrate the method's capability to resolve fine-scale structures in the phase space while adaptively adjusting the rank distribution across the hierarchical dimensions. The wall-clock times for the three simulations with resolutions
$64^6$, $128^6$, and $256^6$ were 54705 s, 171830 s, and 323545 s, respectively, corresponding to approximately 15.2, 47.7, and 89.9 hours.

\begin{figure}[htb]
    \centering
    \subfigure[]{
        \includegraphics[width=0.299\linewidth]{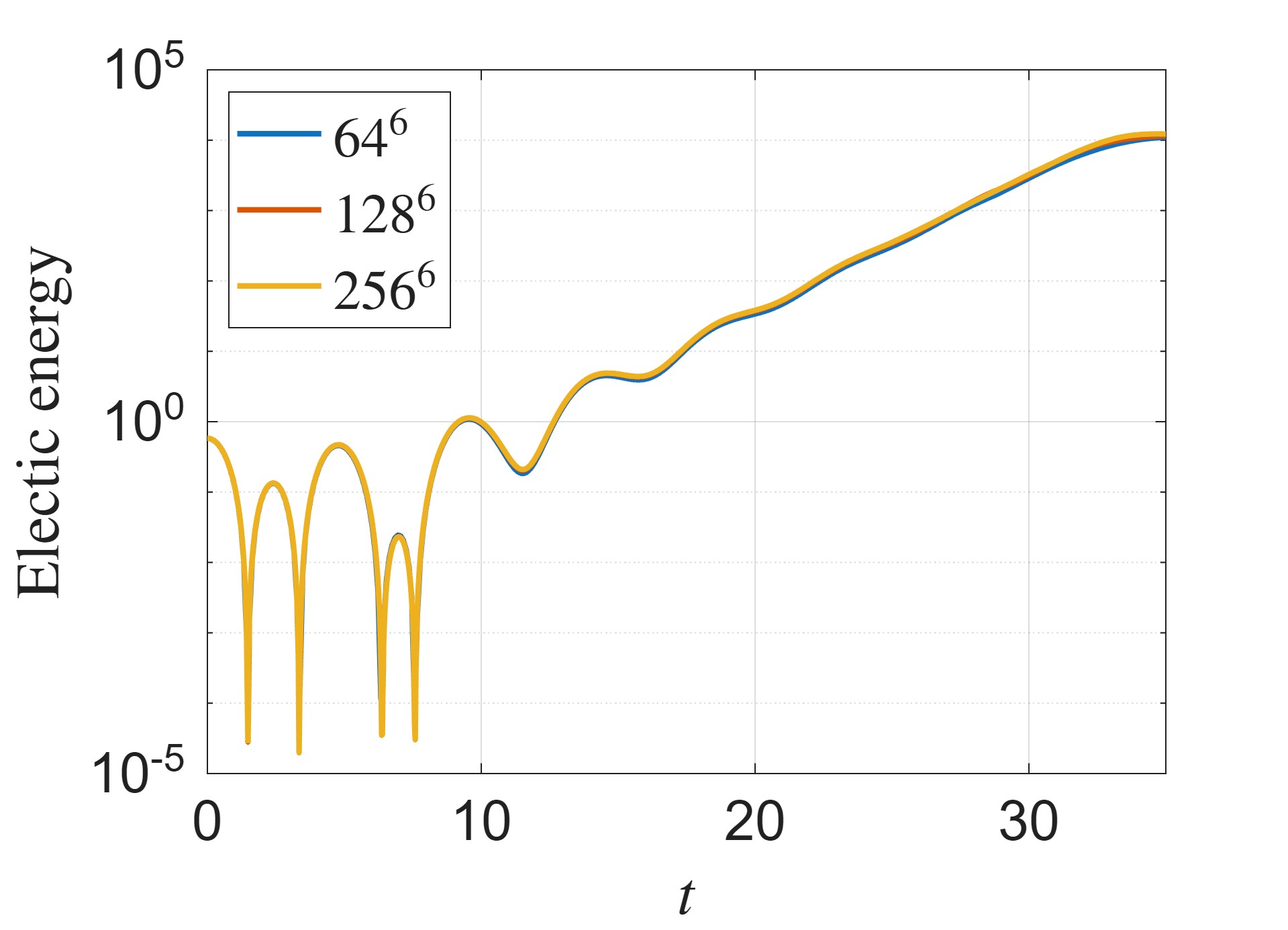}
    }
    \subfigure[]{
        \includegraphics[width=0.299\linewidth]{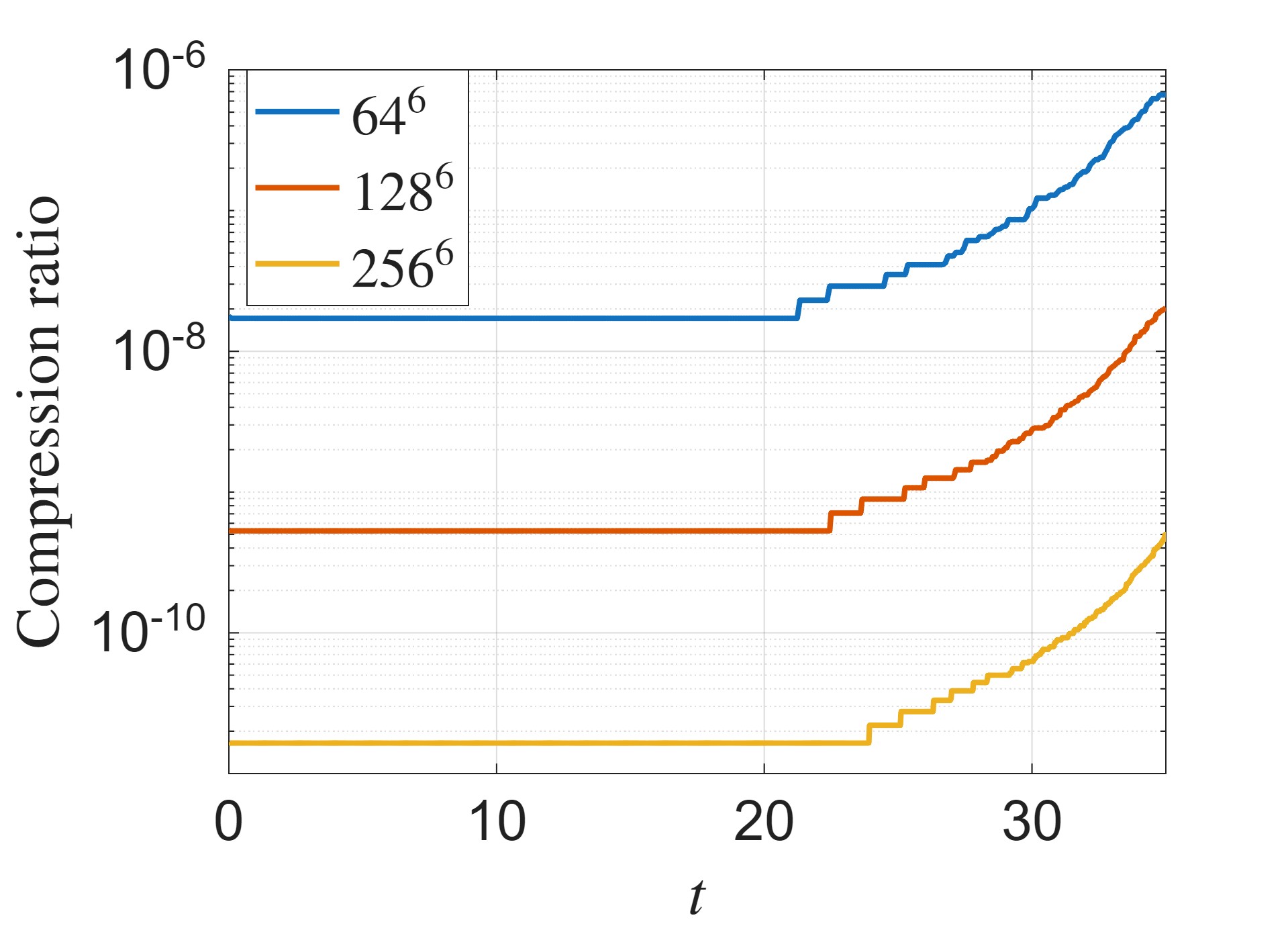}
    }
    \subfigure[]{
        \includegraphics[width=0.299\linewidth]{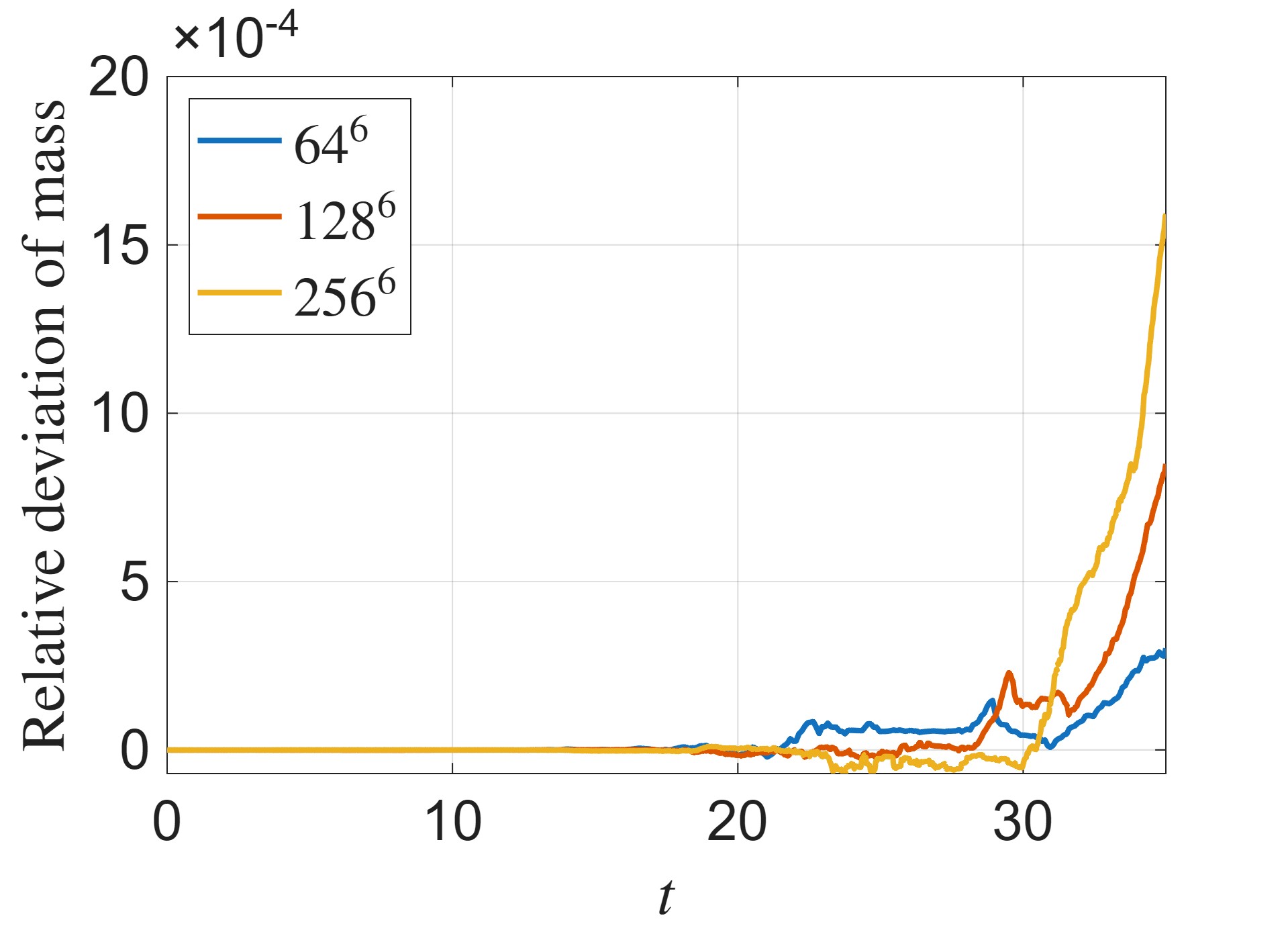}
    }
    
    \subfigure[]{
        \includegraphics[width=0.299\linewidth]{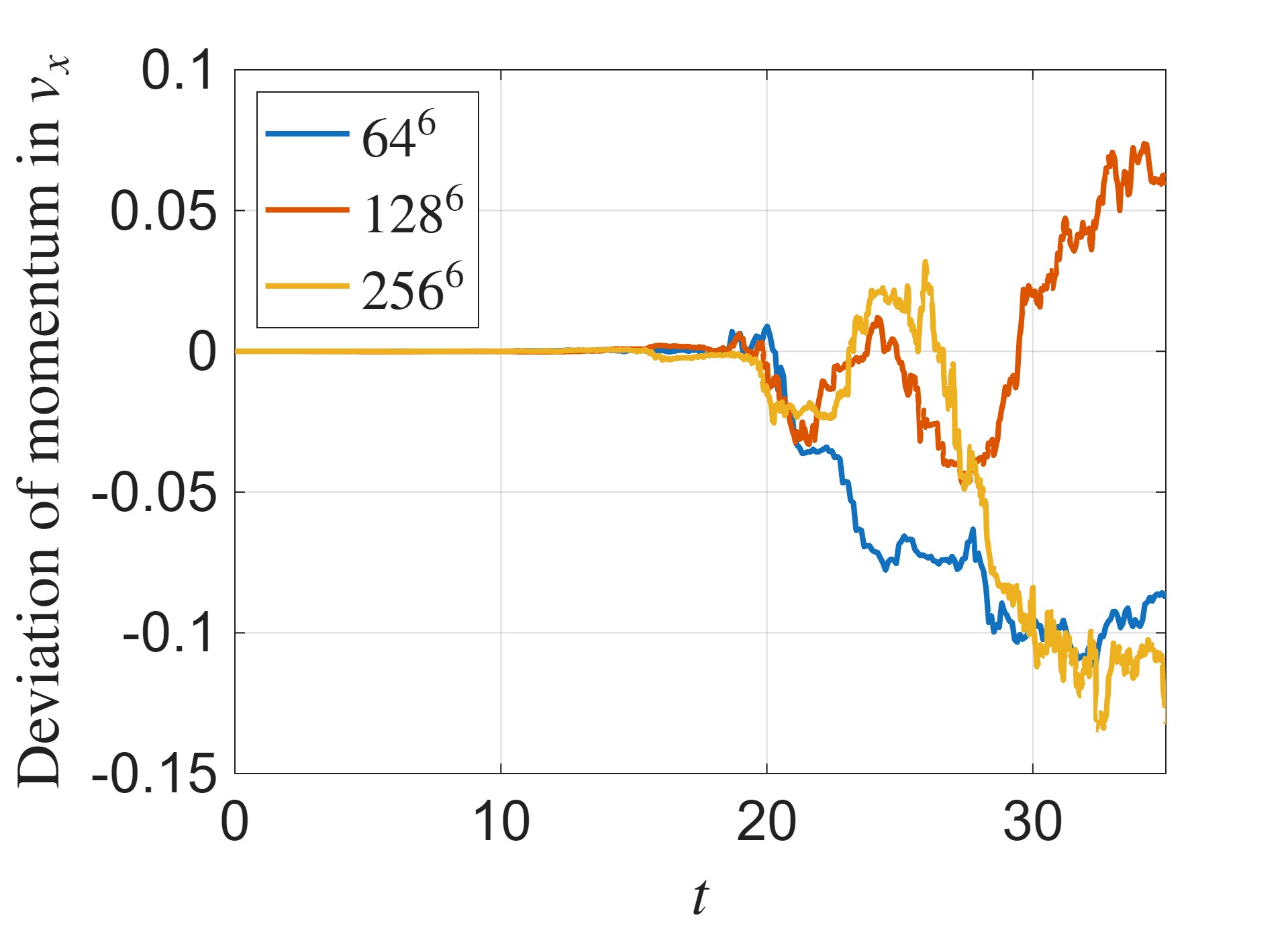}
    }
    \subfigure[]{
        \includegraphics[width=0.299\linewidth]{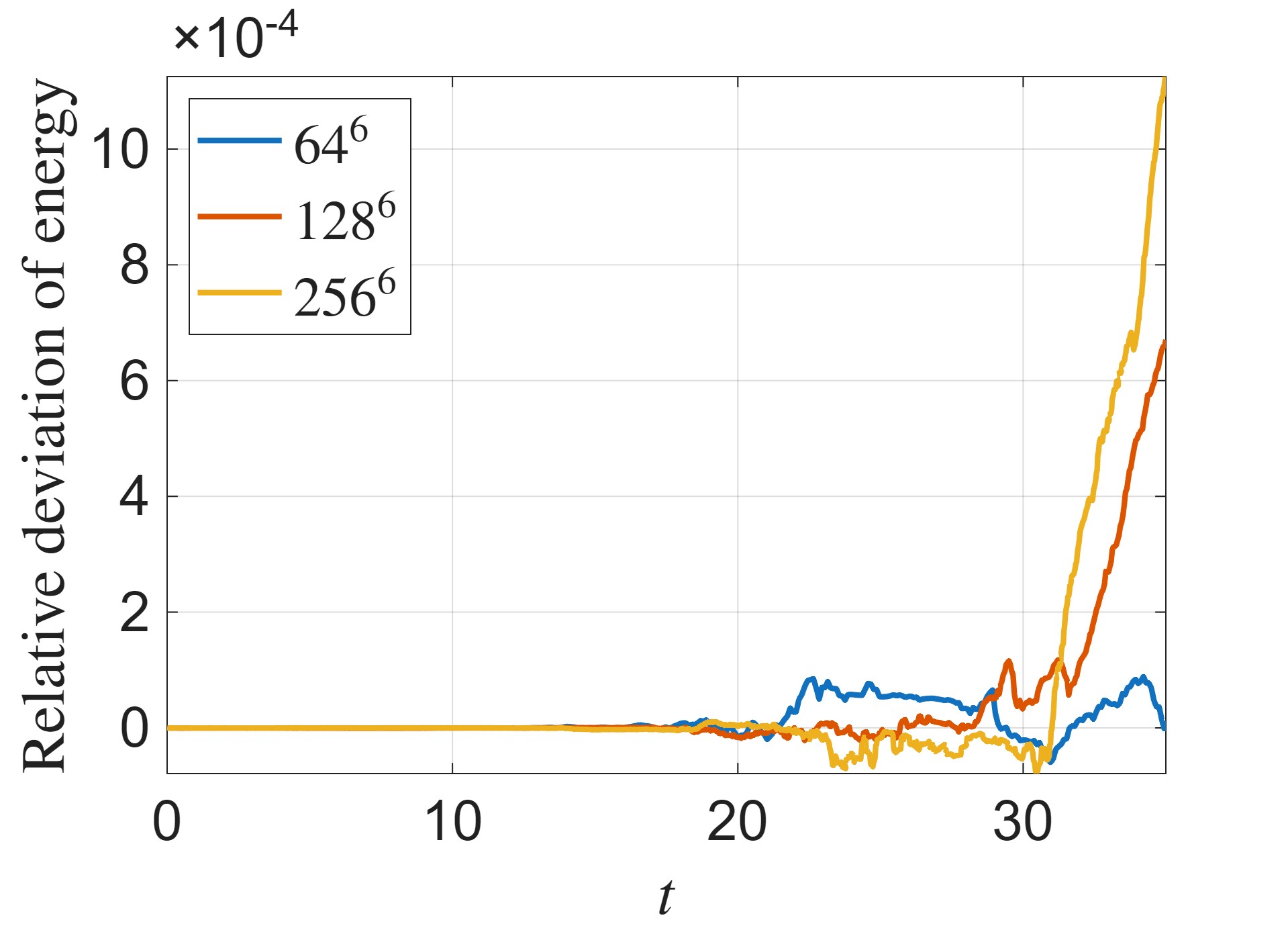}
    }
    \subfigure[]{
        \includegraphics[width=0.299\linewidth]{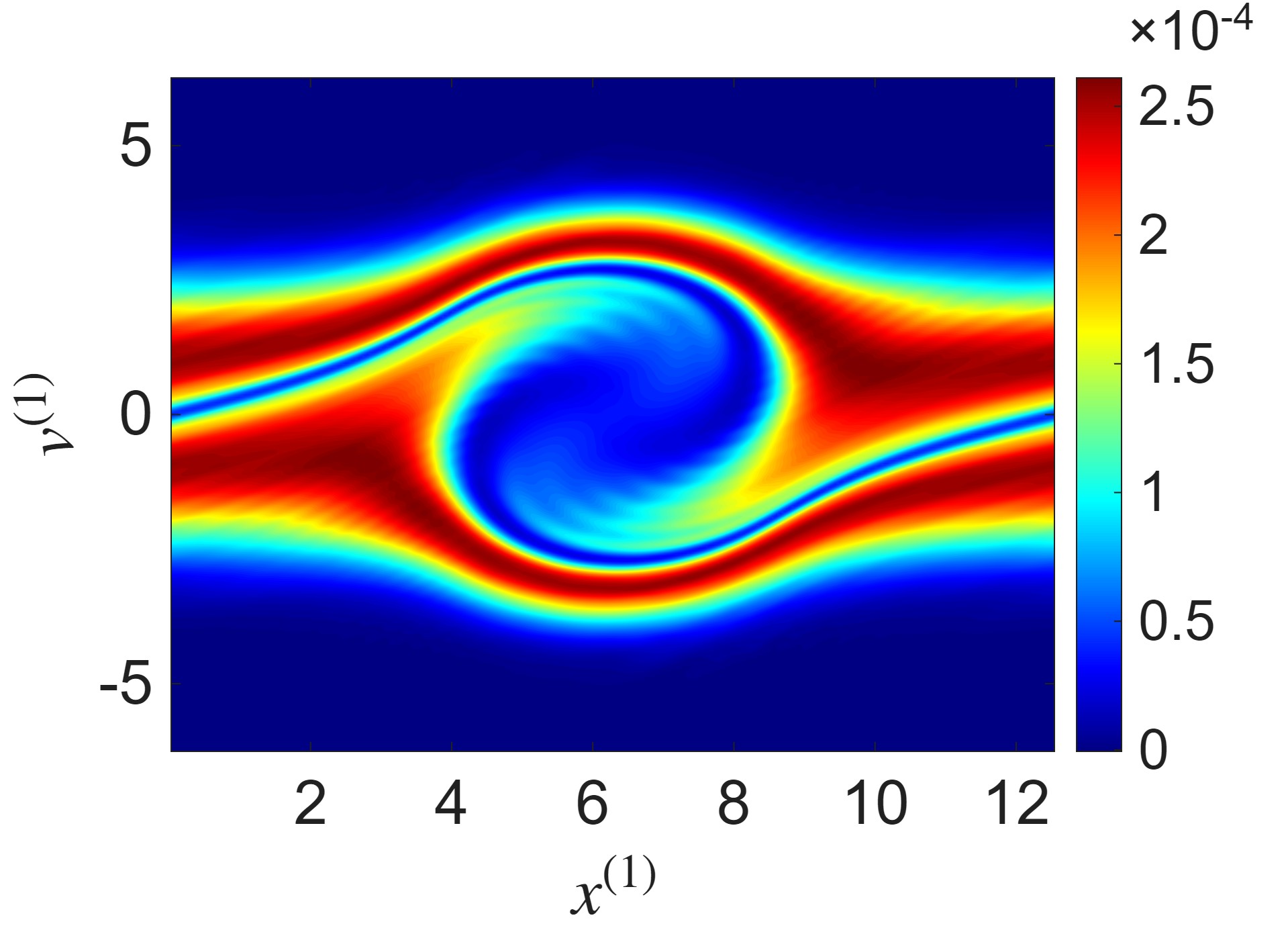}
    }
    \caption{(3D3V two-stream instability). Selected results of the SLAR method: (a)–(e) Time histories of electric energy, compression ratio, relative deviation of mass, deviation of momentum in $v_x$, and relative deviation of energy, for resolutions $64^6$, $128^6$, and $256^6$. (f) $f(x^{(1)},v^{(1)})$ on the slice $(x^{(2)}, v^{(2)}, x^{(3)}, v^{(3)})=(0,0,0,0)$ at $t=35$ for resolution $256^6$.
    Simulation settings: $v_{\max}=8$, CFL = 5, $\varepsilon_{\text{Base}}=5\times10^{-4}$, with minimum leaf ranks fixed at 3 and unrestricted non-leaf ranks.}
    \label{fig:3d3v_two_stream_histories}
\end{figure}

\end{example}

\begin{example}[Angled two-stream benchmark]
We consider the 2D2V VP system with an initial condition consisting of two drifting Maxwellians with equal magnitude $v_0$, aligned by the angle $\theta$. Specifically,
\begin{align*}
f(\bm{x},\bm{v},0)
&=
\frac{1}{4\pi}
\left(
1+\alpha\cos\!\left(k\frac{\bm{u}}{v_0}\cdot\bm{x}\right)
+\alpha\cos\!\left(k\frac{\bm{w}(\theta)}{v_0}\cdot\bm{x}\right)
\right)
g_\theta(\bm v),\\
g_\theta(\bm v)
&=
\exp\left(-\frac{\|\bm{v}-\bm{u}\|^2}{2}\right)
+
\exp\left(-\frac{\|\bm{v}-\bm{w}(\theta)\|^2}{2}\right).
\end{align*}
Here the first stream is fixed as $\bm{u}=\frac{v_0}{\sqrt{2}}(1,1)$, and the second stream is defined by 
$\bm{w}(\theta)=v_0\big(\cos\theta\,\frac{(1,1)}{\sqrt{2}}+\sin\theta\,\frac{(1,-1)}{\sqrt{2}}\big)$, 
so that the angle between the two streams is exactly $\theta$. 
Here $\alpha=0.005$, $v_0=2.4$, $k=0.2$, and we consider $\theta \in \left\{\tfrac{\pi}{4},\tfrac{\pi}{2},\tfrac{3\pi}{4}\right\}$, corresponding to moderately oblique, orthogonal, and strongly oblique configurations.

Figure~\ref{fig:2d2v_angled_two_stream_histories} compares the two
dimension orderings for the angled two-stream benchmark. The top row shows the
time histories of the electric energy, while the bottom row reports the
corresponding compression ratios. For each angle, the two orderings produce
nearly indistinguishable electric-energy histories over the common simulated
time interval, indicating that the ordering primarily affects the efficiency of
the HT representation rather than the resolved dynamics. The qualitative
behavior depends on the angle between the beams. For \(\theta=\pi/4\), the
evolution is closer to damping-dominated behavior, while for the more strongly
oblique case \(\theta=3\pi/4\), the beams have a more pronounced
counter-streaming component and the electric energy exhibits a history more
reminiscent of the two-stream instability. The orthogonal case
\(\theta=\pi/2\) lies between these two regimes.

\begin{figure}[htb]
    \centering
    \subfigure[$\theta = \pi/4$]{
        \includegraphics[width=0.299\linewidth]{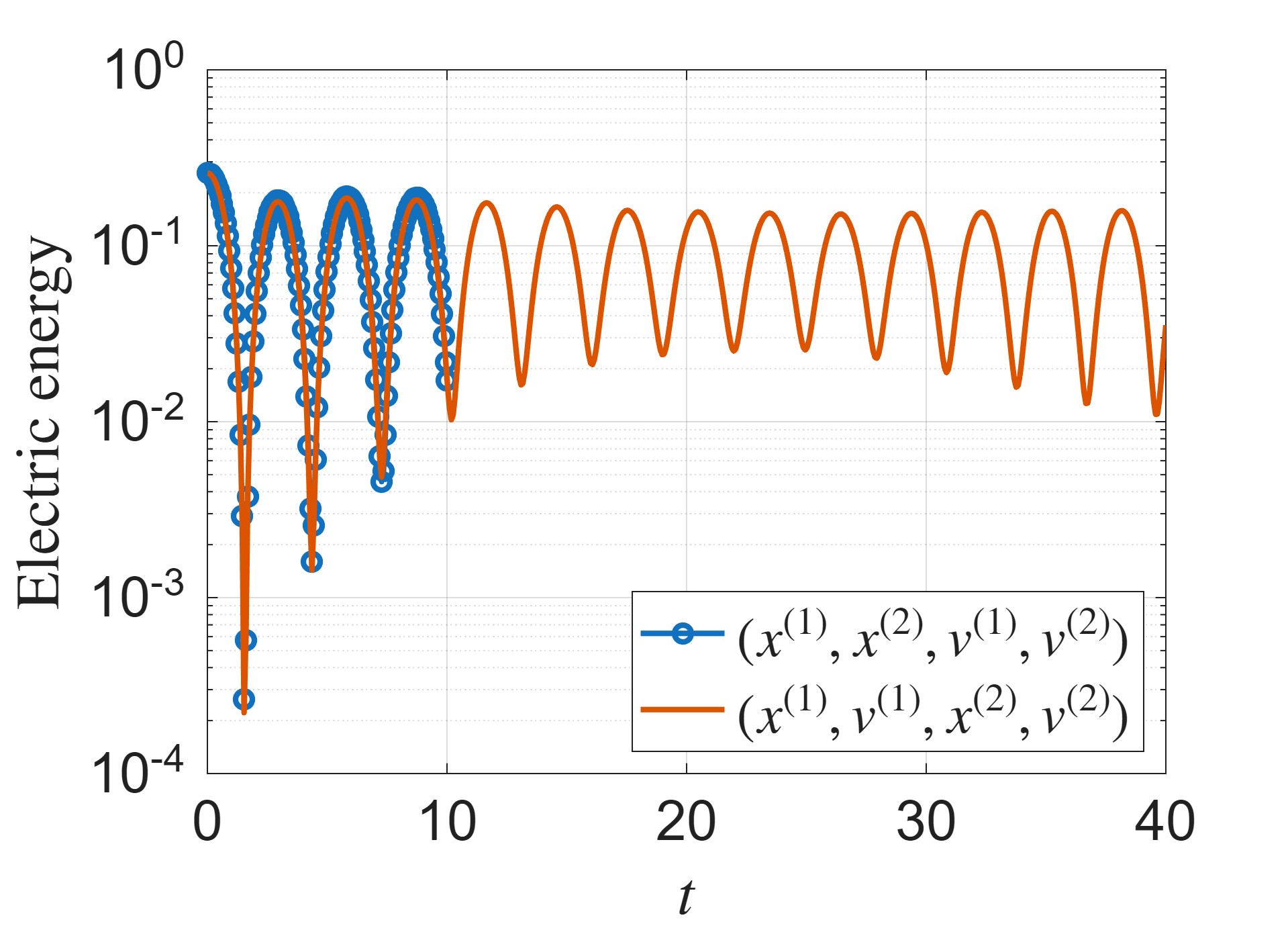}
    }
    \subfigure[$\theta = \pi/2$]{
        \includegraphics[width=0.299\linewidth]{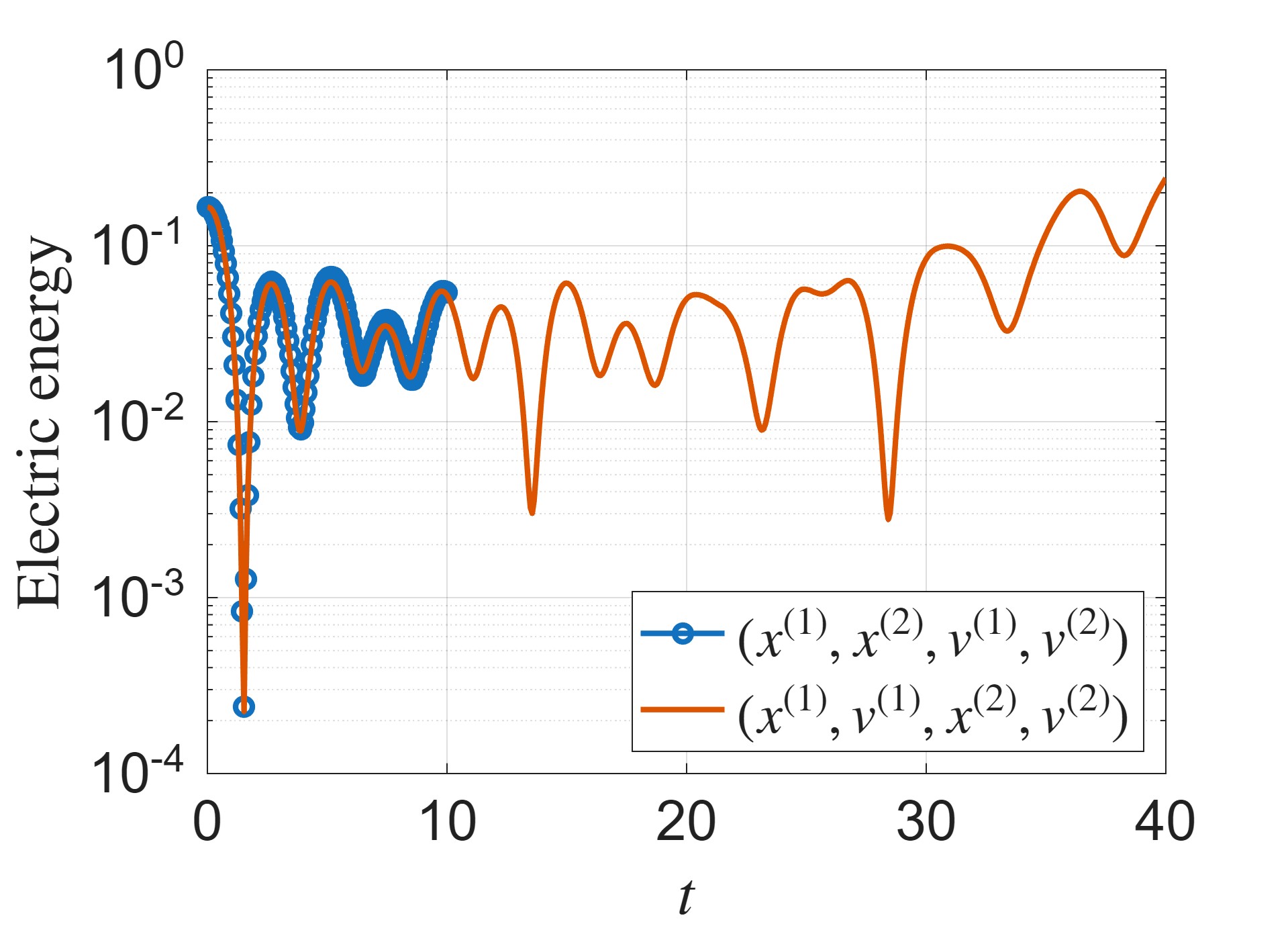}
    }
    \subfigure[$\theta = 3\pi/4$]{
        \includegraphics[width=0.299\linewidth]{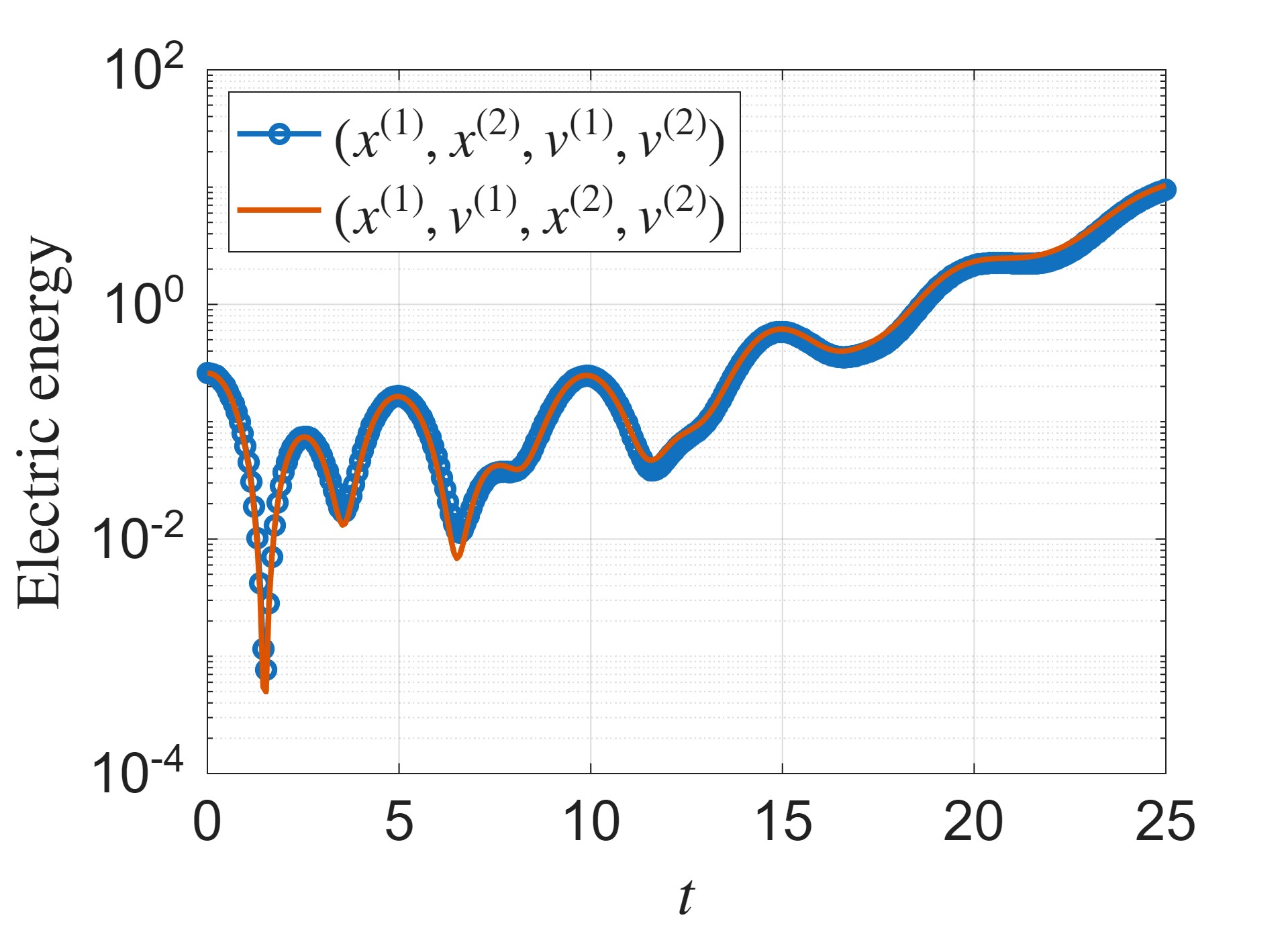}
    }
    
    \subfigure[$\theta = \pi/4$]{
        \includegraphics[width=0.299\linewidth]{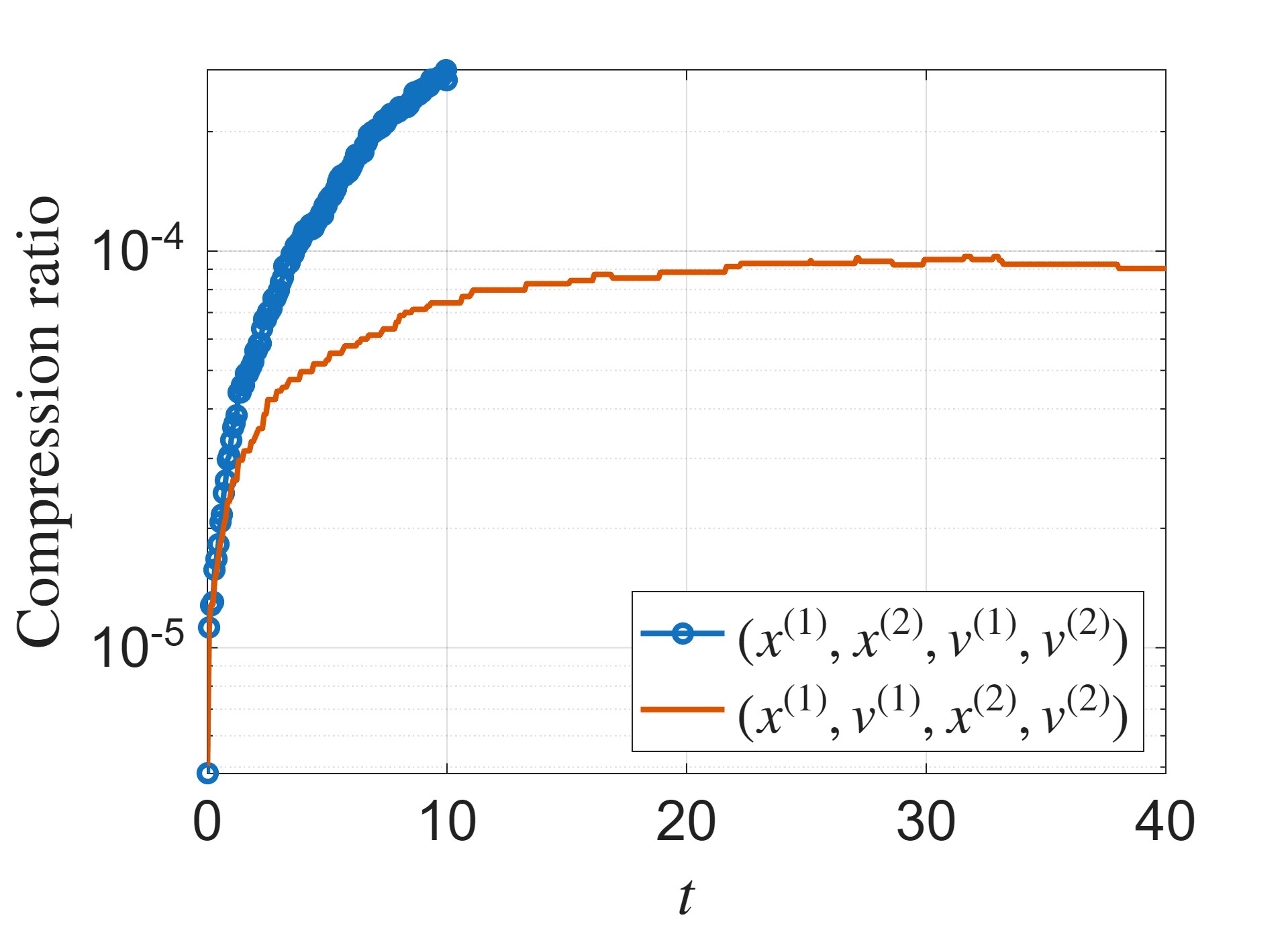}
    }
    \subfigure[$\theta = \pi/2$]{
        \includegraphics[width=0.299\linewidth]{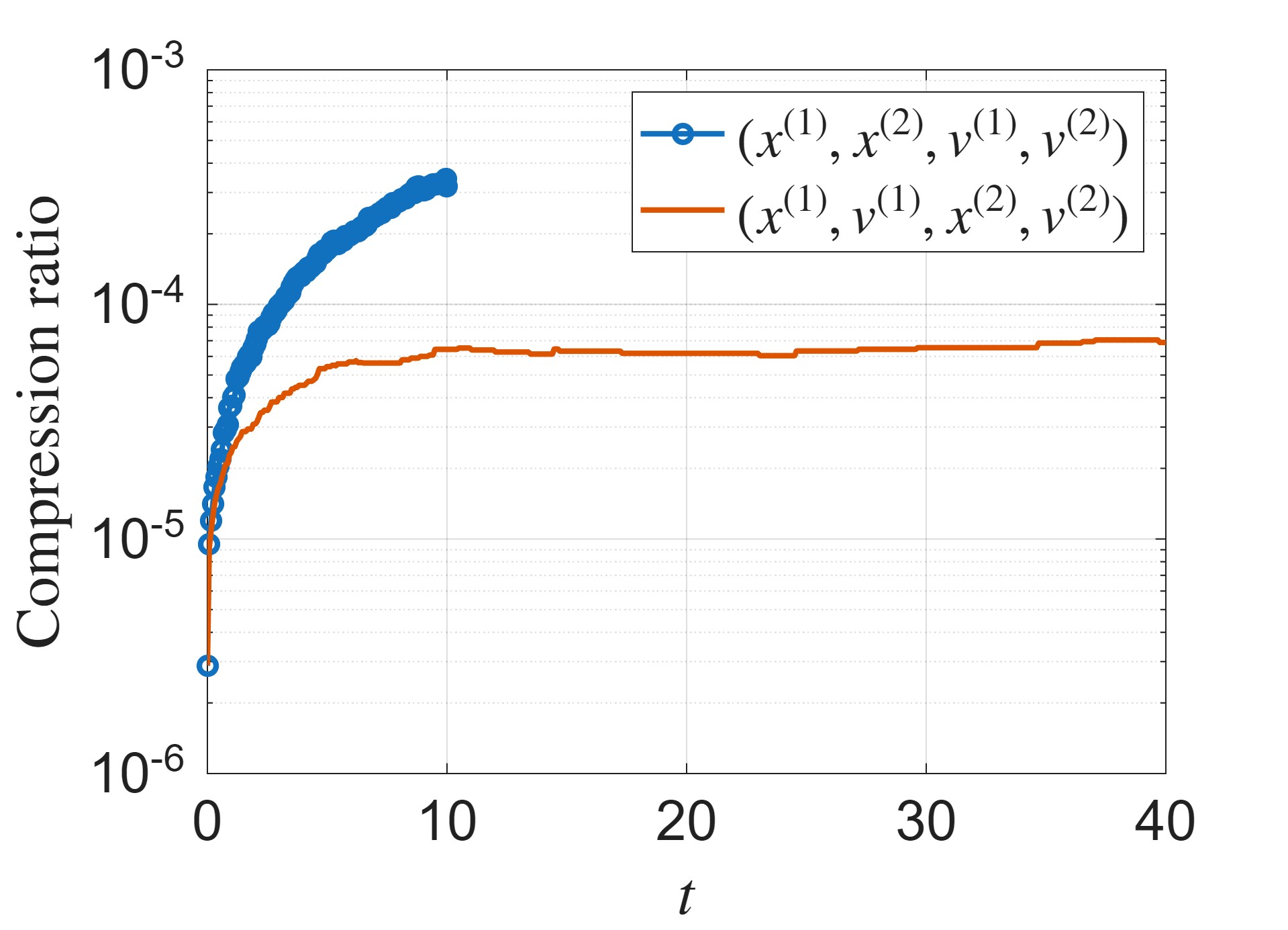}
    }
    \subfigure[$\theta = 3\pi/4$]{
        \includegraphics[width=0.299\linewidth]{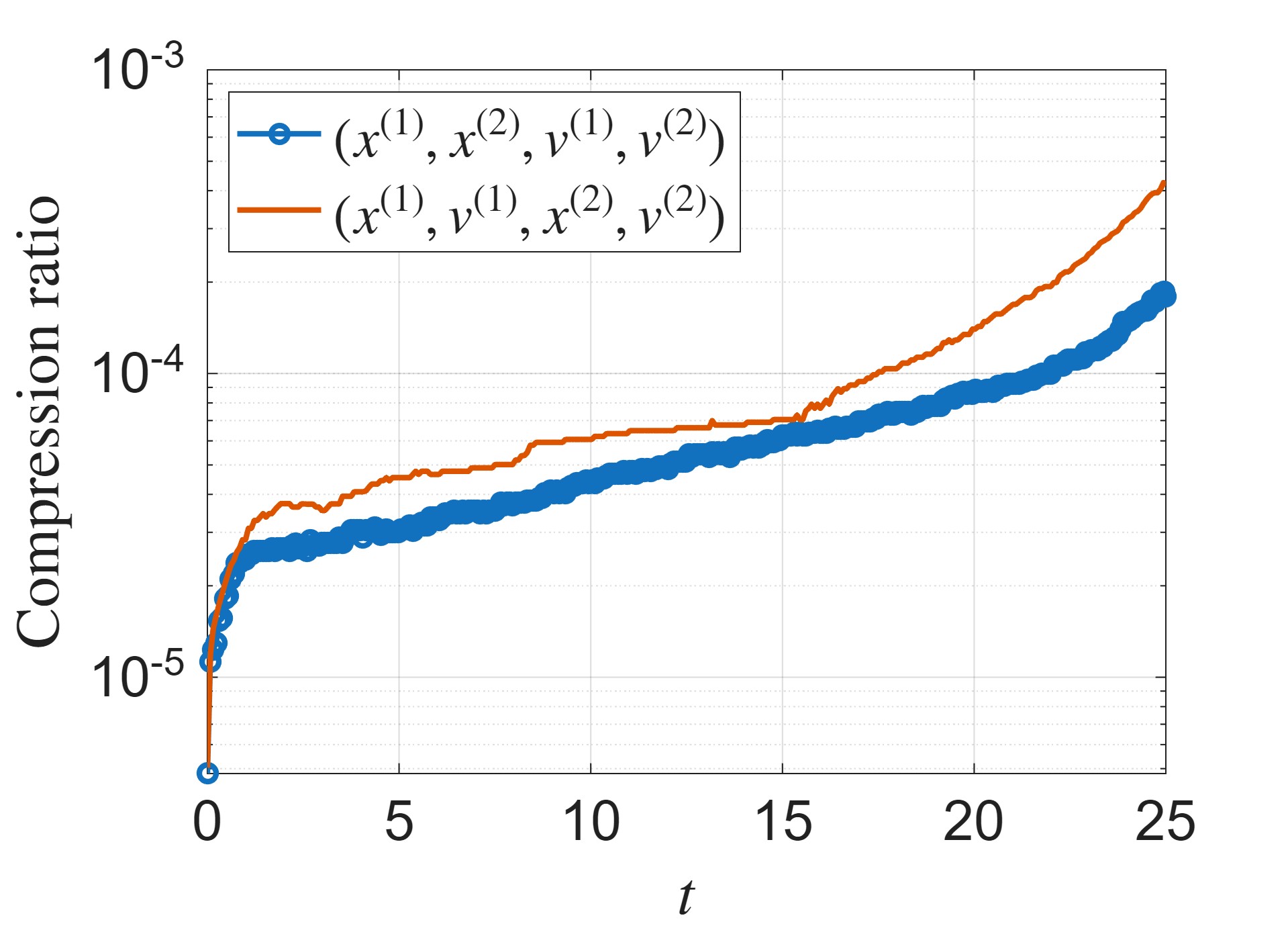}
    }
    \caption{(Angled two-stream benchmark) Comparison of two dimension orderings for the 2D2V angled two-stream instability.
(a)--(c) Time histories of the electric energy and (d)--(f) compression ratios
for \(\theta=\pi/4\), \(\theta=\pi/2\), and \(\theta=3\pi/4\), respectively.
The blue curves correspond to the ordering
\(\{x^{(1)},x^{(2)},v^{(1)},v^{(2)}\}\), while the orange curves correspond to
the same-dimension pairing
\(\{x^{(1)},v^{(1)},x^{(2)},v^{(2)}\}\).
Simulation settings: \(v_{\max}=8\), mesh \(128^4\), CFL \(=5\),
\(\epsilon_{\mathrm{Base}}=5\times 10^{-5}\) for \(\theta=\pi/4\),
and \(\epsilon_{\mathrm{Base}}=1\times 10^{-4}\) for
\(\theta=\pi/2\) and \(\theta=3\pi/4\), with no rank limitations.}
    \label{fig:2d2v_angled_two_stream_histories}
\end{figure}

The compression behavior is more sensitive to the dimension ordering.
For \(\theta=\pi/4\) and \(\theta=\pi/2\), the ordering
\(\{x^{(1)},x^{(2)},v^{(1)},v^{(2)}\}\) leads to rapid growth of the
hierarchical ranks and a much larger compression ratio. As a result, these two
simulations are terminated at \(t=10\). The corresponding wall-clock times are
\(48517\) s and \(107866\) s, approximately \(13.5\) and \(30.0\) hours,
respectively. By contrast, the completed simulations using the ordering
\(\{x^{(1)},v^{(1)},x^{(2)},v^{(2)}\}\) require \(23968\) s and \(10258\) s,
approximately \(6.7\) and \(2.8\) hours. For the strongly oblique case
\(\theta=3\pi/4\), the ordering
\(\{x^{(1)},x^{(2)},v^{(1)},v^{(2)}\}\) gives better compression behavior, with
wall-clock time \(30577\) s, compared with \(58019\) s for
\(\{x^{(1)},v^{(1)},x^{(2)},v^{(2)}\}\), corresponding to approximately
\(8.5\) and \(16.1\) hours, respectively. Consistent with these observations,
\Cref{fig:2d2v_angled_two_stream_rank_histories} shows that, for a fixed
\(\theta\), the leaf-rank histories are generally similar for the two
orderings, while the non-leaf ranks differ substantially. The ordering with the
smaller compression ratio is precisely the one with smaller non-leaf ranks.

\begin{figure}[htb]
    \centering
    \subfigure[$\theta = \pi/4$]{
        \includegraphics[width=0.299\linewidth]{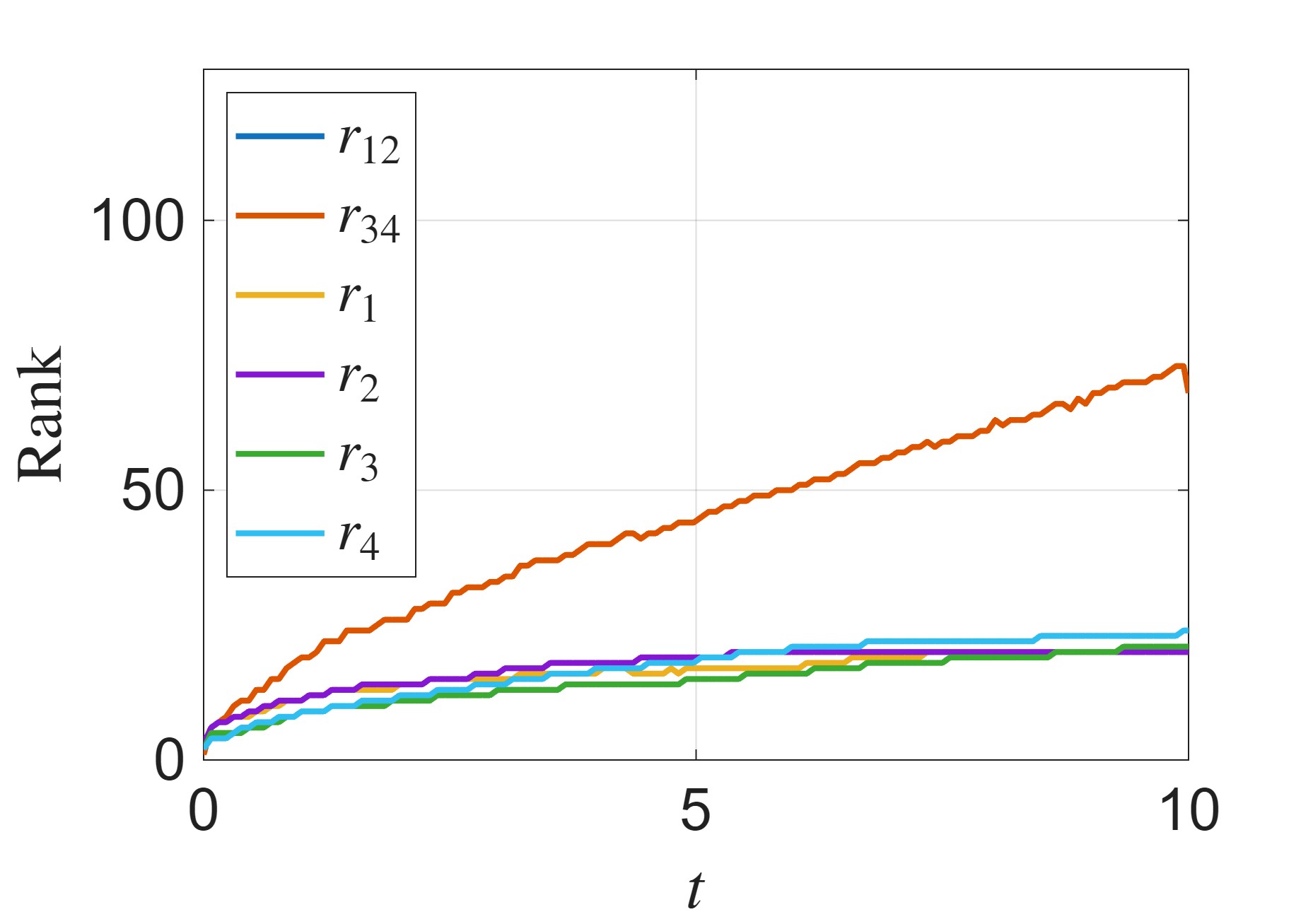}
    }
    \subfigure[$\theta = \pi/2$]{
        \includegraphics[width=0.299\linewidth]{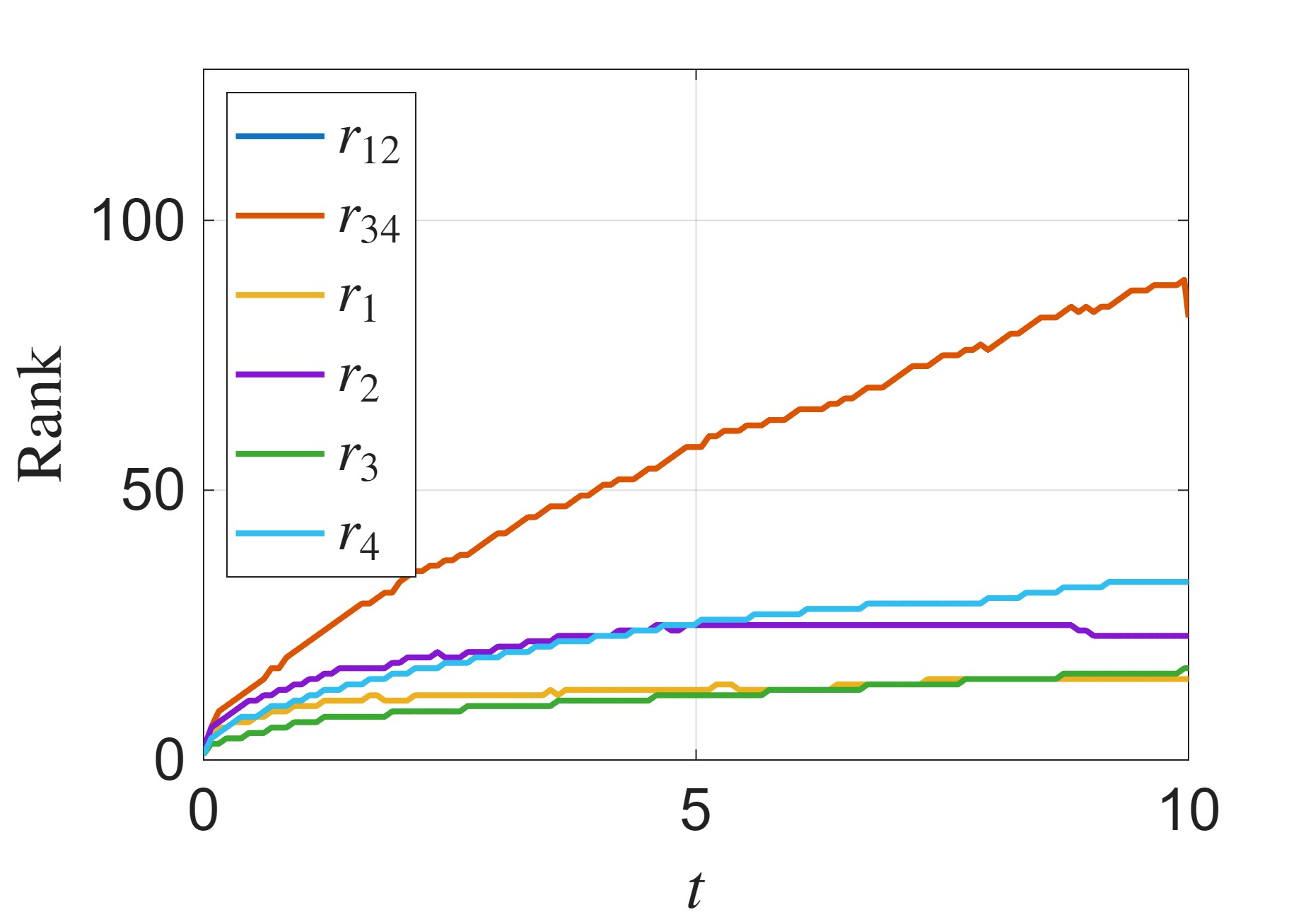}
    }
    \subfigure[$\theta = 3\pi/4$]{
        \includegraphics[width=0.299\linewidth]{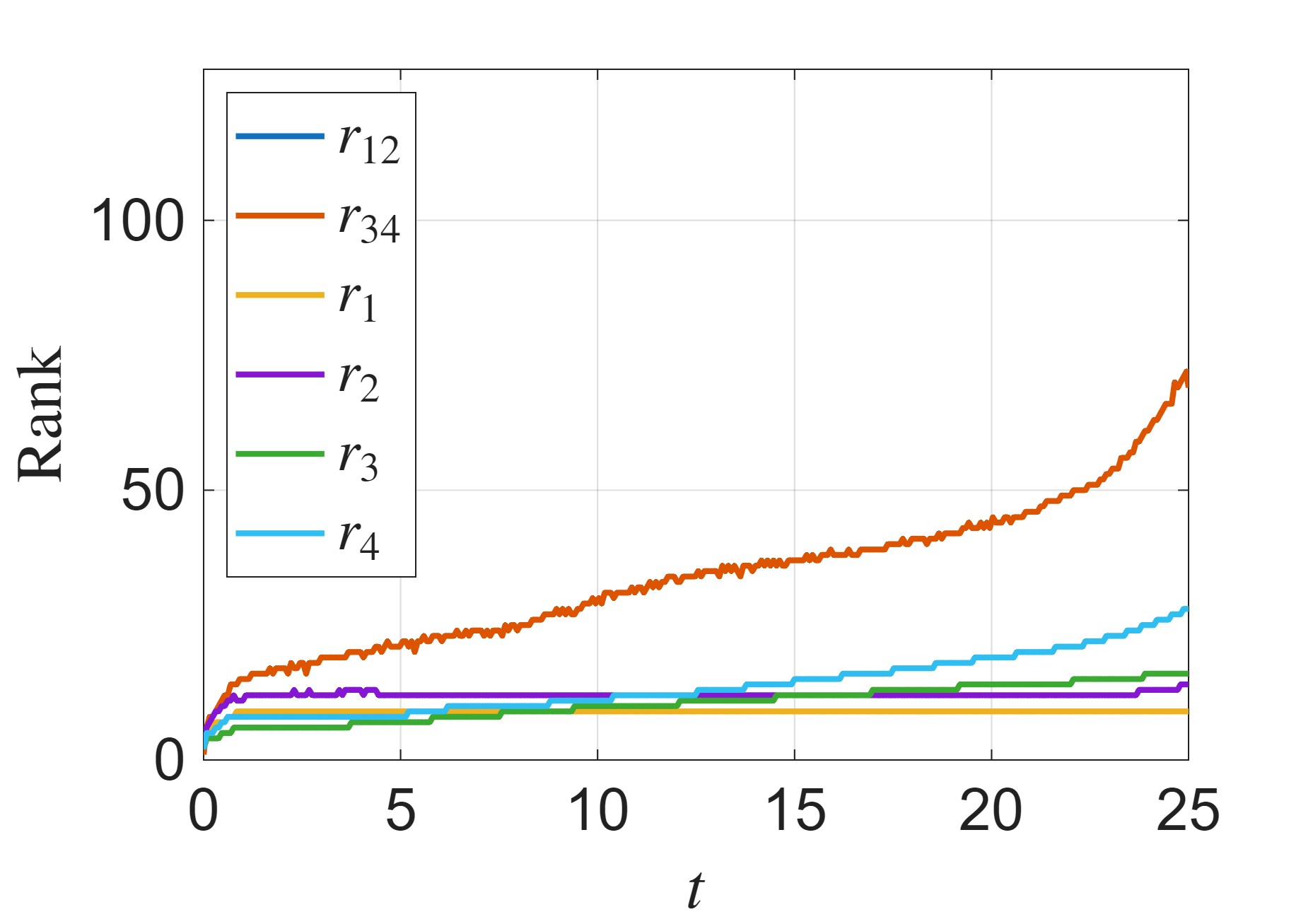}
    }

    \subfigure[$\theta = \pi/4$]{
        \includegraphics[width=0.299\linewidth]{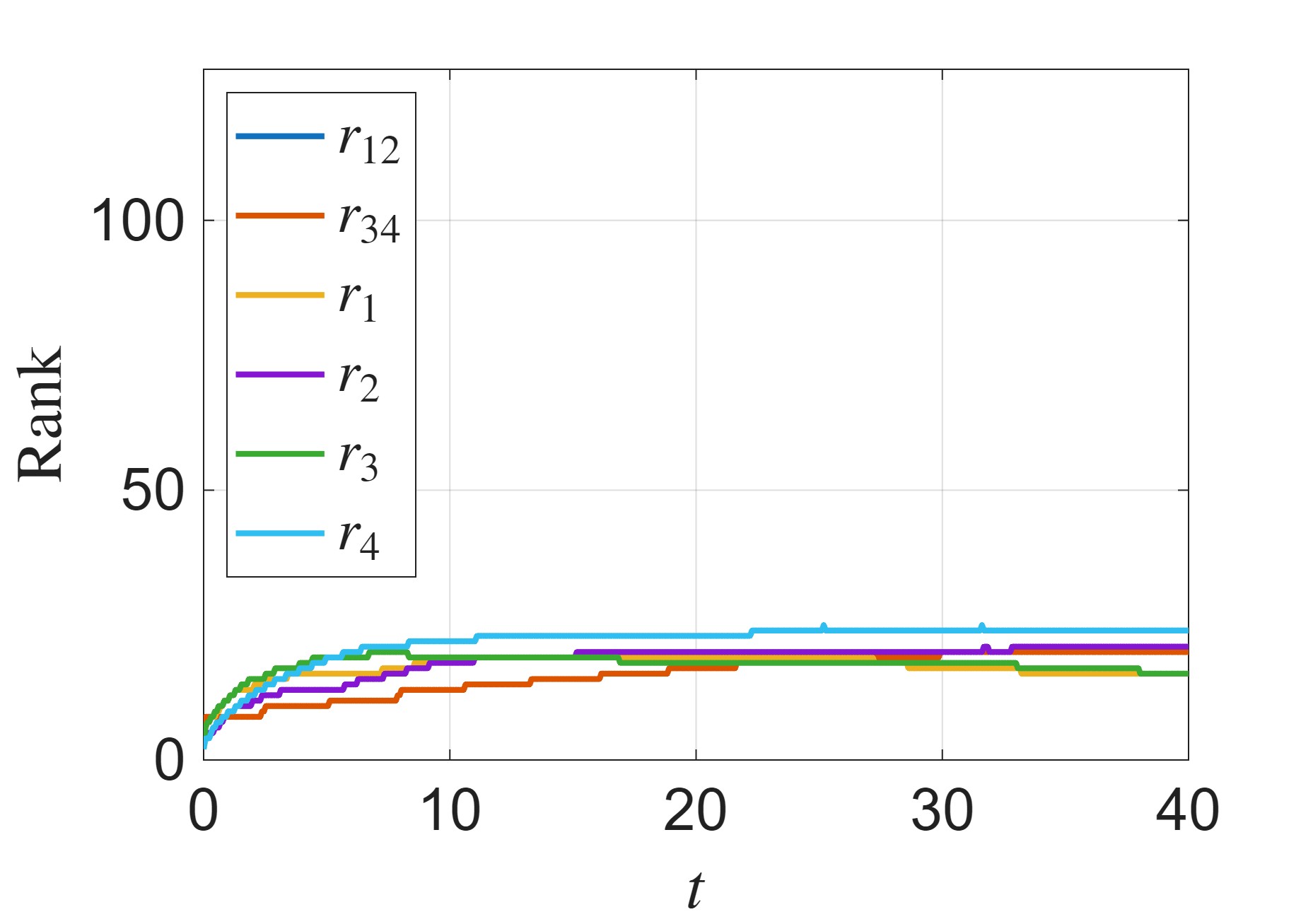}
    }
    \subfigure[$\theta = \pi/2$]{
        \includegraphics[width=0.299\linewidth]{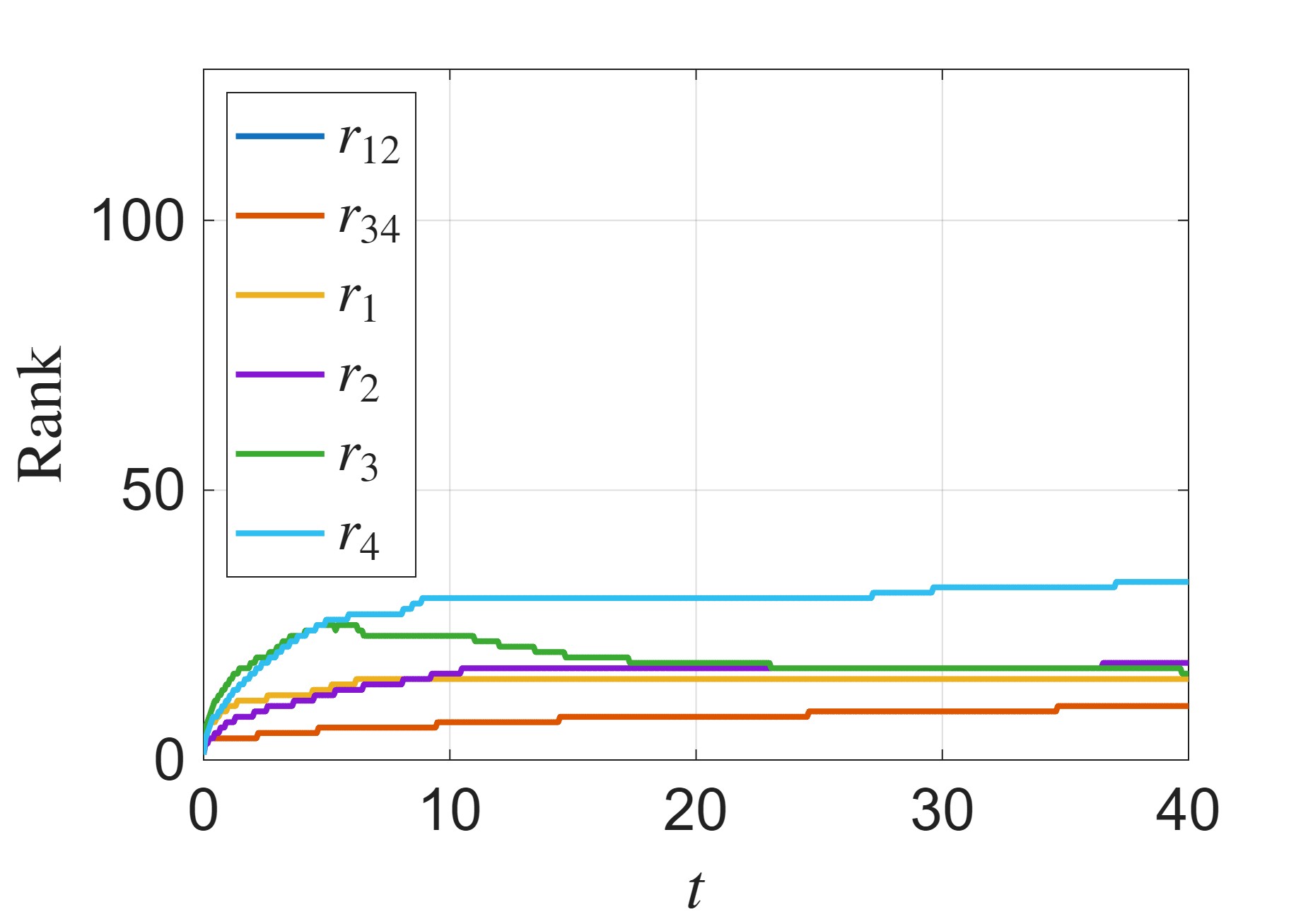}
    }
    \subfigure[$\theta = 3\pi/4$]{
        \includegraphics[width=0.299\linewidth]{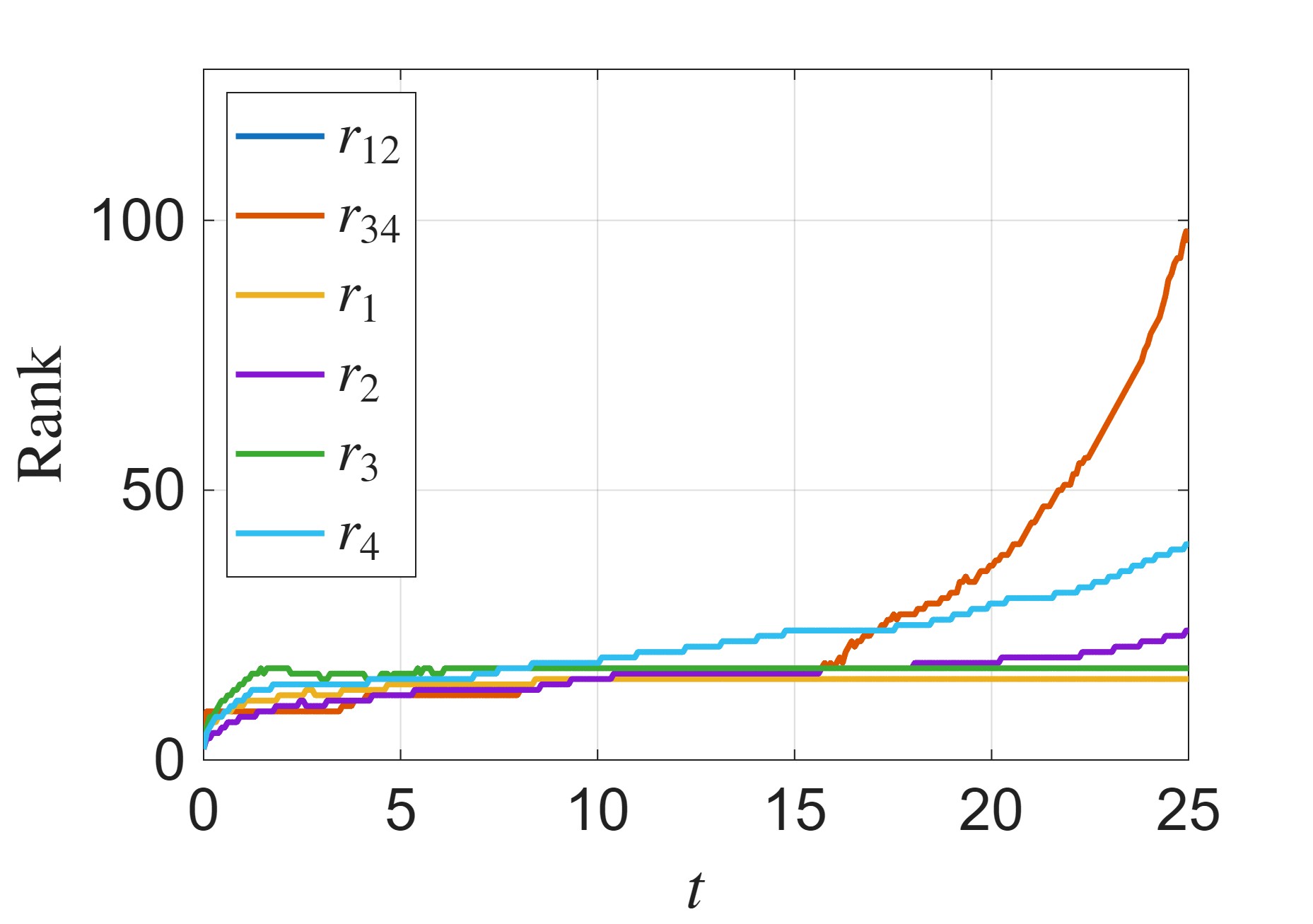}
    }
    \caption{(Angled two-stream benchmark) Comparison of two dimension orderings for the 2D2V angled two-stream instability.
(a)--(c) Rank histories for the ordering
$\{x^{(1)},x^{(2)},v^{(1)},v^{(2)}\}$
with $\theta=\pi/4$, $\pi/2$, and $3\pi/4$, respectively.
(d)--(f) Rank histories for the ordering
$\{x^{(1)},v^{(1)},x^{(2)},v^{(2)}\}$
with $\theta=\pi/4$, $\pi/2$, and $3\pi/4$, respectively.
Simulation settings: \(v_{\max}=8\), mesh \(128^4\), CFL \(=5\),
\(\epsilon_{\mathrm{Base}}=5\times 10^{-5}\) for \(\theta=\pi/4\),
and \(\epsilon_{\mathrm{Base}}=1\times 10^{-4}\) for
\(\theta=\pi/2\) and \(\theta=3\pi/4\), with no rank limitations.}
    \label{fig:2d2v_angled_two_stream_rank_histories}
\end{figure}

To provide intuition for the large non-leaf ranks observed in some of the
angled two-stream simulations, \Cref{fig:2d2v_ATS_visualizations} presents
a visualization analogous to that in \Cref{fig:2d2v_visualizations}, using
the ordering \(\{x^{(1)},v^{(1)},x^{(2)},v^{(2)}\}\) and
\(\theta=3\pi/4\). In
\Cref{fig:2d2v_ATS_visualizations}(a), the \(\{1,2\}\)-matricization contains
two localized features with a clear oblique orientation in the
\((x^{(1)},v^{(1)})\)--\((x^{(2)},v^{(2)})\) plane. This structure is not well
aligned with the coordinate directions of the matricization, and therefore it
cannot be represented using only a small number of coordinate-aligned row and
column samples. As a result, HTACA selects many rows and columns of this
matricization, as shown in \Cref{fig:2d2v_ATS_visualizations}(d), leading to a
large rank at the associated non-leaf node. By contrast, the first selected residual subtensors for the mode sets \(\{1,2\}\) and \(\{3,4\}\), shown in \Cref{fig:2d2v_ATS_visualizations}(b) and (c), exhibit more organized, nearly banded structures. Accordingly, the selected rows and columns in \Cref{fig:2d2v_ATS_visualizations}(e) and (f) are less concentrated around strongly two-dimensional features. This comparison suggests that the
rank growth is primarily driven by the coupling between the two grouped
phase-space pairs rather than by the one-dimensional structure within each
pair.

\begin{figure}[htb]
    \centering
    \subfigure[]{
        \includegraphics[width=0.299\linewidth]{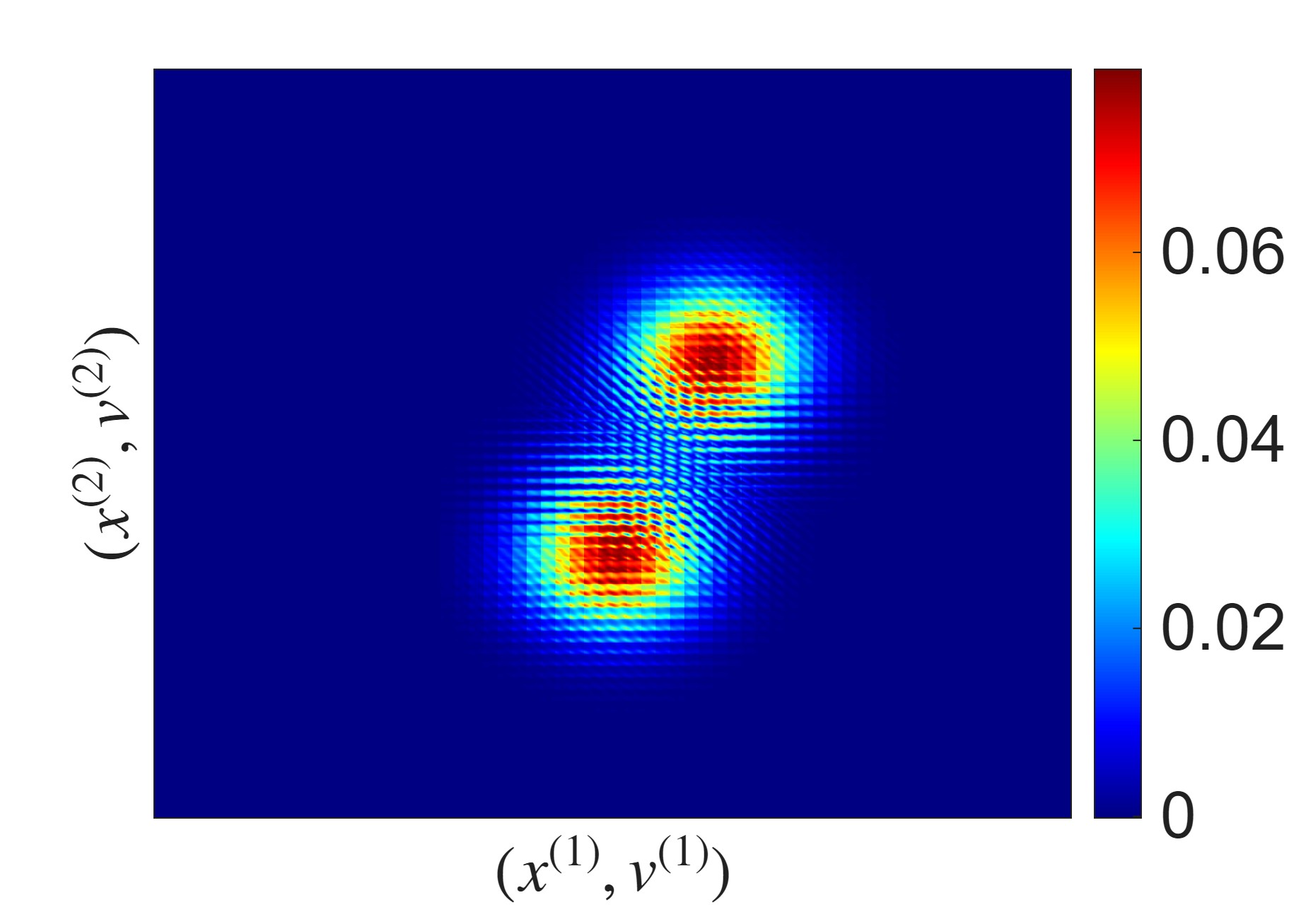}
    }
    \subfigure[]{
        \includegraphics[width=0.299\linewidth]{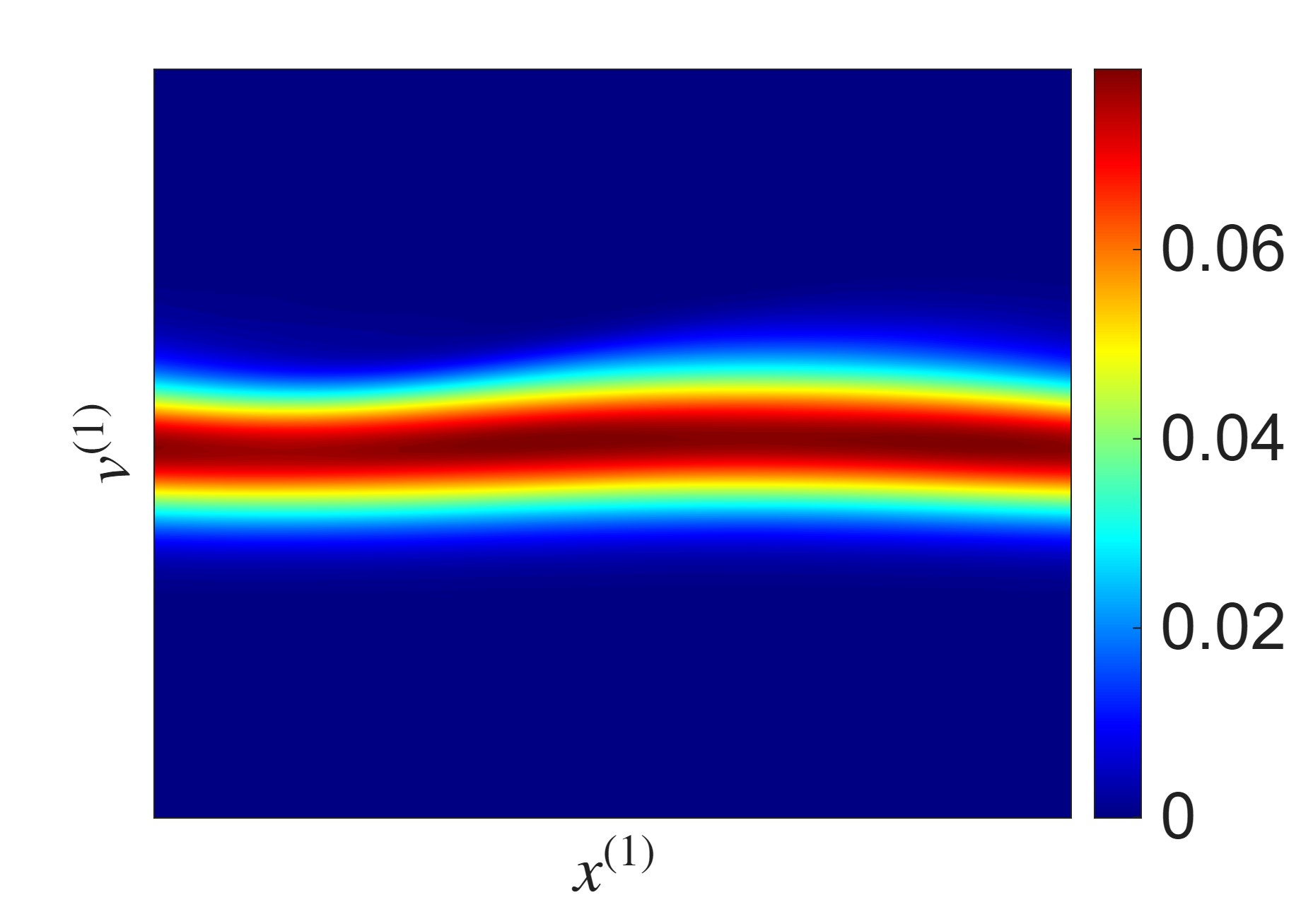}
    }
    \subfigure[]{
        \includegraphics[width=0.299\linewidth]{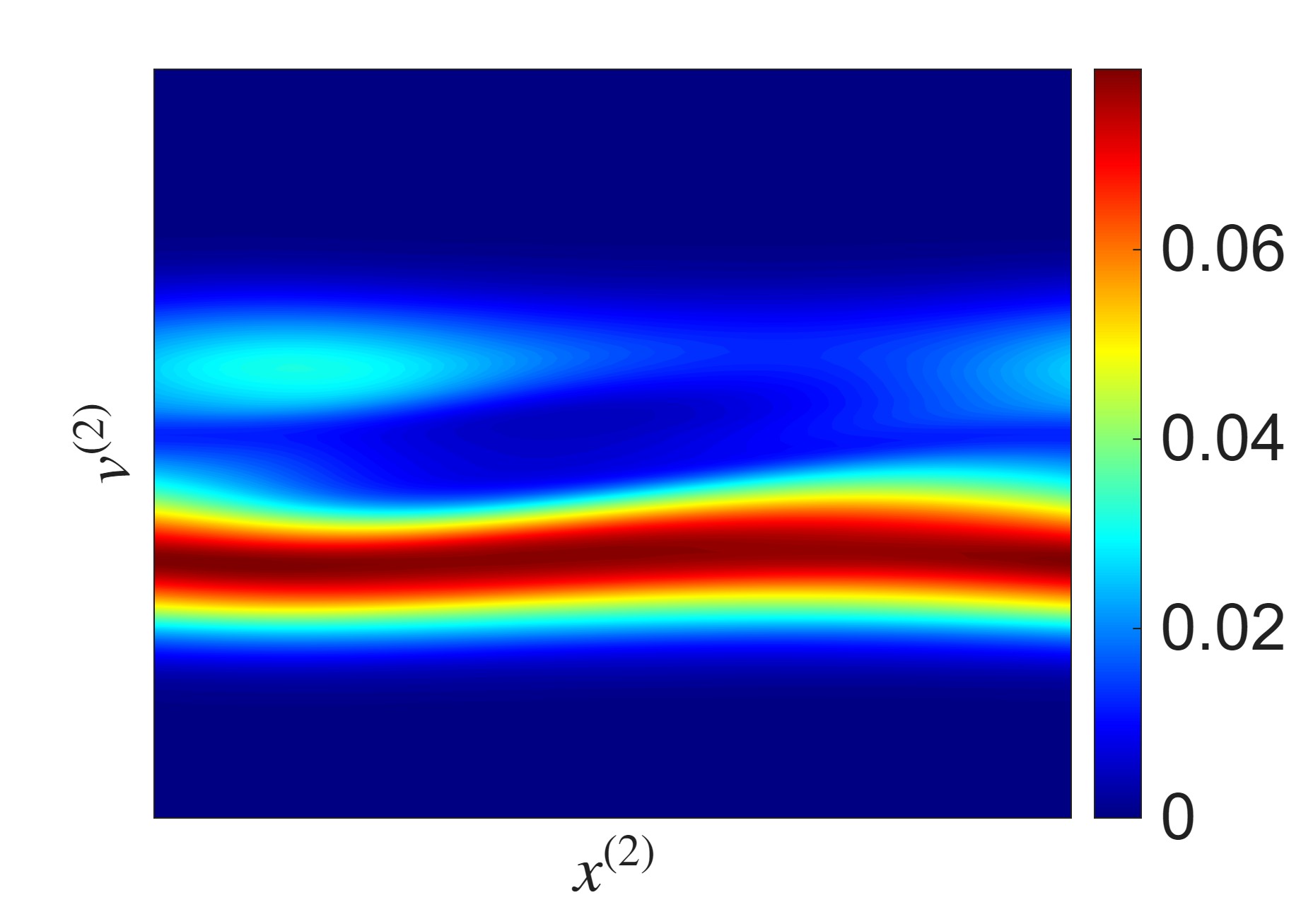}
    }

    \subfigure[]{
        \includegraphics[width=0.299\linewidth]{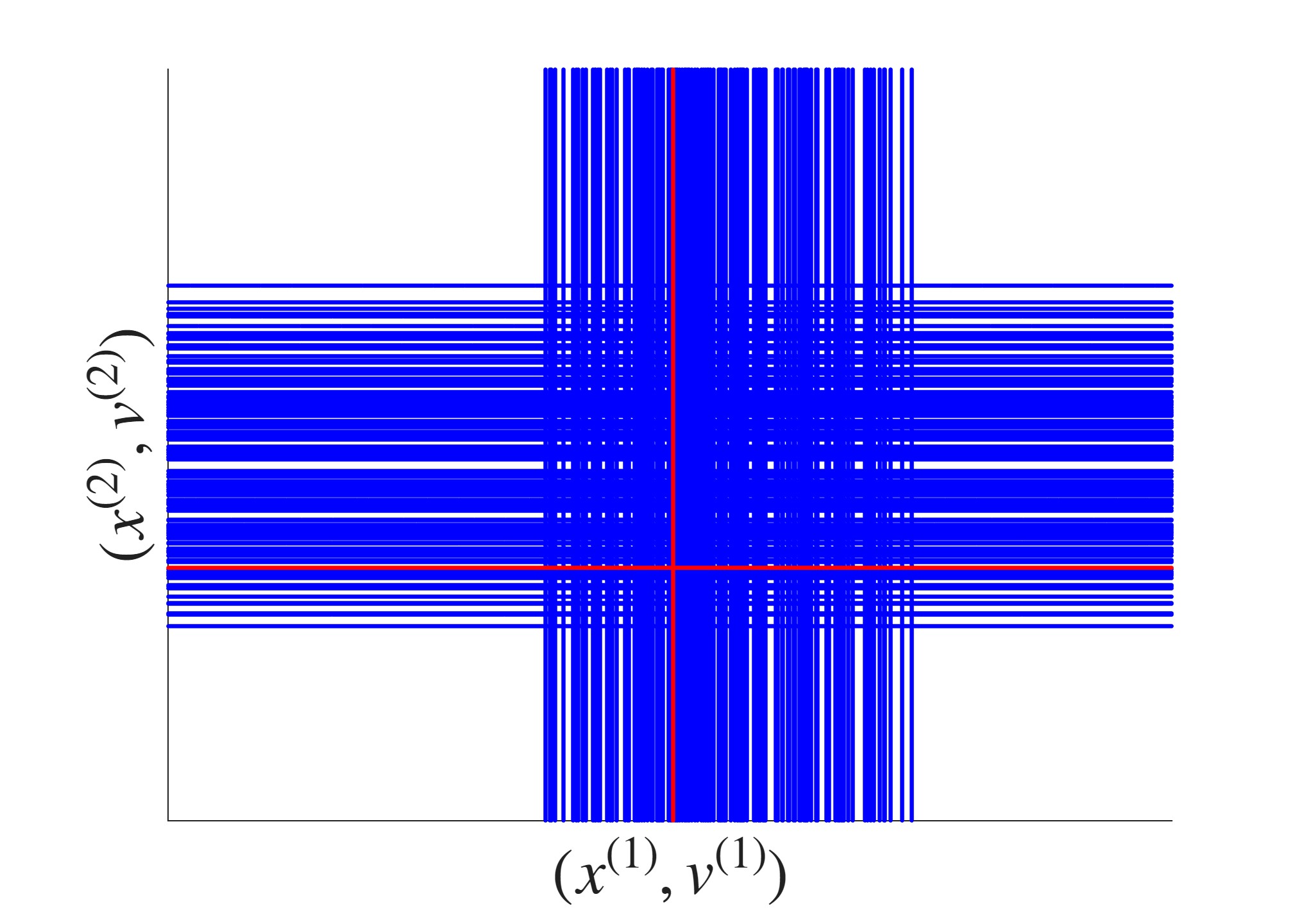}
    }
    \subfigure[]{
        \includegraphics[width=0.299\linewidth]{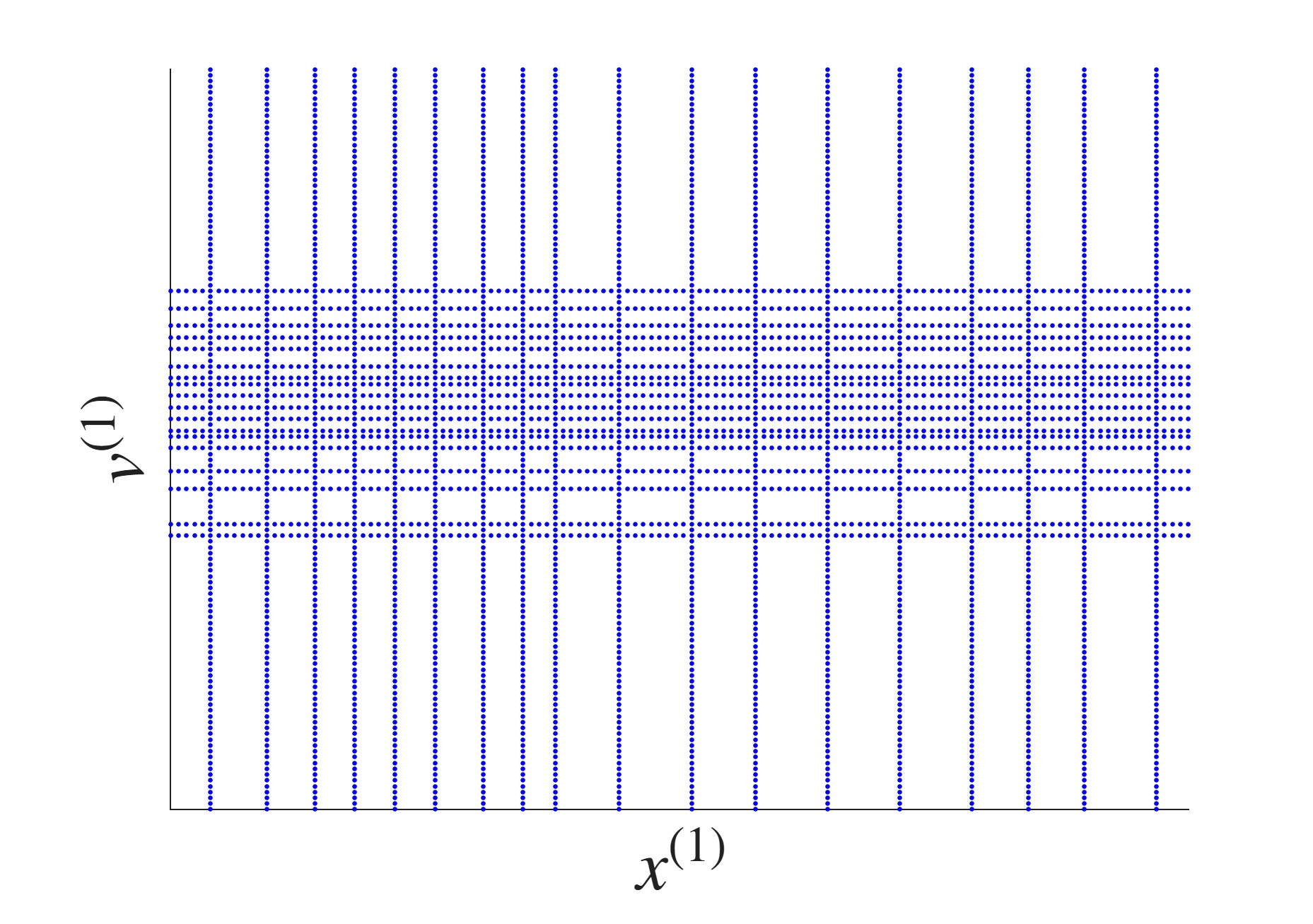}
    }
    \subfigure[]{
        \includegraphics[width=0.299\linewidth]{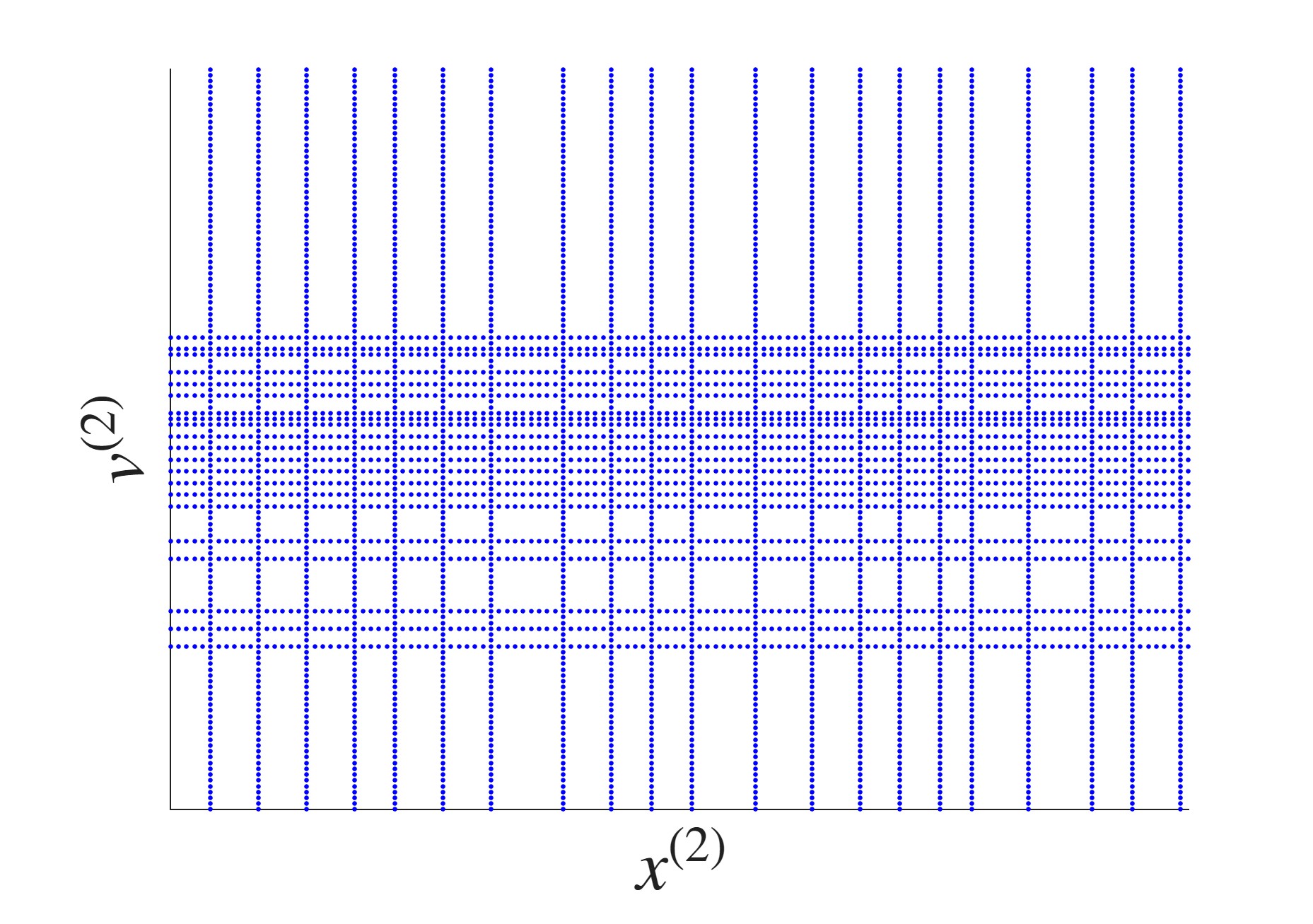}
    }
    \caption{(Angled two-stream benchmark). Contour plots (a)-(c) and selected row/column visualizations (d)-(f) of the numerical distribution for the ordering \(\{x^{(1)},v^{(1)},x^{(2)},v^{(2)}\}\) with $\theta = 3\pi/4$ at \(t = 25\).  
(a) The \(\{1,2\}\)-matricization.  
(b) The first selected \(\{1,2\}\) residual subtensor.  
(c) The first selected \(\{3,4\}\) residual subtensor.  
(d) Selected rows and columns of the \(\{1,2\}\)-matricization (the red row and column are the first selected ones).  
(e) Selected rows and columns of the first selected \(\{1,2\}\) residual subtensor.  
(f) Selected rows and columns of the first selected \(\{3,4\}\) residual subtensor.  
Simulation settings: \(v_{\max} = 2\pi\), CFL = 5, \(\varepsilon_{\text{Base}} = 5 \times 10^{-4}\), and no rank limitations.
}
    \label{fig:2d2v_ATS_visualizations}
\end{figure}

\end{example}

\section{Conclusion}\label{sec:Conclusion}

We proposed a high-dimensional SLAR method based on the HT tensor approximation of the VP system. This work extends our previous SLAR method from the matrix setting to a general high-order tensor setting, enabling simulations in phase
spaces of up to six dimensions. The proposed scheme retains high-order spatial
and temporal accuracy, allows for large time steps, and avoids dimensional
splitting errors. Its adaptive-rank mechanism effectively tracks the evolving
dynamics of kinetic solutions, providing a scalable approach for
high-dimensional problems.

Beyond the VP system considered here, the methods developed in this work also
broaden the applicability of HT-based discretizations. In many
high-dimensional PDEs, certain terms may not be available in separable form, which can limit the direct use of tensor-product low-rank
representations. The compact reconstruction and recursive HTACA procedures
provide a systematic way to construct HT approximations of such nonseparable
terms directly from mesh-based data. This capability extends the class of PDE
models for which HT representations can be used effectively.

Future work will focus on improving the robustness, physical fidelity, and
generality of the proposed HTACA framework. One important direction is the extension
of the implicit LoMaC
framework~\cite{sands2026adaptive} to the 3D3V setting, with the goal of
enhancing long-time conservation in strongly nonlinear regimes while retaining
the efficiency of the adaptive HT representation. We also plan to extend the
current framework to other multiscale kinetic systems, such as the
Vlasov--Maxwell system.

Another important direction is the adaptive selection of effective
dimension orderings, or more generally effective dimension trees. The results
presented here show that dimension ordering can have a substantial impact on
the efficiency of low-rank simulations and suggest that no single ordering is
uniformly optimal. The adaptive selection of an effective dimension tree
therefore remains a challenging open problem, especially for time-dependent PDE simulations. The present study provides only a preliminary exploration of this issue, but its results may help guide the development of heuristics for
selecting dimension trees that yield improved compression
\cite{guo2026hierarchical}.

\bibliographystyle{siam}
\bibliography{Reference}

@article{ballani2013black,
  title={Black box approximation of tensors in hierarchical {Tucker} format},
  author={Ballani, Jonas and Grasedyck, Lars and Kluge, Melanie},
  journal={Linear algebra and its applications},
  volume={438},
  number={2},
  pages={639--657},
  year={2013},
  publisher={Elsevier}
}

@article{zheng2025semi,
  title={{A Semi-{L}agrangian adaptive-rank ({SLAR}) method for linear advection and nonlinear Vlasov-Poisson system}},
  author={Zheng, Nanyi and Hayes, Daniel and Christlieb, Andrew and Qiu, Jing-Mei},
  journal={Journal of Computational Physics},
  pages={113970},
  year={2025},
  publisher={Elsevier}
}

@article{kressner2012htucker,
  title={{htucker — {A} {MATLAB} toolbox for tensors in hierarchical {T}ucker format}},
  author={Kressner, Daniel and Tobler, Christine},
  journal={Mathicse, EPF Lausanne},
  pages={11},
  year={2012}
}

@article{grasedyck2010hierarchical,
  title={{Hierarchical low rank approximation of tensors and multivariate functions}},
  author={Grasedyck, Lars},
  journal={Lecture notes of the Z{\"u}rich summer school on Sparse Tensor Discretizations of High-Dimensional Problems},
  year={2010}
}

@article{grasedyck2010hierarchical_SIMAA,
  title={Hierarchical singular value decomposition of tensors},
  author={{Grasedyck}, Lars},
  journal={SIAM journal on matrix analysis and applications},
  volume={31},
  number={4},
  pages={2029--2054},
  year={2010},
  publisher={SIAM}
}

@article{shi2026distributed,
  title={Distributed memory parallel adaptive tensor-train cross approximation},
  author={Shi, Tianyi and Hayes, Daniel and Qiu, Jing-Mei},
  journal={SIAM Journal on Scientific Computing},
  pages={S486-S512},
  year={2026},
  publisher={SIAM}
}

@article{civril2007finding,
  title={Finding maximum volume sub-matrices of a matrix},
  author={Civril, Ali and Magdon-Ismail, Malik},
  journal={RPI Comp Sci Dept TR},
  pages={07--08},
  year={2007},
  publisher={Citeseer}
}

@article{cai2021high,
  title={High order semi-{Lagrangian} discontinuous {Galerkin} method coupled with {Runge-Kutta} exponential integrators for nonlinear {Vlasov} dynamics},
  author={Cai, Xiaofeng and Boscarino, Sebastiano and Qiu, Jing-Mei},
  journal={Journal of Computational Physics},
  volume={427},
  pages={110036},
  year={2021},
  publisher={Elsevier}
}

@book{boyd2001chebyshev,
  title={Chebyshev and {Fourier} spectral methods},
  author={Boyd, John P},
  year={2001},
  publisher={Courier Corporation}
}

@article{sands2026adaptive,
    title={An adaptive-rank approach with greedy sampling for multi-scale {BGK} Equations}, 
    author={William A. Sands and Jing-Mei Qiu and Daniel Hayes and Nanyi Zheng},
    journal={Journal of Computational Physics},
    volume={547},
    pages={114523},
    year={2026},
    publisher={Elsevier}
}

@incollection{cho2024conservative,
  title={Conservative semi-{L}agrangian Methods for Kinetic Equations},
  author={Cho, Seung-Yeon and Groppi, Maria and Qiu, Jing-Mei and Russo, Giovanni and Yun, Seok-Bae},
  booktitle={Active Particles, Volume 4: Theory, Models, Applications},
  pages={283--420},
  year={2024},
  publisher={Springer}
}

@article{li2023high,
  title={A high-order {semi-Lagrangian} finite difference method for nonlinear {Vlasov} and {BGK} models},
  author={Li, Linjin and Qiu, Jingmei and Russo, Giovanni},
  journal={Communications on Applied Mathematics and Computation},
  volume={5},
  number={1},
  pages={170--198},
  year={2023},
  publisher={Springer}
}

@article{einkemmer2025review,
  title={A review of low-rank methods for time-dependent kinetic simulations},
  author={Einkemmer, Lukas and Kormann, Katharina and Kusch, Jonas and McClarren, Ryan G and Qiu, Jing-Mei},
  journal={Journal of Computational Physics},
  volume={538},
  pages={114191},
  year={2025},
  publisher={Elsevier}
}

@article{ding2020semi,
  title={A {semi-Lagrangian} discontinuous {Galerkin (DG)}--local {DG} method for solving convection-diffusion equations},
  author={Ding, Mingchang and Cai, Xiaofeng and Guo, Wei and Qiu, Jing-Mei},
  journal={Journal of Computational Physics},
  volume={409},
  pages={109295},
  year={2020},
  publisher={Elsevier}
}

@article{qiu2011conservative,
  title={Conservative high order {semi-Lagrangian} finite difference {WENO} methods for advection in incompressible flow},
  author={Qiu, Jing-Mei and Shu, Chi-Wang},
  journal={Journal of Computational Physics},
  volume={230},
  number={4},
  pages={863--889},
  year={2011},
  publisher={Elsevier}
}

@article{qiu2010conservative,
  title={A conservative high order {semi-Lagrangian} {WENO} method for the {Vlasov} equation},
  author={Qiu, Jing-Mei and Christlieb, Andrew},
  journal={Journal of Computational Physics},
  volume={229},
  number={4},
  pages={1130--1149},
  year={2010},
  publisher={Elsevier}
}

@article{kormann2015semi,
  title={A {semi-Lagrangian} {Vlasov} solver in tensor train format},
  author={Kormann, Katharina},
  journal={SIAM Journal on Scientific Computing},
  volume={37},
  number={4},
  pages={B613--B632},
  year={2015},
  publisher={SIAM}
}

@article{guo2022low,
  title={A low rank tensor representation of linear transport and nonlinear {Vlasov} solutions and their associated flow maps},
  author={Guo, Wei and Qiu, Jing-Mei},
  journal={Journal of Computational Physics},
  volume={458},
  pages={111089},
  year={2022},
  publisher={Elsevier}
}

@article{guo2024local,
  title={A Local Macroscopic Conservative {(LoMaC)} low rank tensor method for the {Vlasov} dynamics},
  author={Guo, Wei and Qiu, Jing-Mei},
  journal={Journal of Scientific Computing},
  volume={101},
  number={3},
  pages={61},
  year={2024},
  publisher={Springer}
}

@article{sands2025transport,
  title     = {High-order adaptive rank integrators for multi-scale linear kinetic transport equations in the hierarchical {Tucker} format},
  author    = {Sands, William A. and Guo, Wei and Qiu, Jing-Mei and Xiong, Tao},
  journal   = {SIAM Journal on Scientific Computing},
  volume    = {47},
  number    = {6},
  pages     = {A3383--A3412},
  year      = {2025},
  publisher = {SIAM}
}

@article{einkemmer2018low,
  title={A low-rank projector-splitting integrator for the {Vlasov--Poisson} equation},
  author={Einkemmer, Lukas and Lubich, Christian},
  journal={SIAM Journal on Scientific Computing},
  volume={40},
  number={5},
  pages={B1330--B1360},
  year={2018},
  publisher={SIAM}
}

@article{einkemmer2021mass,
  title={A mass, momentum, and energy conservative dynamical low-rank scheme for the {Vlasov} equation},
  author={Einkemmer, Lukas and Joseph, Ilon},
  journal={Journal of Computational Physics},
  volume={443},
  pages={110495},
  year={2021},
  publisher={Elsevier}
}

@article{dektor2025interpolatory,
  title={Interpolatory dynamical low-rank approximation for the 3+ 3d {Boltzmann-BGK} equation},
  author={Dektor, Alec and Einkemmer, Lukas},
  journal={Journal of Computational Physics},
  volume={547},
  pages={114515},
  year={2025},
  publisher={Elsevier}
}

@article{hackbusch2009new,
  title={A new scheme for the tensor representation},
  author={Hackbusch, Wolfgang and K{\"u}hn, Stefan},
  journal={Journal of Fourier analysis and applications},
  volume={15},
  number={5},
  pages={706--722},
  year={2009},
  publisher={Springer}
}

@article{espig2009black,
  title={Black box low tensor-rank approximation using fiber-crosses},
  author={Espig, Mike and Grasedyck, Lars and Hackbusch, Wolfgang},
  journal={Constructive approximation},
  volume={30},
  number={3},
  pages={557--597},
  year={2009},
  publisher={Springer}
}

@article{sonnendrucker1999semi,
  title={The {semi-Lagrangian} method for the numerical resolution of the {Vlasov} equation},
  author={Sonnendr{\"u}cker, Eric and Roche, Jean and Bertrand, Pierre and Ghizzo, Alain},
  journal={Journal of computational physics},
  volume={149},
  number={2},
  pages={201--220},
  year={1999},
  publisher={Elsevier}
}

@article{rossmanith2011positivity,
  title={A positivity-preserving high-order {semi-Lagrangian discontinuous Galerkin} scheme for the {Vlasov--Poisson} equations},
  author={Rossmanith, James A and Seal, David C},
  journal={Journal of Computational Physics},
  volume={230},
  number={16},
  pages={6203--6232},
  year={2011},
  publisher={Elsevier}
}

@article{DimarcoReview2014,
    author = {Giacomo Dimarco and Lorenzo Pareschi},
    title = {Numerical methods for kinetic equations},
    journal = {Acta Numerica},
    volume = {23},
    pages = {369-520},
    publisher = {Cambridge University Press},
    year = {2014}
}

@article{guo2026hierarchical,
  title={Hierarchical Tensor Network Structure Search for High-Dimensional Data},
  author={Guo, Zheng and Deshpande, Aditya and Wang, Xinyu and Kiedrowski, Brian C and Gorodetsky, Alex A},
  journal={arXiv preprint arXiv:2603.27856},
  year={2026}
}

\end{document}